\documentclass[url=false]{elsarticle}

\usepackage{amsthm, amssymb, amsmath, latexsym, amsfonts, mathcomp, amscd}
\usepackage{adjustbox}

\usepackage{hyperref}
\journal{Applied and Computational Harmonic Analysis}

\usepackage{comment}
\usepackage{xcolor}
\usepackage{tikz-cd}

\allowdisplaybreaks

\def\Psiec{\Psi \mathrm{ec}(\R^n)}
\def\fdeg{r}

\theoremstyle{plain}

\newtheorem{proposition}{Proposition}
\newtheorem{theorem}{Theorem}
\newtheorem{definition}{Definition}

\theoremstyle{remark}
\newtheorem{remark}{Remark}
\newtheorem{example}{Example}

\def\d{\mathrm{d}}
\def\dd{\mathrm{d}}

\def\Z{\mathbb{Z}}
\def\R{\mathbb{R}}

\def\M{\mathcal{M}}

\def\ip{\mathrm{i}}

\DeclareMathOperator*{\sign}{sgn}

\newcommand{\dpp}[2]{\ensuremath{\frac{\partial #1}{\partial #2}}}

\usepackage{mathtools}     %in order to write Meijer G-functions
\usepackage{suffix}         %in order to write Meijer G-functions
\usepackage{relsize}

\DeclarePairedDelimiterX\MeijerM[3]{\lparen}{\rparen}%                    
{\begin{smallmatrix}#1 \\ #2\end{smallmatrix}\delimsize\vert\,#3}     

\newcommand\MeijerG[8][]{%                           
  G^{\,#2,#3}_{#4,#5}\MeijerM[#1]{#6}{#7}{#8}}       

\WithSuffix\newcommand\MeijerG*[7]{%                 
  G^{\,#1,#2}_{#3,#4}\MeijerM*{#5}{#6}{#7}}

\providecommand{\commentCL}[1]{}

\begin{document}

\begin{frontmatter}

\title{$\Psi$ec: A Local Spectral Exterior Calculus}

\author{Christian Lessig\fnref{cl}}
 \address{Otto-von-Guericke-Universit{\"a}t Magdeburg}
\fntext[cl]{christian.lessig@ovgu.de}



\begin{abstract}
  We introduce $\Psiec$, a discretization of Cartan's exterior calculus of differential forms using wavelets. 
  Our construction consists of differential $\fdeg$-form wavelets with flexible directional localization that provide tight frames for the spaces $\Omega^\fdeg(\R^n)$ of forms in $\R^2$ and $\R^3$. 
%  The wavelets also provide a flexible directional localization, between fully isotropic and curvelet- and ridgelet-like, 
  By construction, the wavelets satisfy the de Rahm co-chain complex, the Hodge decomposition, and that the $k$-dimensional integral of an $\fdeg$-form is an $(\fdeg-k)$-form.
  They also verify Stokes' theorem for differential forms, with the most efficient finite dimensional approximation attained using directionally localized, curvelet- or ridgelet-like forms.
%  which is approximated most efficiently using directionally localized, curvelet- or ridgelet-like forms.
  The construction of $\Psiec$ builds on the geometric simplicity of the exterior calculus in the Fourier domain. 
  We establish this structure by extending existing results on the Fourier transform of differential forms to a frequency description of the exterior calculus, including, for example, a Plancherel theorem for forms and a description of the symbols of all important operators. 
\end{abstract}


\end{frontmatter}

%\linenumbers

%%%%%%%%%%%%%%%%%%%%%%%%%%%%%%%%%%%%%%%%%%%%%%%%%%%%%%%%%%%%%%%%%%%%%%%%%%%%%%%%%%%%%%%%%%%%%%%%%%%%%%%%%%%%%%%%%%%%%%%%%%%%%%

\section{Introduction}
\label{sec:introduction}

The exterior calculus, first introduced by Cartan~\cite{Cartan1945}, provides a formulation of scalar and vector-valued functions that encodes their differential relationships in a coordinate-invariant language.
The central objects of the calculus are differential forms, which provide a covariant formalization of vector fields, and the exterior derivative, which is a first order differential operator  generalizing gradient, curl, and divergence.
%The importance of respecting the exterior calculus' inherent geometric structure for the numerical solution of partial differential equations, even in $\R^n$ and with a flat metric, has been shown for various applications in the last 30 years.

The importance of the exterior calculus for the numerical solution of partial differential equations, even in $\R^n$ and with a flat metric, has been shown in various applications.
One of the first ones was electromagnetic theory where already in the 1960s it was realized~\cite{Yee1966} that the electric and magnetic fields have to be discretized differently to obtain satisfactory numerics.
Later, it was found that this could be understood by the electric and magnetic fields being differential forms of different degrees~\cite{Bossavit1997,Hiptmair2002}.
%Another application is geophysical fluid dynamics~\cite{Platzman1960,Arakawa1977,Arakawa1981} where the conservation properties that are crucial for long term weather and climate simulations arise from discretizations that respect the intrinsic structure of the exterior calculus. 
Motivated by these and analogous results in other fields, e.g. in geophysical fluid dynamics~\cite{Platzman1960,Arakawa1977,Arakawa1981}, various discrete realizations of Cartan's exterior calculus have been proposed, e.g.~\cite{Bochev2006,Desbrun2006,Arnold2006,Ringler2010,Rufat2014,DeGoes2016}. 
These have been applied to applications such as elasticity~\cite{Arnold2006}, fluid mechanics~\cite{Elcott2007,Pavlov2011}, climate simulation~\cite{Arakawa1981,Ringler2010,Dubos2013}, magnetohydrodynamics~\cite{Gawlik2011}, electromagnetic theory~\cite{Bossavit1997}, and geometry processing~\cite{DeGoes2013}.

% We propose an alternative approach to obtain a numerically practical form of Cartan's exterior calculus that uses wavelet differential forms to span the spaces of $\fdeg$-forms.
In the present work, we propose $\Psiec$, an alternative discretization of Cartan's exterior calculus that is defined over a family of differential form wavelets with flexible directional localization.
% amenable to numerical calculations.
Towards its construction, we extend existing results from the theoretical physics literature and develop a description of the exterior calculus in the Fourier domain.
This shows that under the Fourier transform it becomes a chain complex whose structure is encoded in a simple, geometric way in spherical coordinates in frequency space.
The exterior derivative, for instance, acts  only on the radial component of a differential form in  frequency space. 
% and exact forms are thus those without one.
% and a form is hence exact, i.e. its exterior derivative vanishes, when tangent to the frequency sphere $S_{\xi}^{n-1}$. 
We therefore discretize the exterior calculus in the Fourier domain using spherical coordinates  and attain space-frequency localization by using polar wavelet windows~\cite{Unser2013,Lessig2018a} that respect the spherical coordinate structure. 
%In particular, we choose aligned differential form basis functions, which are essentially either tangential or radial to the frequency sphere, and localize in space and frequency using scalar wavelet windows that are separable in spherical coordinates.
The construction yields differential form wavelets well localized in space and frequency that obey by construction the de Rahm complex and the distinction between exact and co-exact forms, i.e. the Hodge-Helmholtz decomposition.
%, and it is in our formulation a key to preserve many other structures.
At the same time, the definition in spherical coordinates also allows for a flexible directional localization, between fully isotropic and curvelet- and ridgelet-like.
In contrast to most existing discretizations in the literature, the ``discrete'' differential form wavelets in our work are bona fide forms in the sense of the continuous theory. 
Hence, all operations from the continuous setting are well defined.
A central question therefore becomes efficient numerical evaluation and we will at least partially address it in the present work.

%We christen our wavelet-based discretization of Cartan's exterior calculus $\Psi\mathrm{ec}(\R^n)$. 
The properties of $\Psiec$ and its differential form wavelets $\psi_{s,a}^{\fdeg,\nu,n}$, where $n = 2,3$ is the dimension of the space, $0 \leq \fdeg \leq n$ is the degree of the form, $\nu \in \{ \dd , \delta \}$ denotes if the wavelet is exact or co-exact, and $s = (j_s, k_s, t_s)$ describes level $j_s$, translation $k_s$ and orientation $t_s$, can be summarized as follows (see Sec.~\ref{sec:psiforms} for the precise technical statements):
\begin{itemize}
  \item The $\psi_{s,a}^{\fdeg,\nu,n}$ form tight frames for the homogeneous Sobolev spaces $\dot{L}_2^{\vert \nu \vert}(\Omega_{\nu}^{\fdeg},\R^n)$.
  \item The differential form wavelets satisfy the exterior co-chain complex, i.e. $\dd \psi_{s,a}^{\fdeg,\delta,n} = \psi_{s,a}^{\fdeg+1,\dd,n}$ and $\dd \psi_{s,a}^{\fdeg,\dd,n} = 0$.
      Furthermore, the exterior derivative $\dd$ preserves the localization described by $s$.
  \item Stokes' theorem for differential forms $\alpha \in \Omega^{\fdeg}(\R^n)$,
    \begin{align*}
      \int_{\partial \M} \alpha = \int_{\M} \dd \alpha ,
    \end{align*}
    for $\M \subset \R^n$ a manifold, 
    holds for differential form wavelets as 
      \begin{align}
        \label{eq:intro:stokes_forms}
          \sum_{s} \alpha_s \int_{\R^n} \psi_{s}^{\fdeg-1,\delta} \wedge \chi_{_{\partial \M}}
  &= \sum_{s} \alpha_s \int_{\R^n} \psi_{s}^{\fdeg,\dd} \wedge \chi_{_{\M}} 
      \end{align}
      where the $\alpha_s$ are the frame coefficients of $\alpha \in \Omega^{\fdeg}(\R^n)$ in the differential form wavelet frame, and $\chi_{_{\M}}$ and $\chi_{_{\partial \M}}$ are the characteristic functions of $\M$ and its boundary $\partial \M$, respectively.
      When $\dim(\M) = n$, Stokes theorem therefore contains the approximation problem for cartoon-like functions (right hand side of Eq.~\ref{eq:intro:stokes_forms}).
      This has been studied extensively in the literature on curvelets, shearlets and related constructions, e.g.~\cite{Candes2005a,Candes2005b,Labate2005,Do2003}, and these provide for it quasi-optimal approximation rates. 
      Moreover, $\chi_{_{\partial \M}}$ is exactly the wavefront set associated with $\chi_{_{\M}}$.
  \item The Hodge dual $\star \, \psi_{s,a}^{k,\nu,n}$ has a simple and practical description.
  \item The fiber integral of $\psi_{s,a}^{\fdeg,\nu,n}$ along an arbitrary $p$-dimensional linear sub-manifold is again a differential form wavelet $\psi_{s',a}^{\fdeg-p,\nu,n}$ and with $s'$ preserving the localization of $s$.
  \item The Laplace-de Rahm operator $\Delta : \Omega^k \to \Omega^k$ has a closed form representation and Galerkin projection for the differential form wavelets $\psi_{s,a}^{\fdeg,\nu,n}$.
  \item The $\psi_{s,a}^{\fdeg,\nu,n}$ have closed form expressions in the spatial domain. This clarifies, e.g. how the angular localization in frequency space is carried over to the spatial domain.
\end{itemize}

Multiplicative operators in the exterior calculus, such as the wedge product or the Lie derivative, have currently no natural expression in our calculus. 
In numerical calculations, these can be determined using the transform method~\cite{Eliasen1970,Orszag1970}, i.e. with the multiplication being evaluated pointwise in space, which is facilitated by the closed form expressions that are available there.
A thorough investigation of this has, however, be left for future work.

%Simple example to understand the importance: eigenvalues of Laplacian for non-uniform mesh (still regular but with two different mesh sizes). One needs a non-constant Hodge dual to compensate for this. Although one is arguably still in $\R^2$, a useful way to look at the issue is to consider the space as having a non-constant metric.

%Explicitly, we use the Fourier description of exterior calculus and restrict it into sub-regions using judiciously chosen window functions to obtain, firstly, a tiling of the frequency plane that enables sparse approximations and, secondly, that inherently respects the essential structure of exterior calculus.

%We illustrate our construction with examples from the physical sciences.
%Applications:
%\begin{itemize}
%  \item Eigenmodes of electro-magnetic fields, cf.~\cite{Hiptmair2002}
%  \item Hodge-Helmholtz decomposition~\cite{Tong2003,Bhatia2013}
%  \item Volume-preserving shape modeling: but where would one use de Rahm complex?  
%\end{itemize}

The remainder of the work is structured as follows. 
After discussing related work in the next section, we present in Sec.~\ref{sec:polar_wavelets} some background on the polar wavelets that we use as localization windows in the frequency domain.
Afterwards we develop in Sec.~\ref{sec:exterior_calculus} the Fourier transform of differential forms.
In Sec.~\ref{sec:psiforms} we introduce our differential form wavelets and study how the exterior calculus, including the exterior derivative, the wedge product, Hodge dual, and Stokes theorem, can be expressed using these.
%Numerical results and a simple example application of our local spectral exterior calculus are presented in Sec.~\ref{sec:numerics}.
In Sec.~\ref{sec:conclusion} we summarize our work and discuss directions for future work.
Our notation and conventions are summarized in~\ref{sec:preliminaries}.

\section{Related Work}
\label{sec:related}

Our work builds on ideas from what are usually two distinct fields: firstly, wavelet theory and in particular wavelets constructed in polar and spherical coordinates, and, secondly, exterior calculus and its discretizations.
We will discuss relevant literature in these fields in order.

%%%%%%%%%%%%%%%%%%%%%%%%%%%%%%%%%
\paragraph{Polar Wavelets}

The idea to combine space-frequency localization with a description in polar coordinates in the Fourier domain goes back at least to work by Fefferman in the theory of partial differential equations, in particular on the second dyadic decomposition, cf.~\cite[Ch. IX]{Stein1993}.
It was re-introduced multiple times, e.g. in work on steerable wavelets~\cite{Freeman1991,Perona1991} and curvelets~\cite{Candes1999b,Candes2005a,Candes2005b} where the motivation was the analysis of directional features in images.
A framework that encompasses the latter two approaches was proposed by Unser and co-workers~\cite{Unser2010,Unser2011,Unser2013,Ward2014} and this defines the polar wavelets that are the scalar building block for the construction of $\Psiec$.
We use the term `polar wavelet', or short `polarlet', since it encapsulates what distinguishes the construction from other multi-dimensional wavelets and what is the key to their utility in our work.

%%%%%%%%%%%%%%%%%%%%%%%%%%%%%%%%%
\paragraph{Curl and Divergence Free Wavelets}

Various constructions for curl and divergence free wavelets have been proposed over the years, e.g.~\cite{LemarieRieusset1992,Battle1993,Suter1995,Urban1995,Deriaz2009a,Bostan2015}, and these are related to our differential form wavelets by the musical isomorphism.
Based on seminal work by Lemair{\'e}-Rieusset~\cite{LemarieRieusset1992}, various authors, e.g.~\cite{Jouini1993,Stevenson2011a,Stevenson2011b,Stevenson2016a,Kadriharouna2019} constructed coupled scalar wavelet bases for $1$-dimensional domains that satisfy a differential relationship so that the derivative of one yields (up to a constant) the other or a corresponding integral relationship holds.
These constructions can be extended to curl or divergence free wavelets using the coordinate representations of the differential operators of vector calculus. 
Our differential form wavelets satisfy a similar coupling under differentiation. 
However, our construction is inherently multi-dimensional and we exploit the simple structure of the differential operators in the frequency domain.
The latter has also been used in the work by Bostan, Unser and Ward~\cite{Bostan2015} although divergence freedom is there only enforced numerically.
A precursor to the presented work is those in~\cite{Lessig2018z}.
There, polar wavelets are used as scalar window functions for the construction of polar divergence free wavelets and it is exploited that a vector field is divergence free when its Fourier transform is tangential to the frequency sphere.

%%%%%%%%%%%%%%%%%%%%%%%%%%%%%%%%%
\paragraph{Discretizations of Exterior Calculus}

Discretizations of Cartan's exterior calculus try to preserve important structures of the continuous theory, e.g. the de Rahm co-chain complex or Stokes' theorem, to improve the accuracy and robustness of numerical calculations. 
Different precursors those good numerical performance can, at least in hindsight, be explained through their adherence to the exterior calculus, exist.
For example, spectral discretizations in geophysical fluid dynamics, first introduced in the 1950s and 1960s, yield energy and enstrophy conserving time integrators~\cite{Platzman1960} by implicitly using what one could call a spectral exterior calculus~\cite{Silva2020}.
In the 1970s, Arakawa and Lamb~\cite{Arakawa1977,Arakawa1981} introduced the now eponymous grid-based discretizations for geophysical fluid dynamics to extend energy and enstrophy conservation to finite difference methods, again without an explicit connection to Cartan's calculus but implicitly respecting its structure.
A few years later, N{\'e}d{\'e}lec~\cite{Nedelec1980,Nedelec1986}, introduced consistent discretizations for the spaces $H(\mathrm{curl})$ and $H(\mathrm{div})$ that respect important properties of the classical differential operators of vector calculus, and hence also of the exterior derivative.
The first work that, to our knowledge, made an explicit connection between exterior calculus and numerics was those by Bossavit~\cite{Bossavit1997} in computational electromagnetism who realized that the use of differential forms and the calculus defined on them enables one to systematically understand and extend earlier results that provided good numerical performance for problems in computational electromagnetism~\cite{Yee1966}.
%See~\cite{Arnold2006} for some further remarks on the development of the theory.

Most existing discretizations of exterior calculus are finite element-based, such as finite element exterior calculus developed by Arnold and collaborators~\cite{Arnold2006,Arnold2018}, discrete exterior calculus by Hirani, Desbrun and co-workers~\cite{Hirani2003,Desbrun2006}, mimetic discretizations by Bochev and Hyman~\cite{Bochev2006}, and the TRiSK scheme by Thuburn, Ringler, Klemp, and Skamarock~\cite{Thuburn2009,Ringler2010}.
These works all have in common that they construct a discrete analogue of Cartan's exterior calculus that satisfies important structural properties of the continuous theory, e.g. in that it is a co-chain complex for a discrete exterior derivative or that a discrete Stokes' theorem holds. 
Our construction, in contrast, is best seen as a multi-resolution projection of the continuous theory that preserve its intrinsic structure.
In particular, our differential form wavelets remain genuine forms in the sense of the continuous theory. 
All operators from there are thus naturally defined.
Compared to finite element type discretizations where one seeks discrete definitions mimicking the original theory, the questions shifts in our work thus to how these can be computed efficiently numerically or how this can be ensured in the construction of the wavelets.
% and not ``integrated forms'' as have to be used in finite element-based discretizations.
%Subsequently, various forms have been developed~\cite{Desbrun2006,Arnold2006,Rufat2014,DeGoes2016} and the importance of these ideas has been shown for various other areas, e.g. fluid mechanics~\cite{Elcott2007,Pavlov2011}

Wavelets and exterior calculus have so far only been combined in two works and again within a finite element framework.
Dubos, Kevlahan and collaborators~\cite{Dubos2013,Aechtner2015,Kevlahan2019} extended the TRiSK scheme to a multi-resolution formulation using second generation wavelets and demonstrated that this provides an efficient discretization for geophysical fluid dynamics. 
In work contemporaneous to ours, Budninskiy, Owhadi, and Desbrun~\cite{Budninskiy2019} constructed operator-adapted wavelets for Discrete Exterior Calculus using the efficient algorithms proposed in~\cite{Owhadi2017}.
In contrast to $\Psiec$ whose construction relies on the structure of exterior calculus in the Fourier domain, the above works are defined purely in the spatial domain.

The extension of finite element-based discretizations of exterior calculus to spectral convergence rates has been considered by Rufat, Mason and Desbrun~\cite{Rufat2014}.
These authors accomplished it by introducing interpolation and histapolation maps to relate discrete exterior calculus and continuous forms. 
Gross and Atzberger presented a similar approach using hyper-interpolation for radial manifolds, i.e. manifolds that can be described as a height field over $S^2$.
A discretizations of exterior calculus for subdivision surfaces was introduced by de Goes et al.~\cite{DeGoes2016}.
It shares the multi-resolution structure with our approach and, at least in the limit case, also yields continuous differential forms.

Another recent discretization of exterior calculus is those by Berry and Giannakis~\cite{Berry2018} for manifold learning problems.
They employ the eigenfunctions of the Laplace-de Rahm operator as bases for their forms.
This is analogous to the spectral bases, formed by spherical harmonics, first used in geophysical fluid dynamics and that, as mentioned earlier, essentially form a spectral exterior calculus.
We also build on a spectral formulation but localize it in space and frequency and, furthermore, exploit the additional structure afforded in the Euclidean setting. 

%For simple domains, such as the cube or the sphere, for which the eigenfunctions have closed form expressions, the approach yields basis functions that are continuous forms, as our differential form wavelets.
%A method to localize these results~\cite{deWitt2012,Liu2015} while preserving the advantages of the analytic, closed form expressions was one of the original motivations for the present work.
%Compared to the use the eigenfunctions of the Laplace-de Rahm operator, however, we can no longer work intrinsically on manifolds.

\section{Polar Wavelets}
\label{sec:polar_wavelets}

Polar wavelets are a family of wavelets defined in polar or spherical coordinates in the Fourier domain.
Their construction uses a compactly supported radial window, $\hat{h}( \vert {\xi} \vert)$, which controls the overall frequency localization, and an angular one, $\hat{\gamma}(\bar{\xi})$, which controls the directionality, where $\bar{\xi} = \xi / \vert \xi \vert$.
The mother wavelet is thus given by $\hat{\psi}({\xi}) = \hat{\psi}(\bar{\xi}) \, \hat{h}( \vert {\xi} \vert )$ with the whole family of functions being generated by dilation, translation and rotation.

In two dimensions, the angular window can be described with a Fourier series in the polar angle $\hat{\theta}$.
A polar wavelet is there hence given by
  \label{eq:polarlets:2d}
\begin{align}
  \label{eq:polarlets:2d:hat}
  \hat{\psi}_s({\xi})
  \equiv \hat{\psi}_{jkt}({\xi})
  = \frac{2^j}{2\pi} \Big( \sum_{n = -N_j}^{N_j} \beta_{n}^{j,t} \, e^{i n \hat{\theta}} \Big) \, \hat{h}(2^{-j} \vert {\xi} \vert ) \, e^{-i \langle {\xi} , 2^j {k}\rangle}
\end{align}
where the multi-index $s = (j,k,t)$ describes level $j$, translation $k$, and orientation $t$ and we will write $(j_s,k_s,t_s)$ whenever confusion might arise.
The coefficients $\beta_{n}^{j,t}$ control the angular localization and in the simplest case $\beta_{n}^{j,t} = \delta_{n 0}$ for all $j$, $t$ and one has isotropic, bump-like wavelet functions.
Conversely, when the support of the $\beta_{n}^{j,t}$ are the entire integers in $n$, then one can describe compactly supported angular windows.
The above formulation then also encompasses second generation curvelets~\cite{Candes2005b} and provides more generally a framework to realize $\alpha$-molecule-like constructions~\cite{Grohs2013,Grohs2014,Grohs2016a}.

%To simplify notation we will assume that the angular and radial component use one scaling parameter.
A useful property of polar wavelets is that their spatial representation, given by inverse Fourier transform $\mathfrak{F}^{-1}$ of Eq.~\ref{eq:polarlets:2d:hat}, can be computed in closed form.
Using the Fourier transform in polar coordinates, cf.~\ref{sec:preliminaries:fourier:polar}, one obtains~\cite{Lessig2018a}
\begin{align}
  \label{eq:polarlets:2d:space}
  \psi_{s}({x})
  \equiv \psi_{jkt}({x})
  = \frac{2^j}{2\pi} \sum_{n=-N_j}^{N_j} i^n \, \beta_{n}^{j,t} \, e^{i n \theta_{{x}}} \, h_n( \vert 2^{j} {x} - {k}\vert )
\end{align}
where $h_n(\vert {x} \vert)$ is the Hankel transform of $\hat{h}(\vert {\xi} \vert)$ of order $n$.
When $\hat{h}(\vert {\xi} \vert)$ is the window proposed for the steerable pyramid~\cite{Portilla2000}, then $h_n(\vert {x} \vert)$ also has a closed form expression~\cite{Lessig2018a}.
Note also that the angular localization described by the $\beta_{n}^{j,t}$ is essentially invariant under the Fourier transform and only modified by the factor of $i^n$ that implements a rotation by $\pi / 2$.

For applications, the wavelets in Eq.~\ref{eq:polarlets:2d} are augmented using scaling functions $\phi_{j,k}$.
They are isotropic and defined using a radial window $\hat{g}(\vert \xi \vert)$, i.e. $\phi_{j,0} = \mathfrak{F}^{-1}(\hat{g})$, that is chosen such that the $\phi_{j,k}$ represent a signal's low frequency part.
To simplify notation, we follow the usual convention and write $\psi_{-1,k} \equiv \phi_{0,k}$.
With the scaling functions, the following result holds~\cite{Unser2013,Lessig2018z}.

\begin{proposition}
  \label{prop:polarlet:2d:parseval}
  % N_j is the frequency bandlimit for the \beta
  Let $U_j \in \R^{T_j \times 2 N_j + 1}$ be the matrix formed by the angular localization coefficients $\beta_{n}^{j,t} = \beta_{n}^{j} \, e^{-i n t (2\pi / T_{j})}$ for the $T_j$ different orientations, and let $D_j \in \R^{2 N_j + 1 \times 2 N_j + 1}$ be a diagonal matrix.
When the Calder{\'o}n admissibility condition $\big\vert \hat{g}( \vert \xi \vert ) \big\vert^2 + \sum_{j = 0}^{\infty} \big\vert \hat{h}( 2^{-j} \vert \xi \vert ) \big\vert^2 = 1$, $\forall \xi \in \widehat{\mathbb{R}}^2$, is satisfied and $U_j^H U_j = D_j$ with $\mathrm{tr}( D_j ) = 1$ for all levels $j$,
then any $f \in L_2(\mathbb{R}^{2})$ has the representation
\begin{align*}
  f(x) = \sum_{j = -1}^{\infty} \sum_{k \in \mathbb{Z}^2} \sum_{t=1}^{T_j} \big\langle f(y) , \psi_{j,k,t}(y) \big\rangle \, \psi_{j,k,t}(x)
\end{align*}
with frame functions
\begin{align*}
  \psi_{j,k,t}(x) = \frac{2^j}{2\pi} \, \psi\big( R_{2\pi t / T_j} (2^j x - k) \big)
\end{align*}
where $\psi(x)$ is given by Eq.~\ref{eq:polarlets:2d:space} and $R_{2\pi t / T_j}$ is the rotation by $2\pi t / T_j$.
\end{proposition}

Although the above frame is redundant, since it is Parseval tight it still affords most of the conveniences of an orthonormal basis.
In the following, we will often write $\mathcal{S}$ for the index set of all scales, translations, and rotations.

In three dimensions, polar wavelets are defined analogous to Eq.~\ref{eq:polarlets:2d:hat} by
\begin{align}
  \label{eq:polarlets:3d:hat}
  \hat{\psi}_{s}({\xi}) \equiv \hat{\psi}_{j,k,t}({\xi})
  = \frac{2^{3j/2}}{(2\pi)^{3/2}} \sum_{l=0}^{L_j} \sum_{m=-l}^l \kappa_{lm}^{j,t} \, y_{lm}(\bar{{\xi}}) \, \hat{h}(2^{-j} \vert {\xi} \vert) \, e^{-i \langle {\xi} , 2^j {k}\rangle}
\end{align}
where the directional localization window $\hat{\gamma}(\bar{\xi})$ is now expressed using the spherical harmonics $y_{lm}$ and the $\kappa_{lm}^{j,t}$ are the expansion coefficients; see~\ref{sec:preliminaries:sh} for a brief summary of harmonic analysis of the sphere.
The wavelets have again a closed form expression in the spatial domain,
\begin{align}
  \label{eq:polarlets:3d:space}
  \psi_{s}(x) \equiv \psi_{j,k,t}(x)
  = \frac{2^{3j/2}}{(2\pi)^{3/2}} \sum_{l=0}^{L_j} \sum_{m=-l}^l i^l \kappa_{lm}^{j,t} \, y_{lm}(\bar{x}) \, h_l(\vert 2^{j} {x} - k \vert) .
\end{align}
which can be obtained using the Rayleigh formula, cf.~\ref{sec:preliminaries:fourier:polar}.
Furthermore, the wavelets can again be augmented using scaling functions to represent a signal's low frequency part.
Then the following result holds~\cite{Ward2014,Lessig2018z} (see again~\ref{sec:preliminaries:sh} for notation).

\begin{proposition}
  \label{prop:tight_frame:3d}
  Let $w_{j,t} \in \R^{(L_j+1)^2}$ be the vector formed by the rotated angular localization coefficients $\kappa_{lm}^{j,t} = \sum_{m' = -l}^l W(\lambda_{j,t})_{l,m}^{m'} \, \kappa_{l,m'}$ for a localization window centered at $\lambda_{j,t}$, where $W(\lambda_{j,t})_{l,m}^m$ is the Wigner-D matrix implementing rotation in the spherical harmonics domain, and let $G^{lm} \in \R^{(L_j+1)^2 \times (L_j+1)^2}$ be the matrix formed by the spherical harmonics product coefficients for fixed $(l,m)$.
  When the Calder{\'o}n condition $\big\vert \hat{g}( \vert \xi \vert ) \big\vert^2 + \sum_{j=0}^{\infty} \big\vert \hat{h}( 2^{-j} \vert \xi \vert ) \big\vert^2 = 1$, $\forall \xi \in \widehat{\mathbb{R}}^3$ is satisfied and $\delta_{l,0} \delta_{m,0} = \sum_{t=0}^{T_j} w_{j,t} \, G^{lm} \, w_{j,t}$ for all $j$, $t$ then any $f(x) \in L_2(\mathbb{R}^{3})$ has the representation
  \begin{align*}
     f(x) = \sum_{j=-1}^{\infty} \sum_{k \in \mathbb{Z}^3} \sum_{t=1}^{T_j} \big\langle f(y) \, , \, {\psi}_{j,k,t}(y) \big\rangle \, {\psi}_{j,k,t}(x)
  \end{align*}
  with frame functions
  \begin{align*}
    {\psi}_{j,k,t}(x) = \frac{2^{3j/2}}{(2\pi)^{3/2}} \, {\psi} \big( R_{\lambda_{j,t}} (2^{j} x -  k ) \big) ,
  \end{align*}
  for $\psi(x)$ defined in Eq.~\ref{eq:polarlets:3d:space} and $R_{\lambda_{j,t}}$ the rotation from the North pole to $\lambda_{j,t}$.
\end{proposition}

See~\ref{sec:polarlets:admissibility:3d} for a proof of the admissibility conditions.
We refer to the original works~\cite{Unser2013,Ward2014} and~\cite{Lessig2018a} for a more detailed discussion of polar wavelets.

\section{Differential Forms and their Fourier Transform}
\label{sec:exterior_calculus}

In this section, we first recall basic facts about differential forms and the exterior calculus defined on them.
Our principle reference for this material will be the book by Marsden, Ratiu, and Abraham~\cite{Marsden2004} and we will also use the notation and conventions from there.
Afterwards, starting from existing results in theoretical physics~\cite{Voronov1988,Kalkman1993,Castellani2015a}, we will define and study the Fourier transform of the exterior calculus.

%%%%%%%%%%%%%%%%%%%%%%%%%%%%%%%%%%%%%%%%%%%%%%
\subsection{Tangent and Cotangent Bundle}

The setting of our discussion in the following will be $\R^n$ with $n = \{ 2, 3 \}$.
Although the tangent bundle $T \R^n$ can be identified with $\R^n \times \R^n$, for our work it is of importance to consistently distinguish $\R^n$ and the tangent and cotangent bundles, $T \R^n$ and $T^* \R^n$, respectively.
We will denote a generic basis for the tangent space as
\begin{align*}
 \mathrm{span}\left\{ \frac{\partial}{\partial x^1} , \cdots ,  \frac{\partial}{\partial x^n} \right\}  = T_x \R^n
\end{align*}
with the $\partial / \partial x^i$ in general being dependent on the fiber $T_x \R^n$ they span (such as in spherical coordinates).
%In contrast to the convention from differential geometry, we will work with normalized tangent vectors.
The $\partial / \partial x^i$ induce a dual basis $\dd x^i$ for the cotangent bundle $T^* \R^n$ through the fiberwise duality condition $\dd x^i \left( \partial / \partial x^j \right) = \delta_{ij}$.
We will also typically use the Einstein summation convention over repeated indices and hence employ ``upstairs'' indices for components of contra-variant vectors in $T_x \R^n$ and ``downstairs'' ones for covariant vectors in $T_x^* \R^n$.
For example, a vector field $X \in \mathfrak{X}(\R^n)$ is in coordinates given by $X^i \, \partial / \partial x^i$ and a co-vector field by $\alpha_i \, \dd x^i$ (summation implied).

%%%%%%%%%%%%%%%%%%%%%%%%%%%%%%%%%%%%%%%%%%%%%%
\subsection{Differential Forms and the Exterior Calculus}

A differential $\fdeg$-form $\alpha \in \Omega^{\fdeg}(\R^n)$ on $\R^n$, with $0 \leq \fdeg \leq n$, is a covariant, anti-symmetric tensor field of rank $\fdeg$, i.e. an anti-symmetric, multi-linear functional on the $\fdeg^{\textrm{th}}$ power of the tangent bundle.
When $\fdeg=0$ this corresponds to functions, i.e. $\Omega^0(\R^n) \cong \mathcal{F}(\R^n)$, and for $\fdeg=n$ one has densities, e.g. $f(x) \, \dd x$.
All other degrees $\fdeg$ correspond to ``vector-valued'' fields.
For example, a $1$-form is a co-vector field, i.e. a section of the cotangent bundle $T^* \R^n$, and since it is a functional on $T \R^n$ it naturally pairs with the tangent vector field of a curve.
A $2$-form, correspondingly, pairs with the tangent vectors of a $2$-manifold and in $\R^3$ it can hence be thought of as an infinitesimal area.
The result of the pairing between an $\fdeg$-form and the tangent vectors of an $\fdeg$-manifold can be integrated and one thus also says that an $\fdeg$-form is to be integrated over an $\fdeg$-manifold~\cite{Desbrun2006,Frankel2011}.
%The space of all forms or degree $\fdeg$ will be denoted by $\Omega^{\fdeg}(\R^n)$.

A principal motivation for studying differential forms is that they are closed under the exterior derivative $\dd$,
\begin{align}
  \label{eq:exterior_derivative}
  \dd \, : \, \Omega^{\fdeg}(\R^n) \to \Omega^{\fdeg+1}( \R^n)
\end{align}
which provides a coordinate invariant derivative for tensor fields.
%, by the anti-symmetry of $\dd$ and the differential forms, this is coordinate invariant.
%while ensuring coordinate invariance, which holds by the anti-symmetry of $\dd$ and the differential forms.
%It thereby subsumes the ordinary derivative in one dimensions and the first order differential operators of gradient, curl, and divergence in higher dimensions.
The exterior derivative furthermore satisfies $\dd \circ \dd = 0$.
% which corresponds to $\nabla \times \nabla = 0$ and related identities in classical vector calculus.
The differential forms under $\dd$ hence form a co-chain complex known as the de Rahm complex,
\begin{align*}
  0 \xrightarrow[]{ \ \dd \ } \Omega^0(\R^n) \xrightarrow[]{ \ \dd \ } \cdots \xrightarrow[]{ \ \dd \ } \Omega^n(\R^n) \xrightarrow[]{ \ \dd \ } 0  .
\end{align*}
When a differential form $\alpha \in \Omega^{\fdeg}(\R^n)$ satisfies $\dd \alpha = 0$ then $\alpha$ is closed and when there exists a $\beta$ such that $\beta = \dd \alpha$ then $\alpha$ is  exact.
%The Poincar{\'e} Lemma states that for every topologically trivial region in $\R^n$ any closed $\fdeg$ form is also exact.
The exterior calculus also provides a product on differential forms known as the wedge product
\begin{align}
   \label{eq:def:wedge}
  \wedge : \Omega^{\fdeg} \times \Omega^l \to \Omega^{\fdeg+l} .
\end{align}
It turns the de Rahm complex into a graded algebra.

With the basis $\{ \dd x^i \}_{i=1}^n$ for the cotangent bundle $T^* \R^n$, the coordinate expression for a $1$-form $\alpha \in \Omega^1(\R^3)$,  i.e. a section of $T^* \R^n$, is
\begin{align*}
  \alpha(x) = \alpha_1(x) \, \dd x^1 + \cdots + \alpha_n(x) \, \dd x^n .
\end{align*}
The basis functions for higher degree forms are generated using the wedge product in Eq.~\ref{eq:def:wedge}, which ensures anti-symmetry.
In $\R^3$, for example, one has
\begin{subequations}
\begin{align*}
%  \label{eq:diff_forms:r3}
 \Omega^0(\R^3) \ni \ f &= f(x)
  \\
  \Omega^1(\R^3) \ni \ \alpha &= \alpha_1(x) \, \dd x^1 + \alpha_2(x) \, \dd x^2 + \alpha_3(x) \, \dd x^3
  \\
  \Omega^2(\R^3) \ni \ \beta &= \beta_1(x) \, \dd x^2 \wedge \dd x^3 + \beta_2(x) \, \dd x^3 \wedge \dd x^1 + \beta_3(x) \, \dd x^1 \wedge \dd x^2
  \\
  \Omega^3(\R^3) \ni \ \gamma &= \gamma(x) \, \dd x^1 \wedge \dd x^2 \wedge \dd x^3  .
\end{align*}
\end{subequations}
For an arbitrary form $\alpha \in \Omega^{\fdeg}(\R^n)$ the coordinate expression is
\begin{align*}
    \alpha(x) = \sum_{j_1 < \cdots < j_{\fdeg}} \alpha_{j_1, \cdots , j_{\fdeg}}(x) \, \dd x^{j_1} \wedge \cdots \wedge \dd x^{j_{\fdeg}}
\end{align*}
with the summation being over all ordered sequences $j_1 < \cdots < j_{\fdeg}$, which again ensures the anti-symmetry of the form.

Except for our discussions of Stokes' theorem at the end of this section and in Sec.~\ref{sec:psiforms:stokes}, we will assume that the coordinate functions $\alpha_{j_1, \cdots , j_{\fdeg}}(x)$ are smooth, or at least sufficiently smooth for the operation under consideration, and real-valued.
Differential forms with coordinate functions that are distributions in the sense of Schwartz were introduced by de Rahm~\cite{deRahm1984} and are known as `currents'.
We will denote the space of such $(n-r)$-forms as $\mathcal{S}'(\Omega^{n-\fdeg},\R^n)$ and they are dual to the space $\mathcal{S}(\Omega^{\fdeg},\R^n)$ of differential $r$-forms whose coordinate functions are test functions.
The duality pairing $\langle \! \langle \cdot , \cdot \rangle \! \rangle$ between $\mathcal{S}'(\Omega^{\fdeg},\R^n)$ and $\mathcal{S}(\Omega^{\fdeg},\R^n)$ is, as in the scalar case, analogous to the $L_2$-inner product for differential forms, which will be introduced in the next paragraph.

\paragraph{Hilbert Space Structure and Hodge Dual}

So far we did not require a metric and this is an important feature of the exterior calculus.
When a metric is available, as is the case in $\R^n$, the exterior calculus can be equipped with additional structure.
We will introduce this structure next.

Let a non-degenerate, symmetric, covariant tensor field $g \in \mathcal{T}_2^0(\R^n)$ of rank $2$ be a given metric and $g^{-1} \in \mathcal{T}_0^2(\R^n)$ be the associated co-metric acting on co-vectors; the components will be denoted as $g_{ij}$ and $g^{ij}$, respectively.
Unless mentioned otherwise, we will work with the canonical metric on $\R^n$ (which, however, is not the identity in spherical coordinates).

The metric $g$ and its dual $g^{-1}$ can be used to identify vectors and co-vectors by ``raising'' and ``lowering'' indices.
Let $X,Y \in \mathfrak{X}(\R^n)$ and $\alpha, \beta \in \mathfrak{X}^*(\R^n) \cong \Omega^1(\R^n)$.
Then
\begin{align*}
  X^{\flat}(Y) &\equiv g(X,Y) = g_{ij} X^i \, Y^i \ && \Leftrightarrow && (X^{\flat})_j = g_{ij} \, X^i \, \in \Omega^1(\R^n)
  \\[4pt]
  \alpha^{\sharp}(\beta) &\equiv g(\alpha,\beta) = g^{ij} \alpha_i \, \beta_j  \ && \Leftrightarrow && (\alpha^{\sharp})^j  = g^{ij} \, \alpha_i \in \mathfrak{X}(\R^n) .
\end{align*}
%where $X^{\flat} \in T^* \R^n \cong \Omega^1(\R^n)$ and $\alpha^{\sharp} \in T \R^n$.
The flat operator $^{\flat}$ and the sharp operator $^{\sharp}$ introduced above are together known as musical isomorphism.
With these, the classical differential operators from vector calculus can be related to the exterior derivative.
Let $f \in \mathcal{F}(\R^3)$ and $u \in \mathfrak{X}(\R^3)$ be a vector field. Then
\begin{subequations}
  \label{eq:d:vector_calculus}
\begin{align}
  \nabla f &= ( \dd f )^{\sharp}
  \\
  \nabla \times u &= ( \star ( \dd u^{\flat}) )^{\sharp} .
  \\
  \nabla \cdot u &= \star \dd ( \star u^{\flat} ) .
\end{align}
\end{subequations}
%In many applications, e.g. electromagnetic theory and fluid dynamics, one does, indeed, have a differential form in the first place and hence directly the formulation in terms of differential forms and the exterior derivative can (and should) be used.

The metric induces, next to the musical isomorphism, also an $L_2$-inner product $\langle \! \langle \cdot , \cdot \rangle \! \rangle$ for differential $\fdeg$-forms.
Let $\alpha, \beta \in \Omega^{\fdeg}( \R^n )$. Then
\begin{subequations}
  \label{eq:forms:L2}
\begin{align}
  \label{eq:forms:L2:1}
  \langle \! \langle \alpha , \beta \rangle \! \rangle
  = \int_{\R^n}  \alpha_{j_1, \cdots , j_{\fdeg}} \, \beta^{j_1, \cdots , j_{\fdeg}} \, \dd x^1 \wedge \cdots \wedge \dd x^n
\end{align}
where the contra-variant tensor components $\beta^{j_1, \cdots , j_{\fdeg}}$ are obtained by index raising with the metric,
\begin{align}
  \label{eq:forms:L2:2}
  \beta^{j_1, \cdots , j_{\fdeg}} = g^{j_1,i_1} \cdots g^{j_{\fdeg},i_{\fdeg}} \beta_{i_1, \cdots , i_{\fdeg}} .
\end{align}
\end{subequations}
The space $L_2(  \Omega^{\fdeg} , \R^n  )$ is defined in the usual way as all differential $r$-forms with finite $L_2$ norm.
Unless stated otherwise (e.g. for Stokes' theorem), we will assume that all forms we work with are in $L_2( \Omega^{\fdeg} , \R^n )$.
Eq.~\ref{eq:forms:L2} can equivalently be written using the Hodge dual $\star \beta$ as
\begin{align*}
   \langle \! \langle \alpha , \beta \rangle \! \rangle = \int_{\R^n} \alpha \wedge \star \beta
\end{align*}
and the equality with Eq.~\ref{eq:forms:L2} defines $\star \beta$.
Since the integral in the last equation has to be over a volume form, i.e. one of maximum degree $\fdeg = n$, the Hodge dual has to satisfy
\begin{align*}
  \star : \Omega^{\fdeg}(\R^n) \to \Omega^{n-\fdeg}(\R^n)
\end{align*}
and it also holds that $\star \star \alpha = (-1)^{n(n-r)} \alpha$.
By linearity of the Hodge dual, which follows from those of the contraction in Eq.~\ref{eq:forms:L2:2} and the wedge product, for the coordinate expression of $\star \beta$ it suffices to know those for the form basis functions.
For $\R^3$, for example, these are
\begin{subequations}
  \label{eq:hodge_dual:R3}
\begin{align}
  \star 1 &= \dd x^1 \wedge \dd x^2 \wedge \dd x^3
  \\[3pt]
  \star \dd x^1 &= \dd x^2 \wedge \dd x^3
  \\[3pt]
  \star \dd x^2 &= -\dd x^1 \wedge \dd x^3
  \\[3pt]
  \star \dd x^3 &= \dd x^1 \wedge \dd x^2 .
\end{align}
\end{subequations}
%Unless mentioned otherwise, we will work in the following always with the canonical metric on $\R^n$.
%and one has for the composition of the Hodge dual $\star \, \star = (-1)^{\fdeg(n-\fdeg)}$; see also again Eq.~\ref{eq:diff_forms:r3}.

%For more details on differential forms and exterior calculus we refer again to Marsden, Ratiu, and Abraham~\cite{Marsden2004}.

%%%%%%%%%%%%%%%%%%%%%%%%%%%%%%%%%%%%%%%%%%%%%%%%%%%%%%%%%%%%%%%%%%%%%%%%%%%%%%%%%%%%%%%%%%%%%%%%%%%%
\subsection{The Fourier Transform of the Exterior Algebra}
\label{sec:fourier_forms}
\label{sec:exterior_algebra:fourier}

%As before, let $\Omega^{\fdeg}(\R^n)$ be the space of differential forms of degree $\fdeg$.
\paragraph{Motivation}
Before introducing the formal definition of the Fourier transform of differential forms, it is instructive to understand the duality that underlies it.

The classical result that integration becomes under the Fourier transform pointwise evaluation, i.e.
\begin{align}
  \label{eq:fouier_transform:vol_integral}
  \int_{\R^n} f(x) \, \d x = \hat{f}(0) ,
\end{align}
suggests that the volume form $f(x) \, \d x = f_{1 \cdots n} \d x^1 \wedge \cdots \wedge \d x^n \in \Omega^n(\R^n)$ should map to a $0$-form in frequency space, since for $0$-forms ``integration'' amounts to pointwise evaluation, cf.~\cite{Frankel2011}.
This is also required by covariance because if the volume form $f(x) \, \d x$ would map to another volume form (or to a form of degree other than zero) under the Fourier transform then it would ``pick up'' a Jacobian term under coordinate transformations.
But $\hat{f}(0)$ is already the integrated value and hence invariant under changes of coordinates (e.g. the mass of an object or the energy of a field).
This observation for volume forms can be consistently extended to differential forms of lower degree.
For example, for a plane in $\R^3$, which from the point of view of the exterior algebra is a $2$-form, the Fourier transform is known to be in the direction of the plane's normal, i.e. a $1$-form in this direction in frequency space.
For the $x_1$-$x_2$ plane we have, for example,
\begin{align*}
 \mathfrak{F} \Big( \delta_{{{1,2}}}(x) \Big) = \delta_{{3}}(\xi)
\end{align*}
where $\delta_{{{1,2}}}(x)$ is the distribution representing the $x_1$-$x_2$ plane in $\R^3$ and $\delta_{{3}}(\xi)$ its Fourier transform in the $\xi_3$ direction.
%we have to have that a $2$-form in space $\R^3$, which can be integrated over $\R_^2$, becomes a $1$-form in $\widehat{R}^3$, where it can be integrated over $\R_{\xi_3}^1$.
%Analogously, a volume form $\alpha \in \Omega^n(\R^n)$ in space, which can be identified with a function function using the Hodge dual, has to become a $0$-form in frequency space, since the integral of $\alpha$ over $\R^n$ is given by evaluation of $\widehat{\alpha}$ at the origin.
With the above considerations, the duality between differential forms in space and frequency is expected to be
%\begin{align*}
%  \Omega^{\fdeg}(\R^n) \ \ \substack{\xrightarrow[\hspace*{1.5cm}]{\displaystyle \mathfrak{F}} \\[-8pt] \xleftarrow[\displaystyle \mathfrak{F}^{-1}]{\hspace*{1.5cm}}} \ \ \widehat{\Omega}^{n-\fdeg}(\widehat{\R}^n) .
%\end{align*}
\begin{center}
  \includegraphics[width=0.52\textwidth]{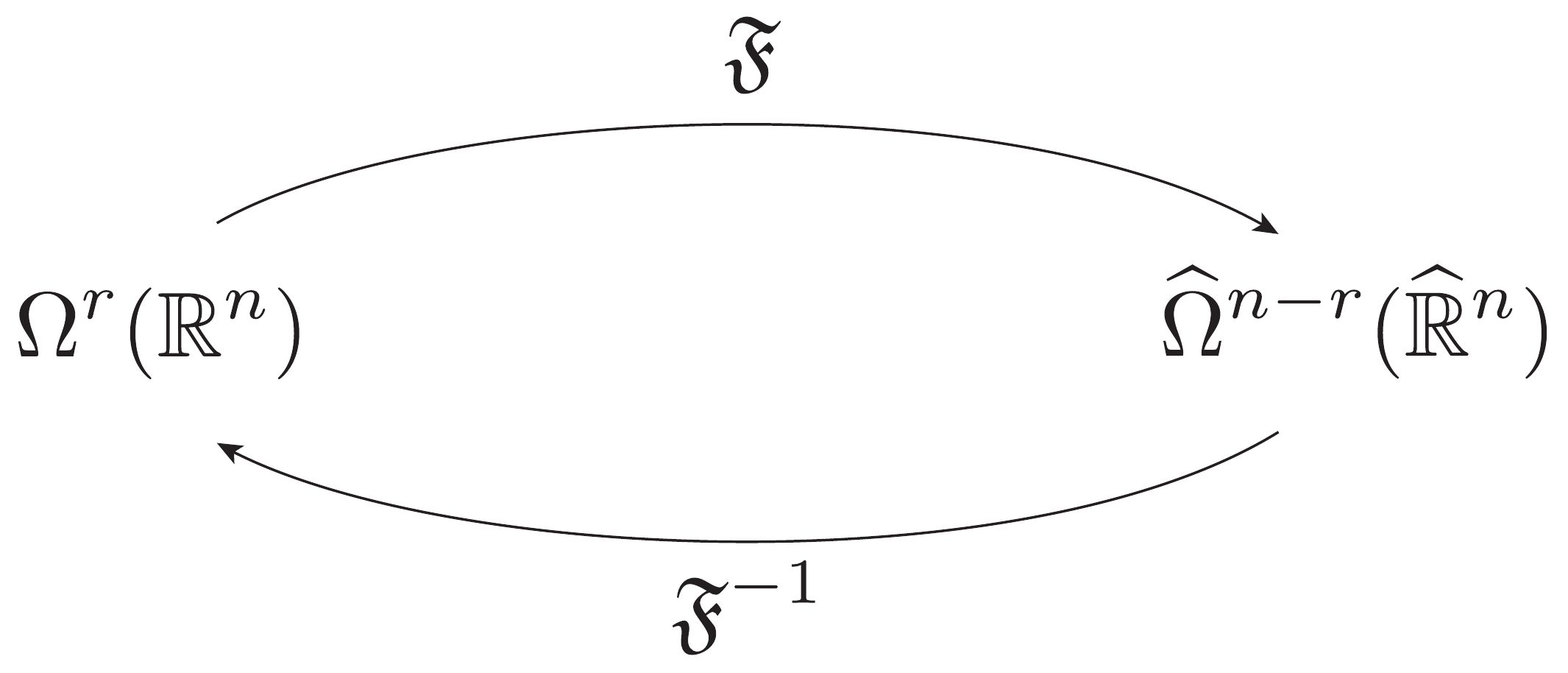}
\end{center}
where we denote by $\widehat{\Omega}^{n-\fdeg}(\widehat{\R}^n)$ the exterior algebra in frequency space $\widehat{R}^n$.

The spatial coordinate $x = x^i \, \partial / \partial x^i \in \R^n$ pairs with the covector $\xi = \xi_i \, \dd x^i \in \widehat{\R}^n$ in the complex exponential $e^{i x^i \xi_i}$ of the Fourier transform.
When we require that the usual correspondence between differential operators $\partial / \partial x^i$ and their Fourier transform $\xi_i \dd x^i$ (no summation) is to hold (in an appropriate sense, see Proposition~\ref{prop:exterior_derivative:fourier} below) we also have to expect that the exterior algebra in frequency space is naturally defined over $\partial / \partial \xi^i \equiv \partial / \partial x^i$ where we use the notation $\partial / \partial \xi^i$ to be explicit if we are considering the primary or dual domain with respect to the Fourier transform.
Differential forms in frequency space are hence defined over form basis functions $\partial / \partial \xi^{j_1} \wedge \cdots \wedge \partial / \partial \xi^{j_r}$ and this provides the second cornerstone for the definition of the Fourier transform of differential forms.
% $\partial / \partial x -> \xi = \xi_i \d x^i$ and thus, by duality $\dd x^i -> \partial / \partial \xi^i$

%%%%%%%%%%%%%%%%%%%%%%%%%%%%%%%%%%%%%%%%%%%%%%%%%
\paragraph{The Fourier Transform of Differential Forms}

With the above motivation, we define the Fourier transform of differential forms, following an approach introduced in mathematical physics~\cite{Kalkman1993,DeWitt1992,Castellani2015a} based on the theory of super vector spaces.
\commentCL{The most relevant work is~\cite{Castellani2015a} where things are spelled out explicitly for differential forms and in an accessible notation.}

\begin{definition}
  \label{def:diff_form:ft}
  Let $\alpha \in \Omega^{\fdeg}(\R^n)$ be a differential form of degree $\fdeg$,
  \begin{align}
    \alpha(x) = \sum_{j_1 < \cdots < j_{\fdeg}} \alpha_{j_1, \cdots , j_{\fdeg}}(x) \, \dd x^{j_1} \wedge \cdots \wedge \dd x^{j_{\fdeg}} . %\in \Omega^{\fdeg}(\R^n) .
    \nonumber
  \end{align}
  Then the \emph{Fourier transform $\widehat{\alpha} = \mathfrak{F}(\alpha) \in \widehat{\Omega}^{n-\fdeg}(\widehat{\R}^n)$ of $\alpha$} is the $(n-\fdeg)$-form on frequency space $\widehat{\R}^n$ given by
  \begin{align}
    \mathfrak{F}(\alpha)(\xi) = \frac{1}{(2\pi)^{n/2}} \sum_{j_1 < \cdots < j_{\fdeg}} \int_{\R^n} \alpha_{j_1, \cdots , j_{\fdeg}}(x) \, \dd x^{j_1} \wedge \cdots \wedge \dd x^{j_{\fdeg}} \wedge e^{-i \, x^{p} \xi_{p}} \wedge e^{i \, \dd x^{q} \dpp{}{\xi^{q}}}
    \nonumber
  \end{align}
  where the complex exponential is defined as
  \begin{align}
    \label{eq:def:diff_form:ft:exp}
       e^{\dd x^{q} \dpp{}{\xi^q}} = \sum_{a=0}^n \frac{(\dd x^q \otimes \dpp{}{\xi^q})^{^{\scriptstyle a}}}{a !} .
  \end{align}
% with summation in $i$ implied.
 The \emph{inverse Fourier transform $\alpha = \mathfrak{F}^{-1}(\widehat{\alpha}) \in \Omega^{\fdeg}(\R_{x}^n)$ of $\widehat{\alpha} \in \widehat{\Omega}^{n-\fdeg}(\widehat{\R}^n)$} is the $\fdeg$-form in space given by
  \begin{align}
    \mathfrak{F}^{-1}(\widehat{\alpha})(x) = \frac{-1}{(2\pi)^{n/2}} \sum_{j_1 < \cdots < j_{n-\fdeg}} \int_{\widehat{\R}^n} \widehat{\alpha}^{j_1, \cdots , j_{n-\fdeg}}(\xi) \, \frac{\partial}{\partial \xi^{j_1}} \wedge \cdots \wedge \frac{\partial}{\partial \xi^{j_{n-\fdeg}}} \wedge e^{i \, \xi_p x^p} \wedge e^{\dpp{}{\xi^q} \, \dd x^q}  .
    \nonumber
  \end{align}
\end{definition}

From a functional analytic perspective, the usual requirements for the existence of the Fourier transform apply with respect to the coordinate functions $\alpha_{j_1, \cdots , j_\fdeg}(x)$ and $\widehat{\alpha}^{j_1, \cdots , j_\fdeg}(\xi)$.

\begin{remark}
  In the literature, e.g.~\cite{Kalkman1993,Castellani2015a}, the exponential $e^{\dd x^q \dpp{}{\xi^q}}$ of the form basis functions is typically defined with a complex unit in the exponent, which is considered there as only serving ``aesthetic'' purposes to visually closer match the scalar case.
  For real-valued functions this is, however, not true and the use of $e^{i \dd x^{q} \dpp{}{\xi^q}}$ leads to inconsistencies in the inverse Fourier transform.
  In particular, there are no differential $r$-form wavelets that satisfy the de Rahm complex and that are real-valued in space.
  We hence use the definition above without the complex unit.
%  For our work, it would lead to an unnecessary cluttering of most equations with powers of the complex unit.
%  We hence omit it, at the price of additional signs.
\end{remark}

To evaluate the exponential in Eq.~\ref{eq:def:diff_form:ft:exp} one uses that the exterior algebra has a $\Z^2$-graded structure, with the even grade given by the differential forms of even degree and the odd one by differential forms of odd degree.\footnote{With the $\Z^2$-grading the exterior algebra is a super vector space. We will not pursue this viewpoint in our work.}
The $\Z^2$-grading is respected in the tensor product by using the following multiplication rule (for homogenous elements):
%For the tensor product $\Omega(\R_{x}^n) \otimes \widehat{\Omega}(\widehat{\R}^n)$ of the two $\Z^2$-graded algebras $\Omega(\R_{x}^n)$ and $\widehat{\Omega}(\widehat{\R}^n)$ one has the multiplication rule (for homogenous elements)
% the equation is essentially the definition of the multiplication for the tensor
\begin{subequations}
    \label{eq:fourierforms:tensor_product:parity}
\begin{align}
  \label{eq:fourierforms:tensor_product:parity:1}
  \big( \alpha \otimes \widehat{\beta} \big) \wedge_{x,\xi} \big( \gamma \otimes \widehat{\delta} \big) = (-1)^{\mathrm{par}(\gamma) \mathrm{par}(\widehat{\beta})} \big( \alpha \wedge_x \gamma \otimes \widehat{\beta} \wedge \widehat{\delta} \big) ,
\end{align}
e.g.~\cite[Ch. III.4]{Bourbaki1989} or~\cite{Castellani2015a}.
Here $\mathrm{par}(\cdot)$ is the parity, which is $1$ when the degree of a differential form is odd and $0$ otherwise.
Since the order of the spaces matters for the multiplication rule, one has for the inverse Fourier transform
\begin{align}
  \label{eq:fourierforms:tensor_product:parity:2}
  \big( \widehat{\alpha} \otimes \beta \big) \wedge_{\xi,x} \big( \widehat{\gamma} \otimes {\delta} \big) = (-1)^{\mathrm{par}(\widehat{\gamma}) \mathrm{par}({\beta})} \big( \widehat{\alpha} \wedge \widehat{\gamma} \otimes {\beta} \wedge_{x} {\delta} \big) .
\end{align}
\end{subequations}

\begin{example}
In $\R^2$ the highest order term in the exponential of the form basis functions in Eq.~\ref{eq:def:diff_form:ft:exp} yields
\begin{align*}
  \left( \! \dd x^q \otimes \dpp{}{\xi^q} \! \right)^2
  &= \left( \dd x^1 \otimes \dpp{}{\xi^1} + \dd x^2 \otimes \dpp{}{\xi^2} \right) \wedge_{x,\xi} \left( \dd x^1 \otimes \dpp{}{\xi^1} + \dd x^2 \otimes \dpp{}{\xi^2}  \right) .
  \intertext{Since the ``square'' terms vanish by the anti-symmetry of the wedge product this equals}
  \left( \! \dd x^q \otimes \dpp{}{\xi^q} \! \right)^2
  &= \dd x^1 \otimes \dpp{}{\xi^1} \wedge_{x,\xi} \dd x^2 \otimes  \dpp{}{\xi^2} + \dd x^2 \otimes  \dpp{}{\xi^2} \wedge_{x,\xi} \dd x^1 \otimes \dpp{}{\xi^1}
  \\[7pt]
  &= (-1)^{1 \cdot 1} \, \dd x^1 \wedge_{x} \dd x^2 \otimes \dpp{}{\xi^1} \wedge \dpp{}{\xi^2} \\
  & \quad \quad \quad \quad \quad \quad \quad + (-1)^{1 \cdot 1} \, \dd x^2 \wedge_{x} \dd x^1 \otimes \dpp{}{\xi^2} \wedge \dpp{}{\xi^1}
  \nonumber
   \\[7pt]
  &= -2 \, \big( \dd x^1 \wedge_{x} \dd x^2 \otimes \dpp{}{\xi^1} \wedge \dpp{}{\xi^2} \big) .
\end{align*}
\commentCL{
\begin{itemize}
  \item The factor of 2 cancels with the $a!$ factor in the power series!
  \item The above calculation carries above to the volume form term in three dimensions.
\end{itemize}
}
\end{example}

%with the ``square'' terms vanishing by the antisymmetry of the wedge product.
%We will also typically omit the tensor product symbol as on the left hand side of the above equation.
To simplify the notation, we will typically not explicitly write the tensor product as above.
When the exponential of the form basis functions in Eq.~\ref{eq:def:diff_form:ft:exp} has been expanded, the usual rules for integration of differential forms apply for evaluating the Fourier transform.
In particular, the integral vanishes unless Eq.~\ref{eq:def:diff_form:ft:exp} together with the the basis functions $\dd x^{j_1} \wedge \cdots \wedge \dd x^{j_{\fdeg}}$ of $\alpha$ yield a volume form.
This is also the reason why it suffices to consider the power series in Eq.~\ref{eq:def:diff_form:ft:exp} up to $n$.

\begin{example}
  \label{ex:ft:1_form:r3}
  For $\alpha = \alpha_1(x) \, \dd x^1 \in \Omega^1(\R^3)$ we want to compute the Fourier transform
  \begin{align*}
    \mathfrak{F}(\alpha)(\xi)
    &= \frac{1}{(2\pi)^{3/2}} \int_{\R^n} \alpha_{1}(x) \, \dd x^1 \wedge e^{-i x_p \xi^p} \wedge e^{\dd x^q \dpp{}{\xi^q}} .
  \end{align*}
  The exponential $e^{\dd x^q \dpp{}{\xi^q}}$ of the form basis functions expands in $\R^3$ as
  \begin{align*}
    & e^{\dd x^q \dpp{}{\xi^q}}
%    = 1 + \left(\dd x^1 \dpp{}{\xi^1} + \dd x^2 \dpp{}{\xi^2} + \dd x^3 \dpp{}{\xi^3} \right)
%    \nonumber
%    \\[3pt]
%    & \quad \ \ \ + \left(\dd x^2 \dpp{}{\xi^2} \wedge \dd x^3 \dpp{}{\xi^3} + \dd x^1 \dpp{}{\xi^1} \wedge \dd x^3 \dpp{}{\xi^3} +  \dd x^1 \dpp{}{\xi^1} \wedge \dd x^2 \dpp{}{\xi^2} \right)
%    \\[3pt]
%    & \quad \ \ \ + \left( \dd x^1 \dpp{}{\xi^1} \wedge \dd x^2 \dpp{}{\xi^2} \wedge \dd x^3 \dpp{}{\xi^3} \right) .
   = 1 + \left(\dd x^1 \dpp{}{\xi^1} + \dd x^2 \dpp{}{\xi^2} + \dd x^3 \dpp{}{\xi^3} \right)
    \\[3pt]
    & \quad \ \ \ - \left(\dd x^2 \wedge \dd x^3 \, \dpp{}{\xi^2} \wedge \dpp{}{\xi^3} + \dd x^1 \wedge \dd x^3 \, \dpp{}{\xi^1} \wedge \dpp{}{\xi^3} +  \dd x^1 \wedge \dd x^2 \, \dpp{}{\xi^1} \wedge \dpp{}{\xi^2} \right)
    \nonumber \\[3pt]
    & \quad \ \ \ - \left( \dd x^1 \dpp{}{\xi^1} \wedge \dd x^2 \dpp{}{\xi^2} \wedge \dd x^3 \dpp{}{\xi^3} \right) .
    \nonumber
    \end{align*}
  For $\alpha = \alpha_1(x) \, \dd x^1$ the integral of the Fourier transform becomes a volume form hence only with the second line and only the first term yields a nonzero expression by the anti-symmetry of the wedge product.
  Thus,
%  \begin{subequations}
  \begin{align*}
    \mathfrak{F}(\alpha)(\xi)
%    &= \frac{1}{(2\pi)^{3/2}} \int_{\R^n} \alpha_{1}(x) \, \dd x^1 \wedge e^{-i \, \xi^i x_i} \wedge e^{\dd x^i \dpp{}{\xi^i}}
%    \\[5pt]
    &= \frac{1}{(2\pi)^{3/2}} \int_{\R^n} \alpha_{1}(x) \, e^{-i \, x_i \, \xi^i } \, \dd x^1 \wedge \left( -\dd x^2 \wedge \dd x^3 \, \dpp{}{\xi^2} \wedge \dpp{}{\xi^3} \right)
    \\[7pt]
    &= \frac{-1}{(2\pi)^{3/2}} \, \widehat{\alpha}_1(\xi) \, \dpp{}{\xi^2} \wedge \dpp{}{\xi^3}
  \end{align*}
%  \end{subequations}
  where $\widehat{\alpha}_1(\xi)$ is the scalar Fourier transform of the coordinate function $\alpha(x)$.
\end{example}

By linearity, for the differential form part of the Fourier transform it suffices to know the  transform of the form basis functions.
In $\R^2$ these are
\begin{subequations}
\label{eq:diff_form:ft:r2}
\begin{align}
  \mathfrak{F} \big( 1_x \big) &= -\dpp{}{\xi^{1}} \wedge \dpp{}{\xi^{2}}
  \\[4pt]
  \mathfrak{F} \big( \dd x^{j_1} \big) &= \sign{(\sigma)} \, \dpp{}{\xi^{j_2}}
  \\[4pt]
  \mathfrak{F} \big( \dd x^{1} \wedge \dd x^{2} \big) &= 1_{\xi}
\end{align}
\end{subequations}
with the permuation $\sigma$ being $\sigma = \big(\substack{1 \, , \, 2 \\ j_1, j_2 } \big)$.
In $\R^3$ one has
\begin{subequations}
\label{eq:diff_form:ft:r3}
\begin{align}
  \mathfrak{F} \big( 1_x \big) &= - \dpp{}{\xi^{1}} \wedge \dpp{}{\xi^{2}} \wedge \dpp{}{\xi^{3}}
  \\[4pt]
  \mathfrak{F} \big( \dd x^{j_1} \big) &= -\sign{(\sigma)} \, \dpp{}{\xi^{j_2}} \wedge \dpp{}{\xi^{j_3}}
  \\[4pt]
  \mathfrak{F} \big( \dd x^{j_1} \wedge \dd x^{j_2} \big) &= \sign(\sigma) \, \dpp{}{\xi^{j_3}}
  \\[4pt]
  \mathfrak{F} \big( \dd x^{1} \wedge \dd x^{2} \wedge \dd x^{3} \big) &= 1_{\xi}
\end{align}
\end{subequations}
with the permutation being $\sigma = \big(\substack{1 \, , \, 2, \, 3 \\ j_1, j_2,j_3} \big)$.

\begin{remark}
  \label{remark:diff_form:ft:index_mapping}
%  The prescription in Eq.~\ref{eq:def:diff_form:ft} is not very practical to determine which $\partial / \partial \xi^i$ appear in the Fourier transform of a form basis function.
  Comparing Eq.~\ref{eq:diff_form:ft:r3} to Eq.~\ref{eq:hodge_dual:R3} we see that the Fourier transform of the differential form basis functions has a structure analogous to those of the Hodge dual.
%  Compared to Eq.~\ref{eq:def:diff_form:ft:exp},
  A general description of the Fourier transform of the $\dd x^{j_1} \wedge \cdots \wedge \dd x^{j_\fdeg}$ can hence be obtained when we borrow some notation from the general description of the Hodge star~\cite[Ch. 7.2]{Marsden2004}.
  Consider a permutation
  \begin{align*}
  \sigma(j) =
  \begin{pmatrix}
    j_1 & \cdots & j_{\fdeg} & \cdots
    \\
    \sigma_1 & \cdots & \sigma_{\fdeg} & \sigma_{n-\fdeg+1} & \cdots & \sigma_n
  \end{pmatrix} \in S_n
  \end{align*}
  whose first $\fdeg$ indices match the multi-index $j=(j_1, \cdots , j_r)$ of the given differential form basis function. Here $S_{n}$ is the group of permutations of $\{ 1, \cdots , n \}$.
  One then has
  \begin{align*}
    \mathfrak{F}\left( \dd x^{j_1} \wedge \cdots \wedge \dd x^{j_{\fdeg}} \right)
    =  -(-1)^{\lfloor r/2 \rfloor} \, \mathrm{sgn}(\sigma) \ \dpp{}{\xi^{\sigma_{n-\fdeg+1}}} \wedge \cdots \wedge \dpp{}{\xi^{\sigma_{n}}} .
% Real version, i.e. for power series defined without complex unit
%    = \mathrm{sgn}(\sigma) \, (-1)^{\lfloor (n-\fdeg) / 2 \rfloor} \, \dpp{}{\xi^{\sigma_{n-\fdeg+1}}} \wedge \cdots \wedge \dpp{}{\xi^{\sigma_{n}}} .
  \end{align*}
  % COMMENT: can the sign factor be obtained from Eq.~\ref{eq:fourierforms:tensor_product:parity}? Can / should it be formulated in terms of the degree
\end{remark}

A direct consequence of our definition of the Fourier transform of differential forms is the following.

\begin{proposition}
  \label{prop:fourier_form:inverse}
  $\mathfrak{F}^{-1} \circ \mathfrak{F}$ is the identity.
\end{proposition}

\begin{proof}
  This can be checked with a coordinate calculation using the above rules (or~\cite[Proposition 5.2]{Kalkman1993}).
\end{proof}

%%%%%%%%%%%%%%%%%%%%%%%%%%%%%%%%%%%%%%%
\paragraph{Exterior derivative}
With the Fourier transform of differential forms being defined, we can consider those of the exterior derivative.
It is given in the next proposition.
%, which is again based on the discussion by Kalkman~\cite[Proposition 5.3]{Kalkman1993}.

\begin{proposition}
  \label{prop:exterior_derivative:fourier}
  The Fourier transform (or principal symbol) of the exterior derivative $\dd : \Omega^{\fdeg}(\R^n) \to \Omega^{\fdeg+1}(\R^n)$ is the anti-derivation $\widehat{\dd} : \widehat{\Omega}^{n-\fdeg}(\widehat{\R}^n) \to \widehat{\Omega}^{n-\fdeg-1}(\widehat{\R}^n)$ given by the interior product $\ip_{i \xi}\widehat{\alpha}$, i.e.
  \begin{align*}
    \mathfrak{F}\big( \dd \alpha )(\xi)
%    = \widehat{\dd \alpha}(\xi)
    = \big(\widehat{\dd} \, \widehat{\alpha}\big)(\xi)
    = \ip_{i \xi}\widehat{\alpha}(\xi)
  \end{align*}
  where $\xi$ is the coordinate vector $\xi = \xi_1 \, \dd \xi^1 + \cdots + \xi_n \, \dd \xi^n$.
  The following diagram thus commutes
  \commentCL{replace with diagram from diagram's package (http://www.paultaylor.eu/diagrams/manual.pdf)}
  \begin{align*}
    \begin{CD}
%    {\rlap{$\scriptstyle{\ \ \ \text{shorter}}$}\phantom{\text{very long label}}}
%    \Omega^{0} \to \cdots \to
    \Omega^{\fdeg}(\R^n) @>  {\rlap{$\ \ \ \ \ \ \ \dd$}\phantom{\text{very long label}}}  >> \Omega^{\fdeg+1}(\R^n)\\[7pt]
    @V {\displaystyle \mathfrak{F}} VV @VV {\displaystyle \mathfrak{F}} V
    \\
%    \widehat{\Omega}^{n} \to \cdots \to
    \widehat{\Omega}^{n-\fdeg}(\widehat{\R}^n) @> {\displaystyle \quad \hat{\dd} = \ip_{i \xi} \quad} >> \widehat{\Omega}^{n-\fdeg-1}(\widehat{\R}^n)
    \end{CD}
  \end{align*}
  and the Fourier transform of the de Rahm co-chain complex is a chain complex.
%  $\, \Omega^{\fdeg}(\R^n) \xrightarrow{ \ \dd \ } \Omega^{\fdeg+1}(\R^n)$ is a chain complex $\widehat{\Omega}^{\fdeg}(\widehat{\R}^n) \xrightarrow{ \ \hat{\dd} = \ip_{i \xi} \ } \Omega^{\fdeg-1}(\widehat{\R}^n)$ .
\end{proposition}

\begin{proof}
  The coordinate expression of the exterior derivative is
  \begin{subequations}
  \label{proof:exterior_derivative:fourier}
  \begin{align*}
   \dd \alpha(x) = \sum_{p=1}^n \sum_{j_1 < \cdots < j_{\fdeg}} \dpp{\alpha_{j_1, \cdots , j_{\fdeg}}}{x^p} \, \dd x^p \wedge \dd x^{j_1} \wedge \cdots \wedge \dd x^{j_{\fdeg}} .
  \end{align*}
  For each nonzero term there is a permutation $\sigma(j,p) = (j_1, \cdots , p, \cdots , j_r) \in S_n$, analogous to those in Remark~\ref{remark:diff_form:ft:index_mapping}, that brings $\dd x^{p}$ in canonical, increasing order, introducing a sign factor $\mathrm{sgn}(\sigma(j_p))$ through the anti-symmetry of the wedge product.
%  with the second sum being over all possible permutations satisfying the anti-symmetry condition $j_1 < \cdots < j_{\fdeg}$.
  By linearity it suffices to consider the $p^{\textrm{th}}$ term in the following.
  Its Fourier transform is
  \begin{align}
   \mathfrak{F} \big( \dd \alpha_p \big)
   = s(r+1) \! \sum_{j_1 < \cdots < j_{\fdeg}} i \xi^p \,  \widehat{\alpha}_{j_1, \cdots , j_{\fdeg}} \, \mathrm{sgn}(\sigma(j_p)) \, \dpp{}{\xi^{\sigma(j_p)_{n-\fdeg+2}}} \wedge \cdots \wedge \dpp{}{\xi^{\sigma(j_p)_n}}
   \nonumber
  \end{align}
  where we use $s(r+1) = - (-1)^{\lfloor (r+1)/2 \rfloor}$.
%  \begin{align}
%  \begin{pmatrix}
%    \label{proof:exterior_derivative:fourier:permuation}
%    \sigma_1 & \sigma_2 & \cdots & \sigma_{k+1}  & \sigma_{n-\fdeg+2} & \cdots & \sigma_n
%    \\
%    i & j_1 & \cdots & j_{\fdeg} & \cdots
%  \end{pmatrix}
%  \end{align}
%  so that the first $(k+1)$ elements describe the basis for the $(\fdeg+1)$-form in the spatial domain and the remaining are those obtained under the Fourier transform.
  On the other hand, the interior product $\ip_{i\xi} \widehat{\alpha}$ on the the right hand side of the proposition is for the $p^{\textrm{th}}$ term in coordinates given by
%  \begin{align}
%    \ip_{i \xi} \widehat{\alpha} = - \!  \sum_{p=1}^n i \xi^p \dd \xi_p \sum_{j_1 < \cdots < j_{\fdeg}} \widehat{\alpha}_{j_1, \cdots , j_{\fdeg}}(\xi) \,  (-1)^{\lfloor r/2 \rfloor} \, \mathrm{sgn}(\bar{\sigma}) \, \dpp{}{\xi^{\bar{\sigma}_{n-\fdeg+1}}} \wedge \cdots \wedge \dpp{}{\xi^{\bar{\sigma}_{n}}} .
%    \nonumber
%  \end{align}
%  Considering again only  and using linearity we have
  \begin{align}
    \ip_{i \xi_p} \widehat{\alpha} = s(r) \! \! \sum_{j_1 < \cdots < j_{\fdeg}} i \xi^p \, \widehat{\alpha}_{j_1, \cdots , j_{\fdeg}} \, \mathrm{sgn}(\bar{\sigma}(j)) \, \dd \xi_p \! \left( \dpp{}{\xi^{\bar{\sigma}(j)_{n-\fdeg+1}}} \wedge \cdots \wedge \dpp{}{\xi^{\bar{\sigma}(j)_{n}}} \right).
    \nonumber
  \end{align}
  where the $\bar{\sigma}(j)$ are the permutations that arises in the Fourier transform of $\alpha$.
  The expression on the right hand side will be nonzero only when $\partial / \partial \xi^{p}$ is present in the wedge product of the basis functions.
  In this case it pairs with $\dd \xi_p$ after a permutation $\bar{\sigma}(p,j)'$ has been applied.
  Denoting the combined permutation as $\tilde{\sigma}(j_p)$, i.e. $\tilde{\sigma}(j_p) = \bar{\sigma}(p,j)' \circ \bar{\sigma}(j)$ we have
  \begin{align}
    \ip_{i \xi_p} \widehat{\alpha} = s(r) \! \! \sum_{j_1 < \cdots < j_{\fdeg}} \! \! i \xi^p \, \widehat{\alpha}_{j_1, \cdots , j_{\fdeg}} \, \textrm{sgn}(\tilde{\sigma}) \, \dd \xi_p \! \left( \dpp{}{\xi^{p}} \right) \! \wedge \dpp{}{\xi^{\tilde{\sigma}_{n-\fdeg+2}}} \wedge \cdots \wedge \dpp{}{\xi^{\tilde{\sigma}_n}}  .
    \nonumber
  \end{align}
  \end{subequations}
  % TODO: what happens with the parity factor (-1)^{\lfloor (n-\fdeg) / 2 \rfloor}
  It can be checked that $\tilde{\sigma}$ equals up to a sign the permutation $\sigma$ in $\mathfrak{F} \big( \dd \alpha_p \big)$ and the sign is compensated by $(-1)^{\lfloor (r+1)/2 \rfloor}$.

  That $\hat{\dd}$ is an anti-derivation follows from the fact that the interior product is one~\cite[Ch. 4.6]{Nguyen2014}.
  This together with the properties of the differential forms over $\partial / \partial \xi^i$ and that $\ip_{i \xi} \, \ip_{i \xi} = 0$~\cite[p. 428]{Marsden2004} implies that the $\widehat{\Omega}^{\fdeg}(\widehat{\R}^n)$ form a chain complex under $\hat{\dd} = \ip_{i \xi}$.
\end{proof}

\begin{remark}[Remarks on Proposition~\ref{prop:exterior_derivative:fourier}]\leavevmode
  \begin{enumerate}

    \item In the literature the interior product $\ip_{i \xi}$ with the coordinate vector $i \xi$ is also known as the Koszul differential, see e.g.~\cite[Ch. 7.2]{Arnold2018}.

   \item The position vector $\xi$ in the symbol $\widehat{\dd} = \ip_{i \xi}$ of the exterior derivative is naturally expressed in spherical coordinates, where $\xi = \vert \xi \vert \, \dd \hat{r}$ and $\dd \hat{r}$ is the radial basis vector. Proposition~\ref{prop:exterior_derivative:fourier} hence suggests that spherical coordinates are natural for working with the de Rahm complex in frequency space.
This is a key observation that underlies the construction of our local spectral exterior calculus in the next section.
It will also play an important role for the description of the Hodge-Helmholtz decomposition in frequency space in the following.

  \end{enumerate}
\end{remark}
%In fact, the interior product is also referred to as interior derivative~\cite[Ch. 4.6]{Nguyen2014}.

\begin{example}
  Let $\alpha(x) = \alpha_1 \dd x^1 + \alpha_2 \dd x^2 + \alpha_3 \dd x^3 \in \Omega^1(R^3)$. Then $\dd \alpha(x) \in \Omega^2(\R^3)$ and it corresponds to the curl of $\alpha^{\sharp}$ in classical vector calculus, cf. Eq.~\ref{eq:d:vector_calculus}.
  In frequency space, $\alpha(x)$ is the $2$-form
  \begin{align*}
    \widehat{\alpha}(\xi) = -\widehat{\alpha}_1(\xi) \, \dpp{}{\xi^2} \wedge \dpp{}{\xi^3} + \widehat{\alpha}_2(\xi) \, \dpp{}{\xi^1} \wedge \dpp{}{\xi^3} - \widehat{\alpha}_3(\xi) \, \dpp{}{\xi^1} \wedge \dpp{}{\xi^2} ,
  \end{align*}
  cf. Example~\ref{ex:ft:1_form:r3}.
  The interior product that defines the exterior derivative in frequency space is evaluated using the Leibniz rule for it and linearity.
  For example, for the first term of the last equation this yields
  \begin{align*}
    \ip_{i \xi} \Big( -\widehat{\alpha}_1(\xi) \, \dpp{}{\xi^2} \wedge \dpp{}{\xi^3} \Big)
    &= -\widehat{\alpha}_1(\xi) \, \ip_{i \xi} \Big( \dpp{}{\xi^2} \wedge \dpp{}{\xi^3} \Big)
    \\[4pt]
    &= -\widehat{\alpha}_1(\xi) \, \Big( \ip_{i \xi} \dpp{}{\xi^2} \wedge \dpp{}{\xi^3} - \dpp{}{\xi^2} \wedge \ip_{i \xi} \dpp{}{\xi^3} \Big)
    \\[4pt]
    &= -\widehat{\alpha}_1(\xi) \, \Big( i \xi_2 \, \dpp{}{\xi^3} - i \xi_3 \, \dpp{}{\xi^2} \Big)
  \end{align*}
  With the analogous calculations also for the other terms we obtain
  \begin{align*}
    -\big(\ip_{i\xi} \widehat{\alpha}\big)(\xi)
    &=
    \ - i \xi_1 \, \widehat{\alpha}_2(\xi) \dpp{}{\xi^3} + i \xi_1 \,\widehat{\alpha}_3(\xi) \dpp{}{\xi^2}
    \nonumber
    \\[3pt]
    & \ \ \ \ + i \xi_2 \,  \widehat{\alpha}_1(\xi) \dpp{}{\xi^3} - i \xi_2 \,\widehat{\alpha}_3(\xi) \dpp{}{\xi^1}
    \\[3pt]
    & \ \ \ \ - i \xi_3 \,  \widehat{\alpha}_1(\xi) \dpp{}{\xi^2} + i \xi_3 \,\widehat{\alpha}_2(\xi) \dpp{}{\xi^1} .
    \nonumber
  \end{align*}
  Using the musical isomorphism, the last equation can also be written as
  \begin{align*}
    \big(\ip_{\xi} \widehat{\alpha}\big)^{\flat}(\xi) = i \vec{\xi} \times \widehat{\alpha}^{\flat}(\xi) .
  \end{align*}
  This is the classical expression for the curl in the Fourier domain.
\end{example}

\begin{remark}
  In the literature, the principal symbol of the exterior derivative is also given as the exterior product
  \begin{align}
    \label{eq:def:dhat:classical}
    \widehat{\dd \alpha}'(\xi) = i \xi \wedge \widehat{\alpha} ,
  \end{align}
  see   e.g.~\cite[Vol. 1, I.7]{Treves1980},~\cite[Sec. 3.2]{Duistermaat1996}, or~\cite[Ch. 8.5]{Marsden2004} (the prime denoting the alternative definition).
  The expression is obtained when one uses the following definition of the Fourier transform of differential forms, e.g.~\cite[Sec. 3]{Hiptmair2012b},
  \begin{align}
    \label{eq:def:fourier:classical}
    \widehat{\alpha}'(\xi)
    = \mathfrak{F}'\Bigg( \! \sum_{j_1 \cdots j_{\fdeg}} \alpha_{j_1 \cdots j_{\fdeg}} \, \dd x^{j_1} \cdots \dd x^{j_{\fdeg}} \! \Bigg)
    = \sum_{j_1 \cdots j_{\fdeg}} \widehat{\alpha}_{j_1 \cdots j_{\fdeg}}(\xi) \, \dd' \xi^{j_1} \cdots \dd' \xi^{j_{\fdeg}}
  \end{align}
  where $\widehat{\alpha}_{j_1 \cdots j_{\fdeg}}(\xi)$ is the usual scalar Fourier transform of the coordinate functions.
  In contrast to Def.~\ref{def:diff_form:ft}, with Eq.~\ref{eq:def:fourier:classical} the Fourier transform of an $\fdeg$-form is hence again an $\fdeg$-form.
  A straight forward coordinate calculation shows that Eq.~\ref{eq:def:dhat:classical} holds with this definition.
  The expression can be related to ours in Proposition~\ref{prop:exterior_derivative:fourier} using what in the literature is sometimes referred to as Hirani's formula~\cite{Hirani2003,Eldred2018}, $\ip_{X} \alpha = (-1)^{r(n-r)} \star (\star \alpha \wedge X^{\flat} )$, where $\alpha$ is an $r$-form and $X$ a vector field.
  Hence, the interior product $\ip_{i \xi} \widehat{\alpha}$ is mapped to the wedge product $i \xi \wedge \widehat{\alpha}$ and by interpreting the $\dd' \xi^{i}$ in Eq.~\ref{eq:def:fourier:classical} as contra-variant vectors both formulations are equivalent.
  In light of Eq.~\ref{eq:fouier_transform:vol_integral} and the discussion at the beginning of the section, it appears to us more natural that the Fourier transform of an $r$-form is an $(n-r)$-form and this is also what follows naturally from the definition of the exponential of the form basis functions in Def.~\ref{def:diff_form:ft} based on the existing work in mathematical physics~\cite{DeWitt1992,Kalkman1993}.
  % COMMENT: The main issue with Hiptmair et al.'s definition seems to be that it is not clear what should happen to the basis functions \dd x^i under the Fourier transform. Also in Duistermaat's book that's rather unclear how this is supposed to work.
 \end{remark}
%
%\begin{remark}
%  From the perspective of microlocal analysis and quantization theory, it makes sense to consider $\xi$ as a variable in the fibers of the co-tantent bundle $T^* \R^n$.
%  We can then identity $\dd x^i$ with $\partial / \partial \xi^i$ so that a frequency vector describes momentum.
%  This is also the usual interpretation of the wave vector of plane waves found in the physics literature.
%\end{remark}

\paragraph{Wedge Product}
In the scalar case, multiplication in space becomes convolution in the frequency domain.
The next proposition, which also follows~\cite{Kalkman1993}, shows that this carries over to the wedge product.

\begin{proposition}
  \label{prop:exterior:convolution}
  Let $\alpha \in \Omega^{\fdeg}(\R^n)$ and $\beta \in \Omega^l(\R^n)$ and let the convolution $\ast : \Omega^{\fdeg} \times \Omega^l \to \Omega^{\fdeg+l-n}$ of differential forms be
  \begin{align*}
    (\alpha \ast \beta)(\dd y) = \int_{\R^n} \alpha(\dd x) \wedge \beta( \dd y - \dd x )
  \end{align*}
  where $\alpha(\dd x)$ denotes the basis representation of $\alpha$ with respect to the differentials $\dd x^i$.
  Then the Fourier transform $\hat{\wedge}$ of the wedge product $\wedge$ is given by
  \begin{align*}
    \mathfrak{F}( \alpha \wedge \beta ) = \widehat{\alpha} \ \hat{\wedge} \ \widehat{\beta} = \widehat{\alpha} \ast \widehat{\beta} \ \in \widehat{\Omega}^{n-\fdeg-l}(\widehat{\R}^n) .
  \end{align*}
  Conversely, the inverse Fourier transform $\check{\wedge}$ of the wedge product in frequency space is
  \begin{align*}
    \mathfrak{F}^{-1}( \widehat{\alpha} \wedge \widehat{\beta} ) = \alpha \ \check{\wedge} \ \beta = \alpha \ast \beta \ \in \Omega^{\fdeg+l-n}(\R^n) .
  \end{align*}
  Hence the following diagram holds
  \begin{center}
    \begin{tikzcd}
      \Omega^{\fdeg+l-n} \arrow[d, "\displaystyle \mathfrak{F}"]
      &[15pt] \Omega^{\fdeg} \times \Omega^l \arrow[l, "\displaystyle \ast"] \arrow[r, "\displaystyle \wedge"] \arrow[d, "\displaystyle \mathfrak{F}"]
      &[15pt] \Omega^{\fdeg+l} \arrow[d, "\displaystyle \mathfrak{F}"]
      \\[25pt]
      \widehat{\Omega}^{2n-\fdeg-l}
      &[15pt] \widehat{\Omega}^{n-\fdeg} \times \widehat{\Omega}^{n-l} \arrow[l, "\displaystyle \wedge = \hat{\ast}"] \arrow[r, "\displaystyle \hat{\wedge} = \ast"]
      &[15pt] \widehat{\Omega}^{n-\fdeg-l}
    \end{tikzcd}
  \end{center}
\end{proposition}

\begin{proof}
  This can be proved by a direct calculation. See again~\cite{Kalkman1993}.
\end{proof}

\begin{remark}
  It follows from the above definition and the usual properties of the wedge product that the convolution vanishes unless $r + l \geq n$.
  For $r + l < n$ one is interested in the convolution over a sub-manifold and working with differential forms requires one to make this explicit (i.e. to first pullback the form to the sub-manifold and then perform the integration there).
\end{remark}

% COMMENT:
%The definition somewhat surprisingly suggests that k=2, l=2 for n=3 yields a nonzero result, namely k+l-n = 1. For this case one has for $\beta = \beta_3 \dd x^2 \wedge \dd x^3$ that
%\begin{align}
%  \beta( \dd y - \dd x) &= \beta( \dd y - \dd x)
%  \\[4pt]
%  &= \beta_3 \, (\dd y^2 - \dd x^2) \wedge (\dd y^3 - \dd x^3)
%  \\[4pt]
%  &= \beta_3 \, \big(\dd y^2 \wedge \dd y^3 - \dd y^2 \wedge \dd x^3 - \dd x^2 \wedge \dd y^3 + \dd x^2 \wedge \dd x^3 \big)
%\end{align}
%so that the integral with a $2$-form is non-zero since there are $1$-form terms.
%\red{But whee are the 1-form terms? See again the Kalkman paper}

\begin{example}
  Let $\alpha = \alpha_3 \, \dd x^1 \wedge \dd x^2 \in \Omega^2(\R^3)$ and $\beta = \beta_1 \, \dd x^1 + \beta_2 \, \dd x^2 + \beta_3 \, \dd x^3  \in \Omega^1(\R^3)$.
  Then
  \begin{align*}
    (\alpha \ast \beta)(\dd y) = \int_{\R^n} \alpha(\dd x) \wedge \beta( \dd y - \dd x )
  \end{align*}
  with $\beta( \dd y - \dd x ) = \sum_{i=1}^3 \beta_i( x - y) (\dd y^i - \dd x^i )$.
  Using linearity and the anti-symmetry of the wedge product we obtain
  \begin{align*}
    (\alpha \ast \beta)(\dd y)
    &= \int_{\R^n} \alpha_3(x) \, \dd x^1 \wedge \dd x^2 \wedge \beta_3(x-y) \, (\dd y^3 - \dd x^3 )
    \\[5pt]
    &= \int_{\R^n} \alpha_3(x) \, \beta_3(x-y) \, \dd x^1 \wedge \dd x^2 \wedge \dd x^3
    \\[5pt]
    &= \alpha_3 \ast \beta_3
  \end{align*}
  where the convolution in the last line is those of scalar functions.
\end{example}

%%%%%%%%%%%%%%%%%%%%%%%%%%%%%%%%%%%%%%%%%%%%%%%%%%%
\paragraph{Hodge star and codifferential}
The following result characterizes the Fourier transform of the Hodge dual.

\begin{proposition}
  \label{prop:hodge_dual:fourier}
  Let $\alpha \in \Omega^{\fdeg}(\R^n)$. Then
  \begin{align*}
    \mathfrak{F}\big( \! \star \alpha \big)
    = \hat{\star} \, \widehat{\alpha}
%    = s_n \, \star \widehat{\alpha}
%    = - (-1)^{n(r+1)} \, i^n \star \widehat{\alpha}
   = -(-1)^{\lfloor r/2 \rfloor} \! \! \sum_{j_1 < \cdots < j_r} \hat{\alpha}_{j_1,\cdots,j_r} \,  s_n(\sigma) \Big( \! \star \dpp{}{\xi^{\sigma_{n-\fdeg+1}}} \wedge \cdots \wedge \dpp{}{\xi^{\sigma_{n}}} \Big)
  \end{align*}
  where the Hodge dual on the right hand side is the standard one for $\widehat{\Omega}(\widehat{\R}^n)$ and
  \begin{align*}
    s_n(\sigma) =
    \left\{
    \begin{array}{cc}
      (-1)^r & n = 2
      \\
      - \mathrm{sgn}(\sigma) & n = 3
    \end{array}
    \right. .
  \end{align*}
\end{proposition}

\begin{proof}
  The result can be verified with a coordinate calculation using Remark~\ref{remark:diff_form:ft:index_mapping} that already established that the Fourier transform has a structure very similar to the Hodge dual.
  % TODO: verify that on the Fourier side the sign for the Hodge dual flips in R^2, i.e. \star \partial / \partial \xi^1 = -  \partial / \partial \xi^2. Only then does the formula actually work.
\end{proof}

We hence have that up to a sign the Fourier transform of the Hodge dual of the differential form basis functions is the Hodge dual in frequency space of their Fourier transforms.

The above result for the Hodge star operator together with Proposition~\ref{prop:exterior_derivative:fourier} implies that the Fourier transform of the codifferential of a form $\alpha \in \Omega^r(\R^n)$ is
\begin{align*}
  \widehat{\delta \alpha} = (-1)^{n-1} \star \, \ip_{i \xi} \star \widehat{\alpha} .
\end{align*}
\commentCL{The signs that are introduced }

%%%%%%%%%%%%%%%%%%%%%%%%%%%%%%%%%%%
\paragraph{Plancherel theorem}
The Plancherel theorem is a central result for the scalar Fourier transform.
The next proposition establishes the analogue for differential forms.

\begin{proposition}[Plancherel theorem for differential forms]
  \label{prop:parseval:forms}
  With $\alpha , \beta \in \Omega^{\fdeg}(\R^n)$,
  \begin{align*}
    \big\langle \! \big\langle \alpha , \beta \big\rangle \! \big\rangle
    = \big\langle \! \big\langle \widehat{\alpha} , \widehat{\beta} \big\rangle \! \big\rangle .
  \end{align*}
\end{proposition}

\begin{proof}
  As in the classical case, the results follows from a direct calculation using
  \begin{align*}
   \big\langle \! \big\langle \alpha , \beta \big\rangle \! \big\rangle
   = \big\langle \! \big\langle \mathfrak{F}^{-1}(\widehat{\alpha}) \, , \, \mathfrak{F}^{-1}(\widehat{\beta}) \big\rangle \! \big\rangle
   = \int_{\R^n} \mathfrak{F}^{-1}(\widehat{\alpha}) \wedge \star \, \mathfrak{F}^{-1}(\widehat{\beta}) .
  \end{align*}
  We will show it for $\alpha = \alpha_3(x) \, \dd x^3$, $\beta = \beta_3(x) \, \dd x^3$. The required Fourier transforms are in this case
  \begin{align*}
    \widehat{\alpha} = -\widehat{\alpha}_3(x) \frac{\partial}{\partial \xi^1} \wedge \frac{\partial}{\partial \xi^2}
    \quad \quad \quad
    \widehat{\star \beta} = \widehat{\beta}_3(x) \frac{\partial}{\partial \xi^3}
    \quad \quad \quad
    \star \widehat{\beta} = -\widehat{\beta}_3(x) \frac{\partial}{\partial \xi^3} .
  \end{align*}
  Then
  \begin{align*}
    \big\langle \! \big\langle \alpha \, , \, \beta \big\rangle \! \big\rangle
    &= \int_{\R^3} \alpha_3(x) \, \dd x^3 \, \wedge \, \beta_{3}(x) \, \dd x^1 \wedge \dd x^2
    \\[5pt]
    &= \int_{\R^3} \left( \frac{1}{(2 \pi)^{3/2}} \int_{\widehat{\R}^3} \widehat{\alpha}_3(\xi) \, e^{i \langle \xi , x \rangle} \, \frac{\partial}{\partial \xi^1} \wedge \frac{\partial}{\partial \xi^2} \wedge \left( \frac{\partial}{\partial \xi^3} \dd x^3 \right) \! \right)
    \\
    & \quad \wedge_x \left(  \frac{-1}{(2 \pi)^{3/2}} \int_{\widehat{\R}^3} \widehat{\beta}_3^*(\eta)  \, e^{-i \langle \eta , x \rangle} \, \frac{\partial}{\partial \eta^3} \wedge \left( -\frac{\partial}{\partial \eta^1} \wedge \frac{\partial}{\partial \eta^2} \, \dd x^1 \wedge_{x} \dd x^2 \right) \! \right) .
    \intertext{With the graded, anti-symmetric multiplication we have}
    \big\langle \! \big\langle \alpha \, , \, \beta \big\rangle \! \big\rangle
    &= \int_{\R^3} \left( \frac{1}{(2 \pi)^{3/2}} \int_{\widehat{\R}^3} \widehat{\alpha}_3(\xi) \, e^{i \langle \xi , x \rangle} \, \frac{\partial}{\partial \xi^1} \wedge \frac{\partial}{\partial \xi^2} \wedge  \frac{\partial}{\partial \xi^3} \otimes \dd x^3 \! \right)
    \\
    & \quad \wedge_x \left(  \frac{1}{(2 \pi)^{3/2}} \int_{\widehat{\R}^3} \widehat{\beta}_3^*(\eta)  \, e^{-i \langle \eta , x \rangle} \, \frac{\partial}{\partial \eta^1} \wedge \frac{\partial}{\partial \eta^2} \wedge \frac{\partial}{\partial \eta^3} \otimes \dd x^1 \wedge_{x} \dd x^2 \! \right) .
  \end{align*}
  \commentCL{The sign comes from the reordering of the tensor product.}
  Using linearity and separating the terms with $x$-dependence we obtain
  % COMMENT: Integral representation of Dirac delta with our conventions: Kirkwood [2018], p. 333, proof  of Parseval's theorem in scalar case.
  \begin{align*}
    \big\langle \! \big\langle \alpha \, , \, \beta \big\rangle \! \big\rangle
    &=
    \frac{-1}{(2 \pi)^{3}} \int_{\widehat{\R}^3} \int_{\widehat{\R}^3} \left( \int_{\R^3}  e^{-i \langle \eta , x \rangle} \dd x^3 \wedge_x \, e^{i \langle \xi , x \rangle} \, \dd x^1 \wedge_{x} \dd x^2 \right)
    \\[4pt]
    & \quad \quad \quad \quad \times \left( \widehat{\alpha}_3(\xi)  \, \frac{\partial}{\partial \xi^1} \wedge \frac{\partial}{\partial \xi^2} \wedge \frac{\partial}{\partial \xi^3} \right) \left( \widehat{\beta}_3^*(\eta) \frac{\partial}{\partial \eta^1} \wedge \frac{\partial}{\partial \eta^2} \wedge \frac{\partial}{\partial \eta^3} \right) .
    \intertext{Evaluating the integral over $x$ yields}
       \big\langle \! \big\langle \alpha \, , \, \beta \big\rangle \! \big\rangle
       &= -\int_{\widehat{\R}^3} \int_{\widehat{\R}^3} \, \delta( \xi - \eta ) \, \left( \widehat{\alpha}_3(\xi)  \, \frac{\partial}{\partial \xi^1} \wedge \frac{\partial}{\partial \xi^2} \wedge \frac{\partial}{\partial \xi^3} \right)
       \\[4pt]
       & \qquad \qquad \qquad \qquad \qquad \qquad \qquad \times \left( \widehat{\beta}_3^*(\eta) \frac{\partial}{\partial \eta^1} \wedge \frac{\partial}{\partial \eta^2} \wedge \frac{\partial}{\partial \eta^3} \right)
   \end{align*}
%  The double wedge product $\wedge_{\xi,\eta}$ can be resolved using Eq.~\ref{eq:fourierforms:tensor_product:parity} and since $\partial / \partial \eta^3$ and $\partial / \partial \xi^3$ are $1$-forms, and hence have odd parity, this equals
%  \begin{align*}
%    \big\langle \! \big\langle \alpha \, , \, \beta \big\rangle \! \big\rangle
%    = -\int_{\widehat{\R}^3} \int_{\widehat{\R}^3} & \widehat{\alpha}_3(\xi) \, \widehat{\beta}_3^*(\eta) \, \delta( \xi - \eta ) \,
%   \frac{\partial}{\partial \xi^1} \wedge \frac{\partial}{\partial \xi^2} \wedge \frac{\partial}{\partial \xi^3}
%      \, \frac{\partial}{\partial \eta^1} \wedge \frac{\partial}{\partial \eta^1} \wedge  \frac{\partial}{\partial \eta^2} .
%      \nonumber
%  \end{align*}
  Finally, with the integration over $\eta$ we have
  \begin{align*}
  \big\langle \! \big\langle \alpha \, , \, \beta \big\rangle \! \big\rangle
    &= -\int_{\widehat{\R}^3} \widehat{\alpha}_3(\xi) \, i \widehat{\beta}_3^*(\xi) \, \frac{\partial}{\partial \xi^1} \wedge \frac{\partial}{\partial \xi^2} \wedge \frac{\partial}{\partial \xi^3}
    \\[7pt]
    &= \int_{\widehat{\R}^3} \underbrace{\left( \widehat{\alpha}_3(\xi) \frac{\partial}{\partial \xi^1} \wedge \frac{\partial}{\partial \xi^2} \right)}_{\displaystyle \widehat{\alpha}}
      \wedge \underbrace{\left( -\widehat{\beta}_3^*(\xi) \frac{\partial}{\partial \xi^3}\right)}_{\displaystyle \star \widehat{\beta}^*}
    \\[1pt]
    &= \big\langle \! \big\langle \widehat{\alpha} \, , \, \widehat{\beta} \big\rangle \! \big\rangle .
  \end{align*}
  The other cases follow by analogous calculations.
\end{proof}

%%%%%%%%%%%%%%%%%%%%%%%%%%%%%%%%%%
\paragraph{Hodge-Helmholtz decomposition}

The Hodge-Helmholtz decomposition provides a characterization of the domain and image of the exterior derivative.
It will be central to our formulation of the local spectral exterior calculus.

\begin{theorem}[Hodge-Helmholtz decomposition]
  \label{thm:hodge_helmholtz}
  Let $L_2(\Omega^{\fdeg},\R^n)$ be the space of differential forms of degree $\fdeg$ on $\R^n$ with finite $L_2$-norm. Then
  \begin{align}
  \label{eq:hodge_decomposition}
    L_2(\Omega^{\fdeg},\R^n) = L_2(\Omega_{\dd}^{\fdeg},\R^n) \oplus L_2(\Omega_{\delta}^{\fdeg},\R^n)
  \end{align}
  where $L_2(\Omega_{\dd}^{\fdeg},\R^n)$ is the $L_2$-space of exact forms and $L_2(\Omega_{\delta}^{\fdeg},\R^n)$ those of co-exact forms.
  Hence any $\omega \in L_2(\Omega^{\fdeg},\R^n)$ can be written as $\omega = \omega_{\dd} + \omega_{\delta} = \dd \alpha + \delta \beta$ where $\alpha \in L_2(\Omega_{\dd}^{\fdeg-1},\R^n)$ and $\beta \in L_2(\Omega_{\delta}^{\fdeg+1},\R^n)$.
  In frequency space,
  \begin{align}
  \label{eq:hodge_decomposition:frequency}
    L_2(\widehat{\Omega}^{n-\fdeg},\widehat{\R}^n) = L_2(\widehat{\Omega}_{\dd}^{n-\fdeg},\widehat{\R}^n) \oplus L_2(\widehat{\Omega}_{\delta}^{n-\fdeg},\widehat{\R}^n) ,
  \end{align}
  i.e. $\widehat{\omega} = \widehat{\omega}_{\dd} + \widehat{\omega}_{\delta}$ with
  \begin{subequations}
  \label{eq:hodge_decomposition:frequency:spherical}
  \begin{align}
    \widehat{\omega}_{\dd} = \sum_{j_1 < \cdots < j_{n-r}} (\widehat{\omega}_{\dd})_{j_1, \cdots, j_{n-r}} \, \frac{\partial}{\partial \hat{\theta}^{j_1}} \wedge \cdots \wedge \frac{\partial}{\partial \hat{\theta}^{j_{n-r}}}
    \\
    \widehat{\omega}_{\delta} = \sum_{j_1 < \cdots < j_{n-r}} (\widehat{\omega}_{\delta})_{j_1, \cdots ,j_{n-r}} \, \frac{\partial}{\partial \hat{\theta}^{j_1}} \wedge \cdots \wedge \frac{\partial}{\partial \hat{\theta}^{j_{n-r-1}}} \wedge \frac{\partial}{\partial \hat{r}}
  \end{align}
  \end{subequations}
  where $\{ \partial / \partial \hat{\theta}^1 , \cdots , \partial / \partial \hat{\theta}^{n-1} , \partial / \partial \hat{r} \}$ is the basis for the tangent space in spherical coordinates in frequency space.
\end{theorem}

\begin{proof}
  The spatial part of the theorem is a standard result, see e.g.~\cite[Thm. 8.5.1]{Marsden2004}.
  The harmonic forms, which usually form the third part of the Hodge decomposition, are not present in our case since $L_2(\Omega^{\fdeg},\R^n)$ does not contain any nonzero polynomials, cf.~\cite{Troyanov2009}.

  Eq.~\ref{eq:hodge_decomposition:frequency} follows from the linearity of the Fourier transform.
  In spherical coordinates $\xi = \vert \xi \vert \dd \hat{r}$, i.e. it is given by the radial basis vector of $T \widehat{\R}^n$.
  By $\hat{\dd} = \ip_{i \xi}$ we have that $\hat{\dd} = \ip_{i \xi} \widehat{\omega}$ vanishes if and only if the coordinate representation of $\widehat{\omega}$ does not contain $\partial / \partial \hat{r}$ so that $\xi = \dd \hat{r}$ can not be paired.
  This implies Eq.~\ref{eq:hodge_decomposition:frequency:spherical}.
\end{proof}

Theorem~\ref{thm:hodge_helmholtz} states that an exact differential form is one those Fourier transform is tangential to the frequency sphere $\widehat{S}^{n-1}$ while for a co-exact one there is a radial component in the direction of $\partial / \partial \hat{r}$.
This is particularly simple in $\R^2$ where exact and co-exact forms are exactly those tangential and radial to the circle in the Fourier domain, respectively.
The geometric characterization of exact and co-exact forms in the Fourier domain in Theorem~\ref{thm:hodge_helmholtz} will play an important role in the following.

%%%%%%%%%%%%%%%%%%%%%%%%%%%%%%%%%%
\paragraph{Homogeneous Sobolev spaces for differential forms}

With the Hodge-Helmholtz decomposition and the Plancherel theorem we can introduce homogeneous Sobolev spaces for differential forms.
These will provide the functional analytic setting of our local spectral exterior calculus.
As we will clarify shortly, it is natural to restrict the definition to co-exact forms.

\begin{definition}
  \label{def:homo_sobolev:forms}
  The \emph{homogeneous Sobolev space $\dot{L}_2^1(\Omega_{\delta}^{\fdeg},\R^n)$} is
\begin{subequations}
\begin{align}
  \label{eq:homo_sobolev:forms:def}
  \dot{L}_2^1(\Omega_{\delta}^{\fdeg} , \R^n ) = \Big\{ \alpha \in \Omega_{\delta}^{\fdeg}( \R^n ) \ \Big\vert \ \Vert \alpha \Vert_{\dot{L}_2^1}^2 = \big\langle \! \big\langle \alpha \, , \, \alpha \big\rangle \! \big\rangle_{\dot{L}_2^1} < \infty \Big\} .
\end{align}
where the $\dot{L}_2^1(\Omega_{\delta}^{\fdeg} , \R^n)$-inner product is given by
\begin{align}
  \label{eq:homo_sobolev:forms:inner}
  \big\langle \! \big\langle \alpha \, , \, \beta \big\rangle \! \big\rangle_{\dot{L}_2^1}
  = \big\langle \! \big\langle \dd \alpha \, , \, \dd \beta \big\rangle \! \big\rangle_{{L}_2}
  = \int_{\widehat{\R}^n} \vert \xi \vert^2 \, \widehat{\alpha} \wedge \star \widehat{\beta}^* .
\end{align}
\end{subequations}
\end{definition}
%As in the scalar case, $\dot{L}_2^1(\R^n , \Omega^{\fdeg})$ is a Hilbert space.

The frequency definition on the right hand side of Eq.~\ref{eq:homo_sobolev:forms:inner}, where no interior products $\hat{\dd} = \ip_{i \xi}$ appears, can be established using a straightforward coordinate calculation.

In the scalar case $\dot{L}_2^1(\R^n)$, the definition of homogeneous Sobolev spaces typically uses co-sets $[f]_0$ of functions modulo constants, i.e. $0$ order polynomials, since this is required for the spaces to be Hilbert, see~\ref{sec:preliminaries:homoegeneous_sobolev}.
For a compact domain $U \subset \R^n$, the Hodge-Helmholtz decomposition for functions is $\Omega^0(U) = \mathcal{H}(U) \oplus \Omega_{\delta}^0(U)$, where $\mathcal{H}(U)$ are the harmonic ones satisfying $\Delta f = 0$, and for a star-shaped domain this are exactly the constants.
The classical definition hence works with the co-exact $0$-forms modulo harmonic ones.
In Def.~\ref{def:homo_sobolev:forms} we work directly with co-exact differentials, which, as the following result shows, also ensures that $\dot{L}_2^1(\R^n , \Omega^{\fdeg})$ is Hilbert.

\begin{proposition}
  \label{prop:homo_sobolev:hilbert}
  For $n=2,3$, the homogeneous Sobolev space $\dot{L}_2^1(\Omega_{\delta}^{\fdeg},\R^n)$ is a Hilbert space.
\end{proposition}

\begin{proof}
  With the definition of the Hodge dual in terms of the metric tensor~\cite[Ch. 7.2]{Marsden2004}, Eq.~\ref{eq:homo_sobolev:forms:inner} can be written as
  \begin{align*}
    \int_{\widehat{\R}^n} \vert \xi \vert^2 \, \widehat{\alpha} \wedge \star \widehat{\beta}^*
    &= \sum_{j=1}^n \int_{\widehat{\R}^n} \vert \xi \vert^2 \, \widehat{\alpha}_j( \xi ) \, \widehat{\beta}^j(\xi) \, \dd \xi^n
  \end{align*}
  where $\dd \xi^n$ is the canonical volume form on $\R^n$.
  But thus the scalar result~\cite[Ch. II.6]{Galdi2011} applies to each coordinate function.
%  Note that in general one has to consider co-sets of polynomials in the definition of homogeneous Sobolev space for these to be Hilbert. This is not necessary in our cases since these are excluded from the outset.
  % TODO: Are they excluded by Eq.~\ref{eq:homo_sobolev:forms:def}. Is there something for classical vector valued case that we can cite that for L_2( R^n , R^{\fdeg} ) is suffices that each component is in L_2(\R^n)
\end{proof}
% COMMENT: The proof is certainly not very elegant or good. One could just avoid to state it as a proposition.}

%As we will see in the following, cleanly separating exact and co-exact forms provides many advantages.

Def.~\ref{def:homo_sobolev:forms} is only for co-exact differential forms.
It can be combined with the usual $L_2$-space for exact ones, i.e. $L_2(\Omega_{\dd}^{\fdeg} , \R^n )$, and in light of the de Rahm complex where exact forms are the image of co-exact ones this is the natural complement to $\dot{L}_2^1(\Omega_{\delta}^{\fdeg} , \R^n )$.
For conciseness we will write $\dot{L}_2^0(\Omega_{\dd}^{\fdeg} , \R^n ) \cong L_2(\Omega_{\dd}^{\fdeg} , \R^n )$ in the following.
Working with $L_2(\Omega_{\dd}^{\fdeg} , \R^n )$ and $\dot{L}_2^1(\Omega_{\delta}^{\fdeg} , \R^n )$ is also compatible with the classical spaces $H(\mathrm{curl},\R^n)$ and $H(\mathrm{div},\R^n)$, as the following remark shows.

\begin{remark}
  \label{remark:hcurl}
  In numerical analysis, the spaces $H(\mathrm{curl},\R^n)$ and $H(\mathrm{div},\R^n)$ of $L_2$-vector fields whose curl respectively divergence are also in $L_2$ provide typically the functional analytic setting for the discretization of vector-valued partial differential equations, e.g.~\cite{Nedelec1980,Girault1986}, and these are also used in finite element-based discretizations of exterior calculus~\cite{Hiptmair2002,Arnold2018}.

  With the musical isomorphism to relate vector fields and $1$-forms and using the Hodge-Helmholtz decomposition, $H(\mathrm{curl},\R^n)$ can be written as
  \begin{align*}
    H(\mathrm{curl},\R^n) = \Big\{ u \in \mathfrak{X}(\R^n) \, \Big\vert \, u_{\dd}^{\flat} \in L_2(\Omega_{\dd}^1,\R^{n}) , \, u_{\delta}^{\flat} \in \dot{L}_2^1(\Omega_{\delta}^1,\R^{n})   \Big\}
  \end{align*}
  where no control of the derivative of the exact part $u_{\dd}^{\flat}$ is necessary since $\dd  u_{\dd}^{\flat} = 0$ and the $L_2$-condition for the co-exact one is automatically satisfied.
  \commentCL{since the lose regularity by taking the derivative in form of the $\vert \xi \vert$ factor in the Fourier domain}
  Thus,
  \begin{align*}
    H(\mathrm{curl},\R^n) &= L_2(\Omega_{\dd}^1,\R^{n})^{\sharp} \, \oplus \, \dot{L}_2^1(\Omega_{\delta}^1,\R^{n})^{\sharp}
  \end{align*}
  with the musical isomorphism on the right hand side understood element-wise.
  Analogously we have (in $\R^2$ the divergence corresponds to $\dd \star u^{\flat}$)
  \begin{align*}
    H(\mathrm{div},\R^2) &= \dot{L}_2^1(\star \Omega_{\dd}^1,\R^{2})^{\sharp} \, \oplus \, L_2(\Omega_{\delta}^1,\R^{2})^{\sharp}
    \\[5pt]
    H(\mathrm{div},\R^3) &= L_2(\Omega_{\dd}^2,\R^{3})^{\sharp} \, \oplus \, \dot{L}_2^1(\Omega_{\delta}^2,\R^{3})^{\sharp} .
  \end{align*}
%  where also the Hodge star is understood element-wise.
  We will return to the last equations in the next section.
\end{remark}

Def.~\ref{def:homo_sobolev:forms} can be generalized to homogeneous Sobolev spaces of arbitrary order $s \in \R$, cf.~\cite{Hiptmair2012b}, but $\dot{L}_2^1(\Omega_{\delta}^{\fdeg},\R^n)$ suffices for our purposes and we hence leave this to future work.

%%%%%%%%%%%%%%%%%%%%%%%%%%%%%%%%%%%%%%%%%%%%%%%%%%%
\paragraph{Stokes' theorem}

An important result in the exterior calculus is Stokes' theorem, which, up to the musical isomorphism, includes the classical theorems of Green, Gauss, and Stokes in vector calculus as special cases.

Let $\M$ be an $\fdeg$-dimensional, smooth, compact sub-manifold of $\R^n$ with smooth boundary $\partial \M$ and let $\alpha \in \Omega^{\fdeg-1}(\R^n)$.
Then Stokes' theorem states that~\cite[Theorem 8.2.8]{Marsden2004}
\begin{align}
  \label{eq:stokes_thm}
  \int_{\partial \M} i^* \alpha = \int_{\M} \dd \alpha
\end{align}
where $i : \partial \M \to \M$ is the inclusion map (cf.~\cite[Ch. 2.7]{Frankel2011}).

The above formulation of Stokes' theorem is not well suited for the Fourier theory of differential forms introduced in the foregoing.
We hence rewrite it in a more amenable one using the characteristic functions $\chi_{\M}$ and $\chi_{\partial \M}$ of $\M$ and the boundary $\partial \M$, respectively. \commentCL{One should formally also use an inclusion map $i : \M \to \R^3$
The smooth extension of $\alpha$ to $\R^n$: use the argument here \url{https://math.stackexchange.com/questions/871038/differential-forms-on-submanifolds}; Thomas Richter also confirmed that this shouldn't be an issue.}
Since the manifold $\M$ is smooth, we can always choose local coordinates $\tilde{x}_i$ such that $\M$ is locally aligned with the first $r$ coordinates $\tilde{x}_1, \cdots , \tilde{x}_r$ of $\R^n$.
This carries over to $\dd \alpha$, which has then the representation $\alpha = \tilde{\alpha}(\tilde{x}) \, \dd \tilde{x}^1 \wedge \cdots \wedge \dd \tilde{x}^r$ for some coordinate function $\tilde{\alpha}(\tilde{x})$.
The form can be extended along the normal bundle of $\M$ to a smooth, compactly supported test form $\dd \alpha \in \mathcal{S}(\Omega^{r},\R^n)$ in the Schwartz space $\mathcal{S}(\Omega^{r},\R^n)$ supported in a local neighborhood of $\M$ in the ambient space $\R^n$ (e.g.~\cite[Lemma 2.27]{Lee2012}); by abuse of notation we will not distinguish $\dd \alpha$ and its smooth extension.
This $r$-form $\dd \alpha$ can then be completed to a volume form (that can be integrated) by $\wedge$-multiplying it with the characteristic differential form $\chi_{_{\M}} \in \mathcal{S}'(\Omega^{n-\fdeg},\R^n)$ of $\M$, i.e. a distribution-valued $(n-\fdeg)$-form whose support is $\M$.
In coordinates, it is given by the partial contraction
\begin{align}
  \label{eq:stokes_thm:characteristic:form}
  \chi_{_{_\M}} = \dd x^1 \wedge \cdots \wedge \dd x^n \left( \frac{\partial}{\partial u^1} , \, \cdots , \frac{\partial}{\partial u^{r}} , \, \cdots \right) ,
\end{align}
where the $\partial / \partial u^i \in \mathfrak{X}(\R^n)$ span the tangent space $T \M$ and have support only on $\M$ , which reflects the distributional character of $\chi_{_{_\M}}$, and $\dd x^1 \wedge \cdots \wedge \dd x^n$ is the canonical volume form on $\R^n$.

\begin{example}
  \label{ex:stokes:characteristic:circle}
  Let $\partial \M = S^1 \subset \R^2$ be the unit circle and $\alpha \in \Omega^1(S^1)$ with polar coordinate representation $\alpha = \alpha(\theta,\phi) \, \dd \theta$.
  Then the $1$-form in $\Omega^1(\R^2)$ obtained by a smooth extension in the normal direction is $\alpha =  \tilde{\alpha}(\theta,\phi) \, \dd \theta \in \Omega^1(\R^2)$.
  The tangent vector field $\vec{u}$ of $S^1$ embedded into $\R^2$ is $\vec{u} = \delta_{S^1}(x) \, \partial / \partial \theta \in \mathfrak{X}(\R^2)$ where the Dirac distribution $\delta_{S^1}(x) \in \mathcal{S}'(\R^n)$ is defined by
\begin{align}
  \label{eq:Dirac:circle}
  \langle \delta_{S^1} , \varphi \rangle = \int_{S^1} \varphi
\end{align}
for all test functions $\varphi \in \mathcal{S}(\R^2)$, cf.~\cite{Grohs2012}.
The characteristic $1$-form completing $\alpha$ is by Eq.~\ref{eq:stokes_thm:characteristic:form} thus
  \begin{align*}
  \chi_{_{S^1}} = \dd\theta \wedge \dd r \Big( \delta_{S^1}(x) \frac{\partial}{\partial \theta} \Big) = \delta_{S^1}(x) \, \dd r  .
  \end{align*}
  The integral of $\alpha$ over $S^1$ can hence also be written as
  \begin{align*}
    \int_{S^1} \alpha
    = \int_{\R^2} \tilde{\alpha}(x) \, \dd \theta \wedge \delta_{S^1}(x) \, \dd r
    = \int_{\R^2} \tilde{\alpha}(x) \, \delta_{S^1}(x) \, \dd \theta \wedge \dd r .
  \end{align*}
%Note also that $\chi_{_{S^1}}$ spans the normal bundle to $S^1$ in $\R^2$, as one would expect from the divergence theorem.
\end{example}

With the above construction, the right hand side of Stokes' theorem in Eq.~\ref{eq:stokes_thm} becomes
\begin{align*}
  \int_{\M} \dd \alpha = \int_{\R^n} \dd \alpha \wedge \chi_{_{\M}}
\end{align*}
where on the right we have a weak pairing between $\mathcal{S}(\Omega^{r},\R^n)$ and $\mathcal{S}'(\Omega^{n-r},\R^n)$.
To apply the formulation also to the left hand side of Stokes' theorem we use the change of variables theorem~\cite[Theorem 8.1.7]{Marsden2004},
\begin{align*}
  \int_{\partial \M} i^* \alpha = \int_{i(\partial \M)} \alpha
\end{align*}
where $i(\partial \M)$ is the sub-manifold in $\R^n$ spanned by the boundary $\partial \M$.
The argument for $\M$ then carries over since both $\M$ and $\partial \M$ are sub-manifolds of $\R^n$.
With this, we obtain the following form for Stokes' theorem
\begin{align}
  \label{eq:stokes_thm:characteristic}
  \int_{\R^n} \alpha \wedge \chi_{_{\partial \M}} = \int_{\R^n} \dd \alpha \wedge \chi_{_{\M}} .
\end{align}
that provides a weak formulation of the classical one in Eq.~\ref{eq:stokes_thm}.

\begin{example}
  \label{ex:chi_circle}
  Let $\fdeg = n$.
  The characteristic form $\chi_{_{\M}} \in \mathcal{S}'(\Omega^{n-r},\R^n)$ completing $\dd \alpha  \in \mathcal{S}(\Omega^{r},\R^n)$ to a volume form is then a $0$-form, i.e. the usual characteristic function of $\M$.
In this case it is well known from the classical divergence theorem that $\chi_{_{\partial \M}}$ is given by the weak gradient of $\chi_{_{\M}}$, cf.~\cite[p. 73]{McLean2000}, i.e.
\begin{align*}
  \chi_{\partial \M} = \delta_{\partial \M}(x) \, n_i \, \dd \tilde{x}^i
\end{align*}
where the $n_i$ are the components of the normal of $\partial \M$ and $\delta_{_{\partial \M}}(x)$ is the Dirac distribution defined analogously to Eq.~\ref{eq:Dirac:circle}.
In other words, when $\fdeg = n$ then $\chi_{_{\partial \M}}$ is nothing but the wavefront set $\mathrm{WF}(\chi_{_{\M}})$ of $\chi_{_{\M}}$~\cite[Ch. VI]{Taylor1981}.
%As one would expect from the general theory, this is indeed a section of $T^* \R^n$.
% we quote here from~\cite{Brouder2014}
We will return to the connection with the wavefront set in Sec.~\ref{sec:psiforms}.
\end{example}

The form of Stokes' theorem in Eq.~\ref{eq:stokes_thm:characteristic} can also be written using the $L_2$-pairing for differential forms, i.e. using the Hodge duals of $\chi_{_{\partial \M}}$ and $\chi_{_{\M}}$.
It then takes the form
\begin{align}
  \label{eq:stokes_thm:characteristic:inner_product}
  \big\langle \! \big\langle \alpha \, , \, \chi_{_{T \partial \M}} \big\rangle \! \big\rangle
  = \big\langle \! \big\langle \dd \alpha \, , \, \chi_{_{T \M}} \big\rangle \! \big\rangle .
\end{align}
It follows from Eq.~\ref{eq:stokes_thm:characteristic:form} that $\star \chi_{_{\partial \M}} = \chi_{_{T \partial \M}}$ is a local representation for the tangent space of $\partial \M$ and similarly $\star \chi_{_{\M}} = \chi_{_{T \M}}$.
In $\R^3$, for example, let the local tangent space of a $2$-surface be aligned with the $x_1$-$x_2$ plane.
Then $\chi_{_{\partial \M}} = \delta(x_3) \, \dd x^3$ and $\chi_{_{T \partial \M}} = \delta(x_3) \, \dd x^1 \wedge \dd x^2$.

Comparing the standard form of Stokes' theorem in Eq.~\ref{eq:stokes_thm} with the one in Eq.~\ref{eq:stokes_thm:characteristic:inner_product} one sees that both provide conceptually a very similar geometric description: the left hand side is the pairing of $\alpha$ with the tangent space $T \partial \M$ of the boundary and the right hand side those of $\dd \alpha$ with $T \M$.
However, in contrast to Eq.~\ref{eq:stokes_thm}, Eq.~\ref{eq:stokes_thm:characteristic} can be transferred to a frequency description using the Plancherel theorem.
This yields
\begin{align}
  \label{eq:stokes_thm:fourier}
  \int_{\widehat{\R}^n} \widehat{\alpha} \wedge \widehat{\chi}_{_{\partial \M}}
%  = \int_{\widehat{\R}} \widehat{\dd \alpha} \wedge \widehat{\chi}_{_{\M}}
  = \int_{\widehat{\R}^n} \hat{\dd} \widehat{\alpha} \wedge \widehat{\chi}_{_{\M}} .
\end{align}
To obtain some insight into this formulation of Stokes' theorem, we will consider a concrete example.

\begin{example}
  \label{eq:stokes:decay:circle:fourier}
  Let $\M = B^2$, $\partial \M = S^1$, and $\alpha = \alpha_{\theta}(r) \, \dd\theta \in \Omega^1(B^2)$.
  Thus $\dd \alpha = - (\partial \alpha / \partial r) \, \dd \theta \wedge \dd r$.
  From Example~\ref{ex:stokes:characteristic:circle} we already know that the completion of $\alpha$ to a volume form is given by $\chi_{_{\partial \M}} = \chi_{_{S^1}} = \delta_{S^1}(r) \, \dd r$.
  The completion of $\dd \alpha$ is the $0$-form $\chi_{_{\M}} = \chi_{_{B^2}}(r)$.
  Stokes' theorem in the form in Eq.~\ref{eq:stokes_thm:characteristic} then becomes
  \begin{align*}
    \int_{\R^2} \alpha_{\theta}(r) \, \dd\theta \wedge \delta_{S^1}(r) \, \dd r
    &= \int_{\R^2} - \frac{\partial \alpha_{\theta}(r)}{\partial r} \, \dd \theta \wedge \dd r \wedge \chi_{_{B^2}}(r)
    \intertext{and re-arranging terms yields}
    \int_{\R^2} \alpha_{\theta}(r) \, \delta_{S^1}(r)\, \dd\theta \wedge \, \dd r
    &= \int_{\R^2} - \frac{\partial \alpha_{\theta}(r)}{\partial r} \, \chi_{_{B^2}}(r) \, \dd \theta \wedge \dd r .
  \end{align*}
  The equality can now easily be deduced using integration by parts.

  To write the last equations in the frequency domain as in Eq.~\ref{eq:stokes_thm:fourier}, the following Fourier transforms are required
  \begin{align*}
    \widehat{\alpha} &= \widehat{\alpha}(\hat{r}) \frac{\partial}{\partial \hat{r}}
    \\[4pt]
    \widehat{\chi}_{_{\partial \M}}
    &= \widehat{\delta}_{_{S^1}}(\xi) \frac{\partial}{\partial \hat{\theta}}
    = -J_1(\hat{r}) \frac{\partial}{\partial \hat{\theta}}
    \\[4pt]
    \widehat{\dd \alpha} &= \widehat{\frac{\partial \alpha}{\partial r}}(\hat{r}) = \widehat{\alpha}(\hat{r}) \, i \hat{r}
    \\[4pt]
    \widehat{\chi}_{_{\M}} &= \widehat{\delta}_{_{B^2}}(\xi) \, \frac{\partial}{\partial \hat{\theta}} \wedge \frac{\partial}{\partial \hat{r}} = \frac{J_1(\hat{r})}{\hat{r}} \frac{\partial}{\partial \hat{\theta}} \wedge \frac{\partial}{\partial \hat{r}}
  \end{align*}
  where $(\hat{\theta},\hat{r})$ are polar coordinates in frequency space,
  $J_1(\cdot)$ is the (cylindrical) Bessel function of order $1$, and for $\widehat{\dd \alpha}$ we used Proposition~\ref{prop:exterior_derivative:fourier} and that $\xi = \hat{r} \dd \hat{r}$ in polar coordinates.
  Inserting the Fourier transforms into Eq.~\ref{eq:stokes_thm:fourier} we obtain
  \begin{align*}
    -\int_{\widehat{\R}^2} \widehat{\alpha}(\hat{r}) \frac{\partial}{\partial \hat{r}} \wedge J_1(\hat{r}) \frac{\partial}{\partial \hat{\theta}}
    &= \int_{\widehat{\R}^2} \widehat{\alpha}(\hat{r}) \, \hat{r}  \wedge \frac{J_1(\hat{r})}{\hat{r}} \frac{\partial}{\partial \hat{\theta}} \wedge \frac{\partial}{\partial \hat{r}} .
    \intertext{Rearranging the terms yields}
   \int_{\widehat{\R}^2} \widehat{\alpha}( \hat{r}) \, J_1(\hat{r}) \, \frac{\partial}{\partial \hat{\theta}} \wedge \frac{\partial}{\partial \hat{r}}
    &= \int_{\widehat{\R}^2} \widehat{\alpha}(\hat{r}) \, \hat{r}  \frac{J_1(\hat{r})}{\hat{r}} \frac{\partial}{\partial \hat{\theta}} \wedge \frac{\partial}{\partial \hat{r}}
    \end{align*}
  and cancelling the $\hat{r}$ factor on the right hand side shows that the equality holds.
\end{example}

The above example provides insight into the mechanics behind Stokes' theorem in the Fourier domain:
The linear growth by $\vert \xi \vert$ in $\widehat{\dd \alpha}$ compared to $\widehat{\alpha}$, which is introduced by the exterior derivative, is compensated by the difference in the decay rates of $\widehat{\chi}_{_{\partial \M}}$ and $\widehat{\chi}_{_{\M}}$.
One hence has an interesting interweaving of functional analytic and geometric aspects in Eq.~\ref{eq:stokes_thm:characteristic} and Eq.~\ref{eq:stokes_thm:fourier} that is not apparent in the classical form of Stokes' theorem.

\begin{figure}
  \centering
  \includegraphics[width=1.\textwidth]{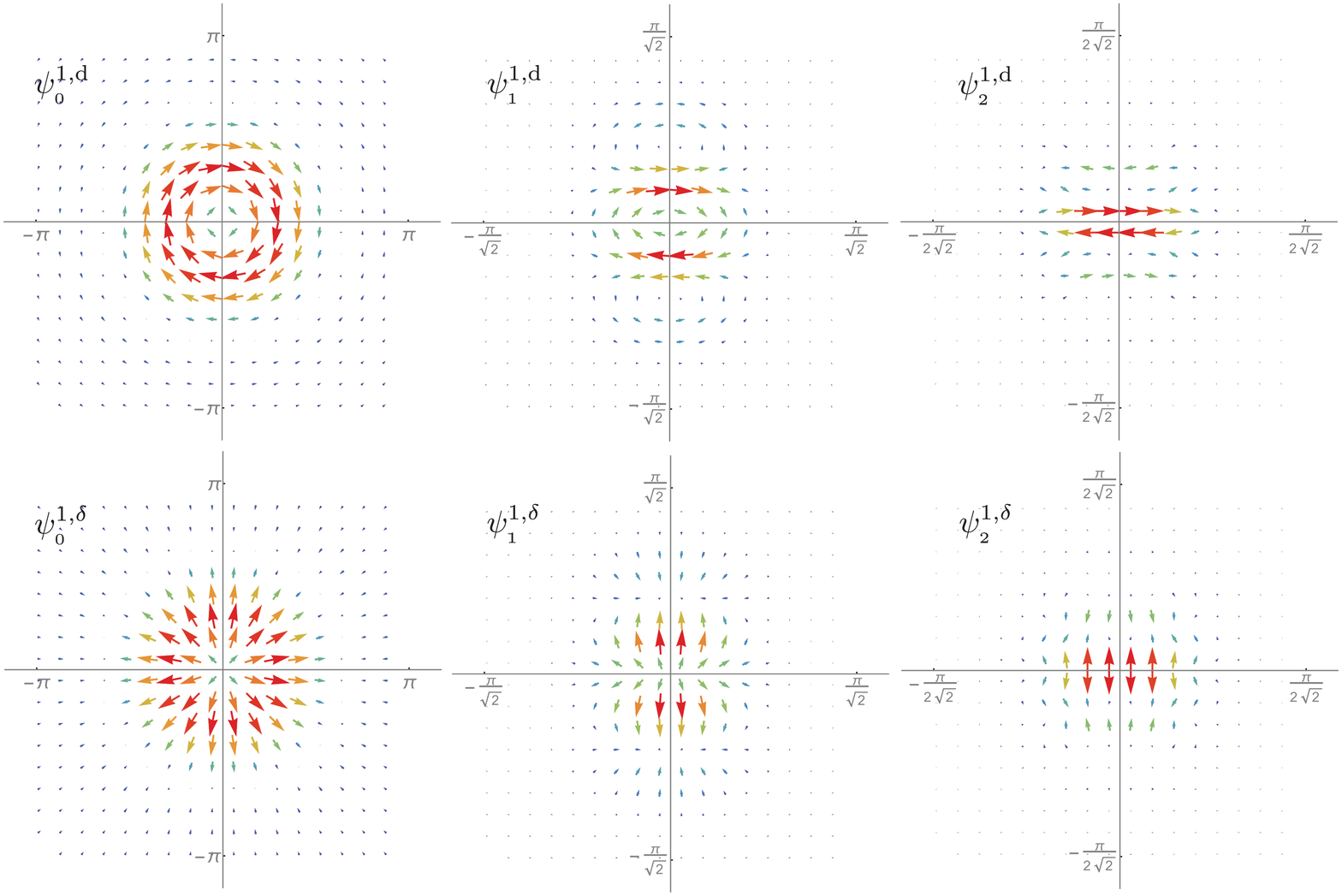}
  \caption{Examples for differential $1$-form mother wavelets $\psi_{j}^{1,\nu} \equiv \psi_{j,0,0}^{1,\nu,2}$ for different levels with exact forms $\psi_{j,0,0}^{1,\dd,2}$ in the top row and co-exact ones $\psi_{j,0,0}^{1,\delta,2}$ in the bottom row (visualized as vector fields). The anisotropic functions exist at different directions and model discontinuities aligned with their orientation (here along the $x_1$-axis), cf. Sec.~\ref{sec:psiforms:stokes} and Example~\ref{ex:laplace_de_rahm:cavity}. }
  \label{fig:psi_delta_22}
\end{figure}

\section[Psiec: A Local Spectral Exterior Calculus]{$\Psiec$: A Local Spectral Exterior Calculus}
\label{sec:psiforms}

In the last section, we studied the Fourier transform of differential forms and saw that the symbol of the exterior derivative is a radial operator, i.e. it acts on the radial part of a differential form's Fourier transform.
This structure aligned with spherical coordinates in the frequency domain is also reflected in the Hodge-Helmholtz decomposition, with an exact differential form being one those Fourier transform has no radial component.
% of the exterior calculus in the frequency domain
Motivated by this, we will define the polar differential form wavelets of $\Psiec$ in spherical coordinates in Fourier space.
% with the  $\partial / \partial \hat{\theta}$, $\partial / \partial \hat{\phi}$, $\partial / \partial \hat{r}$.
%Compatible wavelet windows are
Compatible window functions providing the wavelet's space-frequency localization are given by the polar wavelets of Sec.~\ref{sec:polar_wavelets} and these ensure the admissibility of the frames, allow for flexible angular localization, and yield differential form wavelets that have closed form expressions in the spatial domain.
The functional analytic setting for our construction is motivated by the classical spaces  $H(\mathrm{curl},\R^n)$ and $H(\mathrm{div},\R^n)$ but we will from the outset distinguish between exact and co-exact differential forms, since this provides us with control over the kernel and image of the exterior derivative and the codifferential, cf. Remark~\ref{remark:hcurl}.
%The homogeneous Sobolev spaces introduced in Def.~\ref{def:homo_sobolev:forms} will hence play an important role.
We will hence work with the homogeneous Sobolev spaces $\dot{L}_2^{\vert \nu \vert}(\Omega_{\nu}^{\fdeg},\R^n)$ introduced in Def.~\ref{def:homo_sobolev:forms}.

In the following, we will formally define differential forms wavelet and study their properties. Stokes' theorem and other results related to differential forms are considered in the sequel.
To not unnecessarily clutter the presentation, we collect proofs in Sec.~\ref{sec:psiforms:proofs}.

%The results of the previous section show that the exterior derivative acts in the Fourier domain only on the radial component of a differential form and that the co-chain structure of the de Rahm complex becomes there hence a pointwise condition on the ``direction'' of the form basis functions: a form is exact when it has no radial component.
%Motivated by these insights we define our wavelet differential forms in the frequency domain using differential form basis functions in spherical coordinates, i.e. $\partial / \partial \hat{\theta}$, $\partial / \partial \hat{\phi}$, $\partial / \partial \hat{r}$, and using polar wavelets, i.e. with window functions separable in polar coordinates, since these respect the structure of the exterior calculus in the Fourier domain, allow for flexible angular localization, and ensure that the wavelet differential forms have closed form expressions in the spatial domain.

%%%%%%%%%%%%%%%%%%%%%%%%%%%%%%%%%%%%%%%%%%%%%%%%%%%%%%%%%%%%%%%%%%%%%%%%%%%%%%%%%%%%%%%%%%%%%%%%%%%%%%%%%%%%%%%%%%%%%%
\subsection{Polar Differential Form Wavelets}

\begin{figure}[t]
  \centering
  \includegraphics[width=0.5275\textwidth]{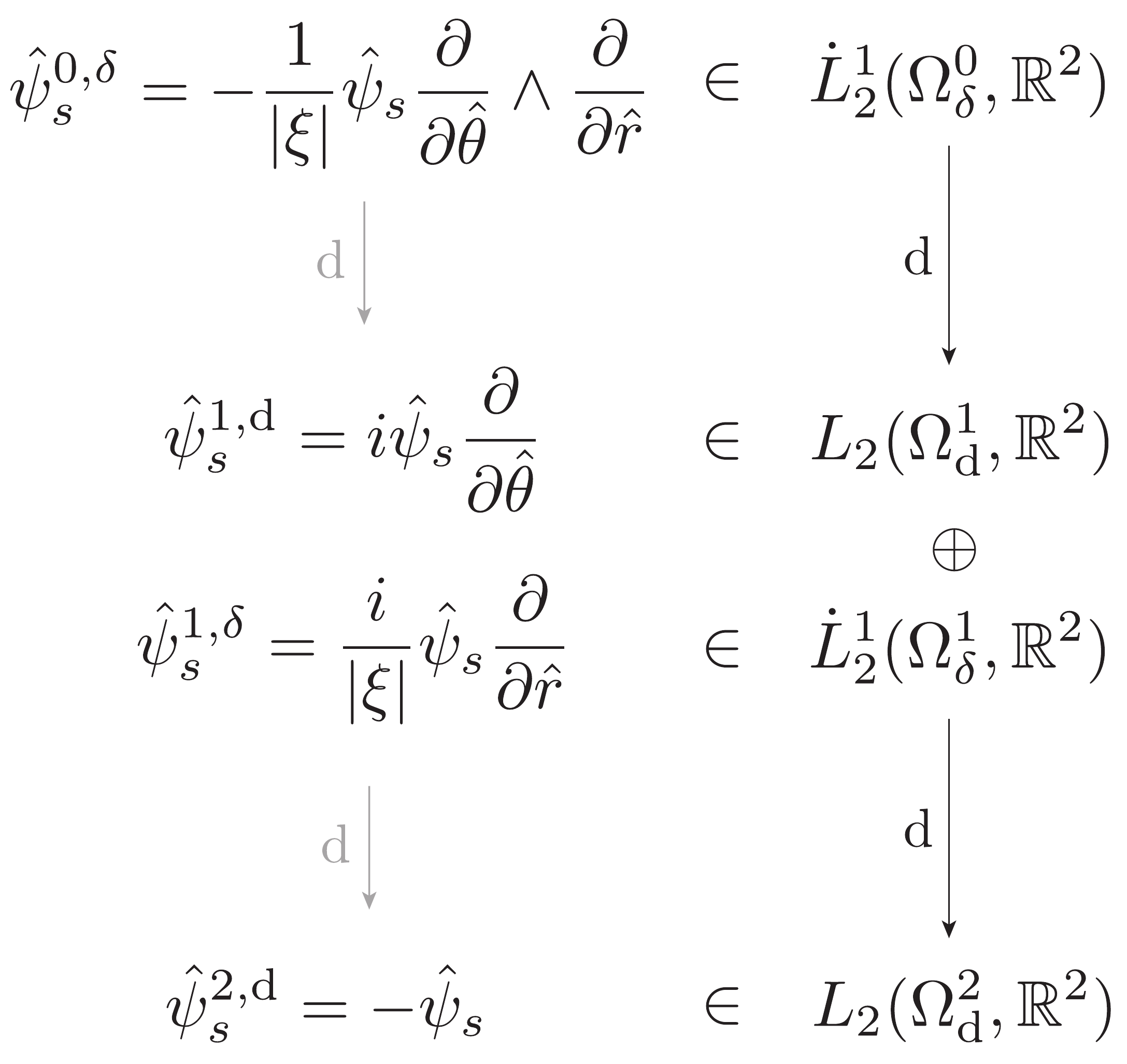}
  \caption{Structure of $\Psi\mathrm{ec}(\R^2)$.}
  \label{fig:psiec_r2}
\end{figure}

We begin by defining polar differential form wavelets.

\begin{definition}
  \label{def:psiforms}
  Let $\{ \partial / \partial \hat{\theta}$, $\partial / \partial \hat{r} \}$ and $\{ \partial / \partial \hat{\theta}$, $\partial / \partial \hat{\phi}$, $\partial / \partial \hat{r} \}$ be orthonormal frames in spherical coordinates for $T^* \widehat{\R}^2$ and $T^* \widehat{\R}^3$, respectively, with the first $n-1$ coordinates spanning $T^* S_{\xi}^{n-1}$.
  Furthermore, let $\nu \in \{ \dd , \delta \}$, with $\vert \dd \vert = 0$ and $\vert \delta \vert = 1$, and the index $a$ satisfy $a \in A_{n,\fdeg,\nu}$ with
  \commentCL{$A_{\fdeg,\nu}$ is needed for 1,2 forms in $\R^3$ where one requires two different differential form wavelets to span $\Omega_{\nu}$ (since $T^* S^2$ is $2$-dimensional)}
  \begin{align*}
    A_{n,\fdeg,\nu} =
    \left\{ \begin{array}{cc}
      \{ 1 , 2 \}  & n=3 \ \mathrm{ and } \ \fdeg=1, \nu = \delta; \fdeg=2,\nu = \dd
      \\
      \{ 1 \} &  \, \mathrm{otherwise}
    \end{array}
    \right.
  \end{align*}
  When $\hat{\psi}_s(\xi)$ is a polar wavelet in the sense of Sec.~\ref{sec:polar_wavelets} then the \emph{polar differential $\fdeg$-forms wavelets $\psi_{s,a}^{r,\nu,n}(x)$ in $\mathrm{\R^n}$}, defined in frequency space through their Fourier transform $\hat{\psi}_{s,a}^{r,\nu,n}(\xi)$, are as given in Fig.~\ref{fig:psiec_r2} and Fig.~\ref{fig:psiec_r3}.
\end{definition}

\vspace{0.2em}
\begin{remark} Remarks on Def.~\ref{def:psiforms}:
  \label{rem:def:psiforms}
  \begin{enumerate}
    \item The tangent space to the sphere $\hat{S}^2$ in frequency space is $2$-dimensional and hence there are two different wavelets that have a tangential component to $\hat{S}^2$. $A_{\fdeg,\nu}$ is used to index these.
    To simplify notation we will typically omit $a$ when $a \in \{ 1 \}$ and the same holds when $n$ is clear from the context.

    \item As discussed in Sec.~\ref{sec:fourier_forms}, the Fourier transform of an $\fdeg$-form is an $(n-\fdeg)$-form in frequency space.
  Although they are defined in the Fourier domain, we use the degree of the form in space in our nomenclature, i.e. $\hat{\psi}_{s,a}^{\fdeg,\nu,n} \in \widehat{\Omega}^{n-\fdeg}(\widehat{\R}^n)$.
    \item The above definition can be extended to differential forms on $\R^1$. There one only has a radial direction and the wavelets are hence given by $\hat{\psi}_{s}^{0,\delta,1} = -i \hat{\psi}_{s}(\xi) \, \partial / \partial \xi$ and $\hat{\psi}_{s}^{1,\dd,1} = \hat{\psi}_{s}(\xi)$.
        This extension is useful, for example, when one considers fiber integration, cf. Sec.~\ref{sec:psiform:slicing}.
  \end{enumerate}
\end{remark}

\begin{figure}[t]
  \centering
  \includegraphics[width=0.94\textwidth]{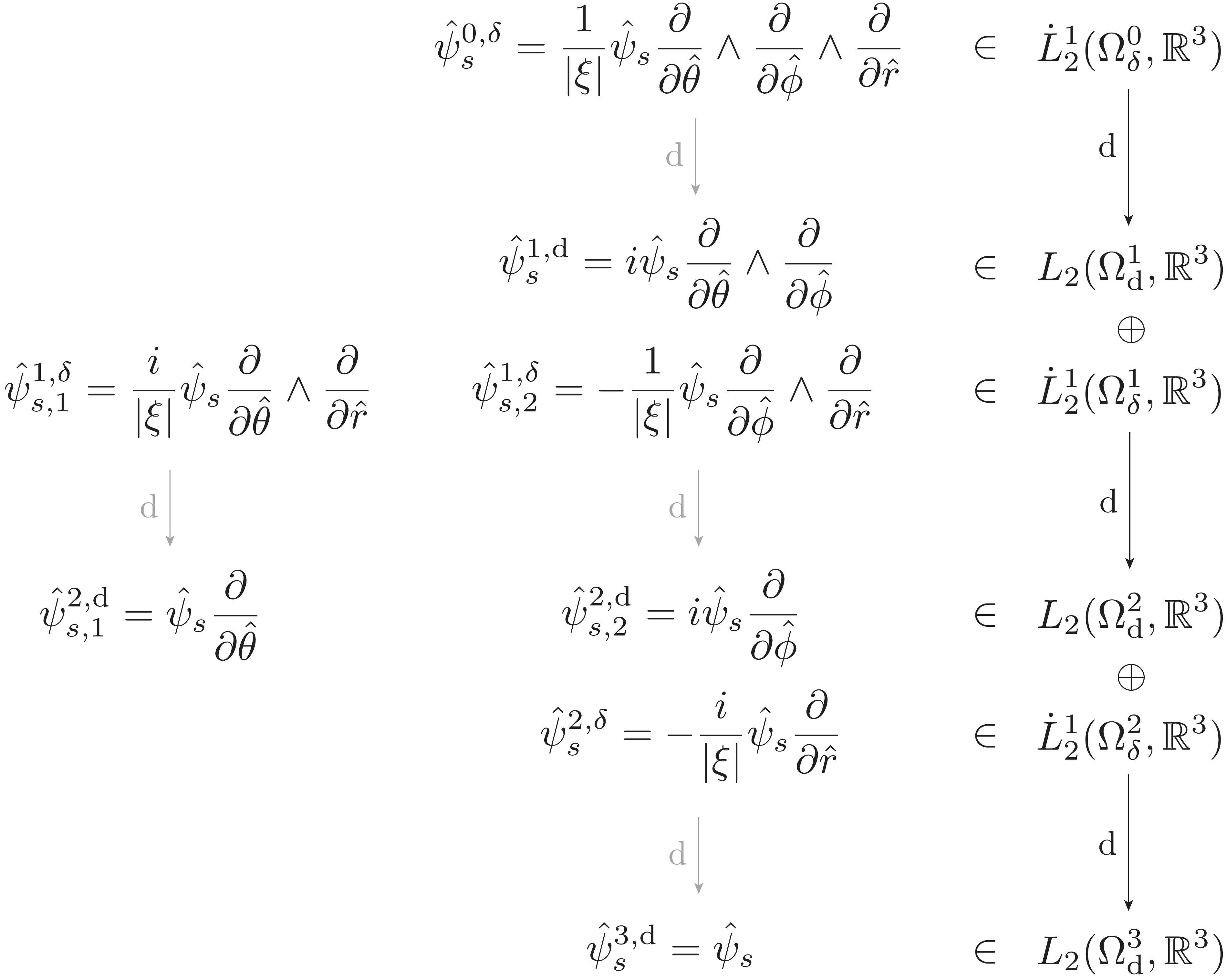}
  \caption{Structure of $\Psi\mathrm{ec}(\R^3)$.}
  \label{fig:psiec_r3}
\end{figure}

It follows immediately from the definition that differential $\fdeg$-forms wavelets are differential forms in sense of the continuous theory.
This is a vital difference to existing discretizations of exterior calculus, e.g.~\cite{Desbrun2006,Arnold2018}.
We will return to this point in the sequel.
The principal utility of differential form wavelets stems from the following result.

\begin{theorem}
  \label{thm:psiforms:frames}
  Let $\{ \psi_s \}_{s \in \mathcal{S}}$ be a scalar, tight polar wavelet frame for $L_2(\R^n)$, i.e. Proposition~\ref{prop:polarlet:2d:parseval} or Proposition~\ref{prop:tight_frame:3d} holds.
  Then the induced polar differential $\fdeg$-forms wavelets $\psi_{s,a}^{\fdeg,\nu,n}$ provide a tight frame for $\dot{L}_2^{\vert \nu \vert}(\Omega_{\nu}^{\fdeg},\R^n)$, i.e. any $r$-form $\alpha \in \dot{L}_2^{\vert \nu \vert}(\Omega_{\nu}^{\fdeg},\R^n)$ has the representation
  \begin{align*}
    \alpha(x) = \sum_{a \in A_{\fdeg,\nu}} \sum_{s \in \mathcal{S}} \Big\langle \! \! \Big\langle \alpha \, , \, \psi_{s,a}^{\fdeg,\nu,n} \Big\rangle \! \! \Big\rangle \, \psi_{s,a}^{\fdeg,\nu,n}(x) .
  \end{align*}

  The polar differential $r$-form wavelets are, furthermore, real-valued in space and have the following explicit expressions.
        In $\R^2$ they are:
      % COMMENT:
      % They are real valued, for 0-forms this are just scalar polar wavelets; for 1-forms the i factor, that, in a certain sense arises through the exterior derivative, is required---originally we added it ad hoc to ensure the div-free wavelets are real-valued
      \begin{align*}
        \psi_{s}^{0,\delta}(x) &= \frac{2^j}{2\pi} \sum_{m} i^m \beta_m^s \, e^{i m \theta_x} \, h_m(\vert 2^{j_s} x - k_s \vert)
        \\[5pt]
        \psi_{s}^{1,\dd}(x) &= \frac{2^j}{4\pi}
          \sum_{\sigma \in \{ -1, 1\} } \sum_{m} i^{m_\sigma} \, \beta_m  \, e^{i m_\sigma \theta} \, h_{m_\sigma}(\vert 2^{j_s} x - k_s \vert) \big( - \! \sigma \, \dd x^1  + i \, \dd x^2)
        \\[5pt]
        \psi_{s}^{2,\delta}(x) &= -\frac{2^j}{4\pi}
         \sum_{\sigma \in \{ -1, 1\} } \sum_{m} i^{m_\sigma} \, \beta_m  \, e^{i m_\sigma \theta} \, h_{m_\sigma}(\vert 2^{j_s} x - k_s \vert)   \big( i \, \dd x^1 + \sigma \, \dd x^2 \big)
         \\[5pt]
        \psi_{s}^{2,\dd}(x) &= \frac{2^j}{2\pi} \sum_{m} i^m \beta_m^s \, e^{i m \theta_x} \, h_m(\vert 2^{j_s} x - k_s \vert) \, \dd x^1 \wedge \dd x^2
      \end{align*}
      where $m_{\sigma} = m + \sigma$ and $h_{m}(\cdot)$ is the Hankel transform of $\hat{h}(\vert \xi \vert)$ of order $m$.
      In $\R^3$ the differential $r$-form wavelets are:
      \begin{align*}
        \psi_{s}^{0,\delta}(x) &= \frac{-2^{3j/2}}{8 \pi^2} \sum_{l,m} i^l \, \kappa_{lm}^{s} \, y_{lm}(\bar{x}) \, h_{l}^{1}(\vert 2^{j_s} x - k_s \vert)
        \\[5pt]
        \psi_{s,a}^{1,\nu}(x) &= \frac{-2^{3j/2}}{8 \pi^2} \sum_{j_1=1}^3  \mathrm{sgn}(\sigma_j) \sum_{l,m} i^l \,\big( \hat{\psi}_{s,a}^{1,\nu} \big)_{lm}^{j_2,j_3} \, y_{lm}(\bar{x}) \, h_l^{\vert \nu \vert}(\vert 2^{j_s} x - k_s \vert) \, \dd x^{j_1}
        \\[5pt]
        \psi_{s,a}^{2,\nu}(x) &= \frac{-2^{3j/2}}{8 \pi^2} \sum_{j_1=1}^3 \mathrm{sgn}(\sigma_j) \sum_{l,m} i^l \big( \hat{\psi}_{s,a}^{2,\nu} \big)_{lm}^{j_1} \, y_{l m}(\bar{x}) \, h_{l}^{\vert \nu \vert}( \vert  2^{j_s} x - k_s \vert ) \, \dd x^{j_2} \wedge \dd x^{j_3}
        \\[5pt]
        \psi_{s}^{3,\dd }(x) &= \frac{-2^{3j/2}}{8 \pi^2} \sum_{l,m} i^l \, \kappa_{lm}^{s} \, y_{lm}(\bar{x}) \, h_{l}(\vert 2^{j_s} x - k_s \vert) \, \dd x^1 \wedge \dd x^2 \wedge \dd x^3
        %%% Old explicit formulas for different 2-forms
        %        \psi_{s,1}^{2,\dd}(x) &= \frac{2^{3j/2}}{(2\pi)^{3/2}} \sum_{j_1=1}^3 \sum_{l,m} i^l \Bigg( \sum_{l_1,m_1} \kappa_{l_1 m_1}^s \sum_{l_2,m_2} \theta_{l_2 m_2}^{j_1} \, G_{l_1,m_1;l_2,m_2}^{l,m} \Bigg)
%                \\
%                & \quad \quad \quad \quad \quad \quad \quad \quad \quad \quad \quad \times y_{l m}(\bar{x}) \, h_{l}( \vert  2^{j_s} x - k_s \vert ) \, \dd x^{j_2} \wedge \dd x^{j_3}
%         \\[5pt]
%        \psi_{s,2}^{2,\dd}(x) &=
%        \frac{-2^{3j/2+1}}{2\pi \, \sqrt{3}} \sum_{\sigma \in \{ -1, 1 \}} \sum_{l,m} i^l \Bigg( \sum_{l_1=l-1}^{l+1}\kappa_{l_1,m-\sigma}^{s} \, G_{l_1,m-\sigma;1,\sigma}^{l,m} \Bigg)
%        \\
%        & \quad \quad \quad \quad \quad \times y_{lm}(\bar{x}) \, h_{l}( \vert  2^{j_s} x - k_s \vert ) \, \big( \dd x^2 \wedge \dd x^3 - i \, \sigma \, \dd x^1 \wedge \dd x^3 \big)
%         \\[5pt]
%        \psi_{s}^{2,\delta}(x) &=
%        \frac{2^{3j/2}}{2\pi \, \sqrt{3}} \sum_{\sigma \in \{ -1,0,1\}}\sum_{l,m} i^l \Bigg( \sum_{l_1=l-1}^{l+1} \kappa_{l_1 m-\sigma}^{s} \, G_{l_1,m-\sigma;1,\sigma}^{l,m} \Bigg)
%        \! \!
%        \\
%        & \! \! \! \! \! \! \! \! \! \! \! \! \! \! \times y_{l,m}(\bar{x}) \,  h_{l_2}( \vert 2^{j_s} x - k_s \vert) \, \big( \sigma \, \dd x^2 \wedge \dd x^3 - i \vert \sigma \vert \, \dd x^1 \wedge \dd x^3 + \sqrt{2} \delta_{\sigma 0} \, \dd x^1 \wedge \dd x^2 \big)
%        \\[5pt]
      \end{align*}
      where
      \commentCL{The $\pm i$ factor in the definition of the wavelets is subsumed into the coefficients. Constant factors and much easier to deal with this way (or at least the resulting expressions are much easier written generically when one does so.}
      \begin{align*}
        \big( \hat{\psi}_{s,2}^{1,\delta} \big)_{lm}^{j_2,j_3} &= -\sum_{l_1,m_1} \sum_{l_2,m_2} \kappa_{l_1 m_2}^s \, \big( \hat{\phi}^{j_1} \! \cdot \! \hat{r}^{j_2} )_{l_2 m_2} \, G_{l_1 m_1;l_2 m_2}^{l m}
        \\[4pt]
        \big( \hat{\psi}_{s,1}^{2,\dd} \big)_{lm}^{j_1} &= \sum_{l_1,m_1} \sum_{l_2,m_2} \kappa_{l_1 m_2}^s \, \hat{\theta}_{l_2 m_2}^{j_1} \, G_{l_1 m_1;l_2 m_2}^{l m}
      \end{align*}
      and analogous for the other cases of $\psi_{s,a}^{1,\nu,3}$ and $\psi_{s,a}^{1,\nu,3}$.
      The $\hat{\phi}^{j}$, $\hat{\phi}^{j}$, $\hat{r}^{j}$ are the components of the spherical form basis functions $\partial / \partial \hat{\theta}$, $\partial / \partial \hat{\phi}$, $\partial / \partial \hat{r}$ with respect to the Cartesian ones  $\partial / \partial \xi^j$.
      The radial window $h_{l}^{q}( \cdot )$ is given by the modified spherical Hankel transform
      \begin{align*}
%        \label{eq:h_hankel_sphercial_weighted}
        h_{l}^{(q)}(r) = \int_{\widehat{\R}^+} \frac{1}{(\hat{r})^q} \hat{h}( \hat{r}) \, j_l( \hat{r} \, r) \, \hat{r}^2 \, \dd \hat{r} .
      \end{align*}

\end{theorem}

The proof is relegated to Sec.~\ref{sec:psiforms:proofs}.
For exact forms in $\Omega_{\dd}^r$, i.e. $\vert \nu \vert = 0$, we thus have that the $\psi_{s,a}^{\fdeg,\nu,n}$ form a tight frame for $L_2(\Omega_{\dd}^{\fdeg},\R^n) \cong \dot{L}_2^0(\Omega_{\dd}^{\fdeg},\R^n)$ and for co-exact ones in $\Omega_{\delta}^r$, i.e. $\vert \nu \vert = 1$, they are a tight frame for the homogeneous Sobolev space $\dot{L}_2^1(\Omega_{\delta}^{\fdeg},\R^n)$, see Fig.~\ref{fig:psiec_r2} and Fig.~\ref{fig:psiec_r3}.
The theorem ensures that the polar differential form wavelets $\psi_{s,a}^{\fdeg,\nu,n}$ can represent the differential forms of interest in applications in $\R^2$ and $\R^3$.
Plots of some $\psi_{s,a}^{\fdeg,\nu,n}$ can be found Fig.~\ref{fig:psi_delta_22} and Fig.~\ref{fig:psis:3d}.
As in the scalar case, the radial window $h_{l}^{q}( \cdot )$ has a closed form expression, for example, when one uses in the Fourier domain the window $\hat{h}(\vert {\xi} \vert)$ that has been proposed for the steerable pyramid~\cite{Portilla2000}.
\commentCL{The same does hold also, e.g., with the box function window, i.e. for radial sinc wavelets.}

\begin{remark}
  \label{rem:choice_of_function_spaces}
  There are alternatives to the function spaces we used in Theorem~\ref{thm:psiforms:frames}.
  For example, starting with $\dot{L}_2^{n/2}(\Omega_{\delta}^0,\R^n)$ for functions and losing one degree of regularity with every exterior derivative, i.e. working with $\dot{L}_2^{n/2-\fdeg}(\Omega_{\nu}^{\fdeg},\R^n)$, one obtains a more symmetric construction.
  When one then also works with the dual spaces $\dot{L}_2^{-n/2+\fdeg}(\Omega_{\delta}^{\fdeg},\R^n)$, one has at least partial closure under the Hodge dual.
  Our choice is motivated by the classical spaces $H(\mathrm{curl},\R^n)$ and $H(\mathrm{div},\R^n)$, cf. Remark~\ref{remark:hcurl}, but it is an interesting direction for future work to more systematically explore alternatives.
\end{remark}

The following proposition establishes that the $\psi_{s,a}^{\fdeg,\nu,n}$ satisfy important properties of Cartan's exterior calculus and hence constitute a discretization of it.
% Note also that exact $\fdeg$-forms are tangential to the sphere $S_{\xi}^{n-1}$ in frequency space and this precisely enforces exactness.

\begin{theorem}
  \label{thm:psiforms:exterior}
  The polar differential $\fdeg$-forms wavelet $\psi_{s,a}^{\fdeg,\nu,n} \in \dot{L}_2^{\vert \nu \vert}(\Omega_{\nu}^r,\R^n)$ satisfy:
  \begin{enumerate}[i)]

    \item Exterior derivative:
      \begin{align*}
        \dd \psi_{s,a}^{\fdeg,\dd,n}(x) &= 0
        \\[5pt]
        \dd \psi_{s,a}^{\fdeg,\delta,n}(x) &= \psi_{s,a}^{\fdeg+1,\dd,n}(x)
      \end{align*}
      and furthermore
      \begin{align*}
        \Big\Vert \, \psi_{s,a}^{\fdeg,\delta,n} \, \Big\Vert_{\dot{L}_2^1}
%        = \Big\Vert \, \dd \psi_{s,a}^{\fdeg,\delta} \, \Big\Vert_{L_2}
        = \Big\Vert \, \psi_{s,a}^{\fdeg+1,\dd,n} \, \Big\Vert_{L_2}
      \end{align*}

     \item Codifferential:
      \begin{align*}
        \delta \psi_{s,a}^{\fdeg,\delta,n}(x) = 0
      \end{align*}

    \item Hodge dual:
      % COMMENT: there is no closure here and this seems to be a real issue for practical numerics
      \begin{align*}
        \star \, \hat{\psi}_{s}^{\fdeg,\nu,2}(\xi) &= \mathrm{sgn}(\sigma) \, \vert \xi \vert^{2 \vert \nu \vert-1} \, \hat{\psi}_{s}^{n-\fdeg,\bar{\nu},2}(\xi)
        \\[7pt]
        \star \, \hat{\psi}_{s,a}^{\fdeg,\nu,3}(\xi) &= (-1)^{\lfloor (r+1)/2 \rfloor} \, \vert \xi \vert^{2\vert \nu \vert-1} \, \hat{\psi}_{s,\bar{a}}^{n-\fdeg,\bar{\nu},3}(\xi)
      \end{align*}
      where $\sigma$ is the permutation associated with the form basis functions of the wavelet and $\bar{\delta} = \dd$ and $\bar{\dd} = \delta$.

    \item Rigid body transformations:
      \begin{align*}
          \alpha = \sum_{a} \sum_{j_s,k_s,t_s}^{J} \alpha_{s,a} \, \psi_{s,a}^{\fdeg,\nu,n} \ \ \Leftrightarrow \ \ A \alpha = \sum_{a} \sum_{j_s,k_s,t_s}^{J} \tilde{\alpha}_{s,a} \, \psi_{s,a}^{\fdeg,\nu,n} ,
      \end{align*}
      where $A \in \mathrm{E}(n)$ is a rigid body transformation and $\alpha \in \dot{L}_2^{\vert \nu \vert}(\Omega_{\nu}^{\fdeg},\R^n)$, i.e. the translation and/or rotation of a $J$-bandlimited representation in polar differential form wavelets can be represented in the same frame with the same bandlimit for some coefficients $\tilde{\alpha}_s$,
      %      with the $ and .
%      The spatial radial windows $h_m(\vert x \vert)$ and $h_l(\vert x \vert)$, which are obtained by the spherical Hankel transform of the frequency window $\hat{h}(\vert \xi \vert)$ have closed form expressions when $\hat{h}(\vert \xi \vert)$ is the window from the steerable pyramid~\cite{Portilla2000} and
%      The expressions for the $1$-form wavelets $\psi_{s,a}^{1,\nu,3}$ that are not listed can be obtained using the Hodge dual.
  \end{enumerate}
\end{theorem}

\begin{figure}
  \centering
  \includegraphics[trim={60 50 40 40},clip,width=\textwidth]{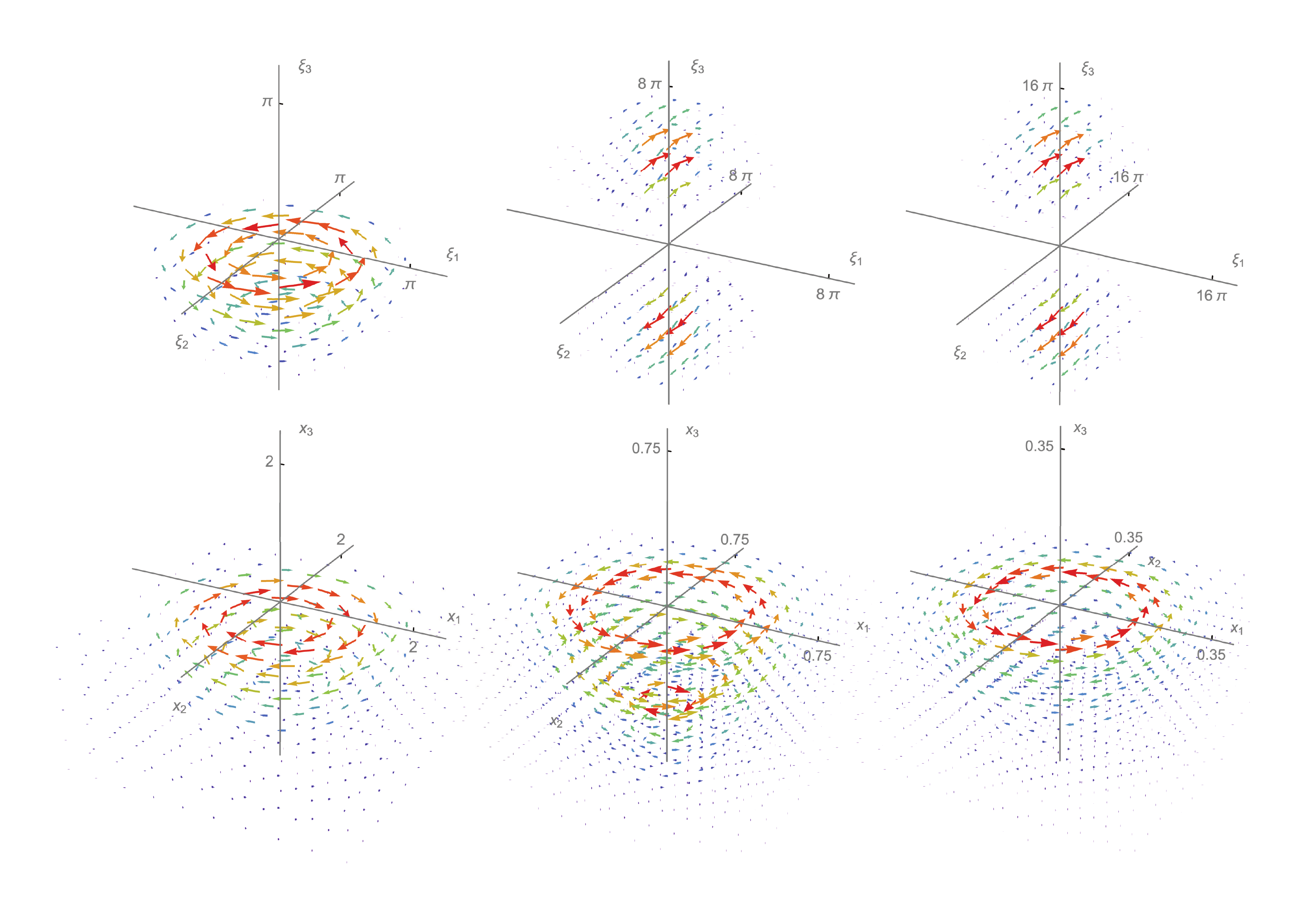}
  \caption{Exact differential form wavelets $\psi_{s,2}^{2,\dd,3}$ in $\R^3$ in the frequency domain (top) and the spatial one (bottom) visualized as vector fields. The left function is isotropic around the $x_3$-axis and the others are directional modeling high frequency features aligned with the $x_1$-$x_2$ plane.}
  \label{fig:psis:3d}
\end{figure}

The proof of the theorem is again relegated to Sec.~\ref{sec:fourier_forms}.
The theorem shows that polar differential form wavelets are closed under the exterior derivative and in the sense of the continuous theory since the $\psi_{s,a}^{\fdeg,\nu,n}$ are continuous differential forms.
They furthermore respect the Hodge-Helmholtz decomposition, since one has separate representations for exact and co-exact forms.
By linearity these properties carry over to arbitrary forms represented in the wavelets $\psi_{s,a}^{\fdeg,\nu,n}$.
This makes them useful for numerical calculations.

\begin{remark} Remarks on Theoren~\ref{thm:psiforms:exterior}:

  \begin{enumerate}[i)]

    \item For the evaluation of the exterior derivative in the Fourier domain, i.e. $\hat{\dd} = \ip_{i \xi}$ the Leibniz rule has to be used and this leads to additional sign changes.

    \item The imaginary unit in the definition of the differential $1$-form wavelets $\hat{\psi}_{s}^{1,\nu,2}$ in the Fourier domain in Fig.~\ref{fig:psiec_r2}, which appears a priori because of the exterior derivative, is also required for the wavelets to be real-valued in the spatial domain, cf. the proof of the theorem.

    \item For the Hodge dual one does not have closure. This arises from the different functions spaces that are used for exact and co-exact forms, which, however, are required for the closure of the exterior derivative.
        One remedy would be to use an alternative functional analytic setting and also work with the dual spaces, cf. Remark~\ref{rem:choice_of_function_spaces}.

    \item For isotropic basis functions, i.e. when $\beta_{m} = \delta_{m0}$, the $1$-form wavelets in $\R^2$ are of the form of an ideal source and an ideal divergence free vortex, i.e. for the mother wavelets one has
        \begin{align*}
          \psi^{1,\dd} = h_1(\vert x \vert) \, \dd r
          \quad \quad \quad \quad
          \psi^{1,\delta} = h_1(\vert x \vert) \, \dd \theta ,
        \end{align*}
        see left most column of Fig.~\ref{fig:psi_delta_22}.
        where $\dd \theta$ and $\dd r$ are the canonical form basis functions in polar coordinates in space.
        Analogously, in $\R^3$ one obtains in the isotropic case an axis-aligned source and vortex tube , cf. Fig.~\ref{fig:psis:3d}.

    \item The form basis vector $\partial / \partial \hat{\theta}$ and $\partial / \partial \hat{\phi}$ have singularities at the poles.
      In~\cite{Lessig2018z} it has been proposed to use a tight, redundant frame to span $T S^2$ and a simple construction of such a frame was proposed there. %, see \ref{sec:hedgehog}.
      While possible, some care is required in our context because in exterior calculus it is usually assumed that the tangent space is spanned by a (biorthogonal) basis.
      We leave it to future work to either re-write the necessary parts of exterior calculus using a redundant frame or find suitable coordinate charts for $S^2$ that work well with polar wavelets and avoid the singularities.
      With compactly supported directional localization windows $\hat{\gamma}(\bar{\xi})$ one can choose $\{ \partial / \partial \hat{\theta}, \partial / \partial \hat{\phi} \}$ such that the singularities are outside the support.

  \end{enumerate}
\end{remark}

%%%%%%%%%%%%%%%%%%%%%%%%%%%%%%%%%%%%%%%%%%%%%%%%%%%%%%%%%%%%%%%%%%%%%%%%%%%%%%%%%
\subsection{Exterior Calculus using Differential Forms Wavelets}

We continue with a discussion of the differential form wavelet interpretation of important results in the exterior calculus.

%%%%%%%%%%%%%%%%%%%%%%%%%%%%%%%%%%%%%%%%%%%%%%%%%%%%
\subsubsection{Stokes theorem}
\label{sec:psiforms:stokes}

We showed in Sec.~\ref{sec:exterior_algebra:fourier} that Stokes' theorem can be written as
\begin{align*}
   \int_{\R^n} \alpha \wedge \chi_{_{\partial \M}} = \int_{\R^n} \dd \alpha \wedge \chi_{_{\M}} .
\end{align*}
%where $\dd \alpha \in \mathcal{S}(\Omega^{n-\fdeg},\R^n)$ is the smooth extension of the original form off $\M$ and
where $\chi_{_{\partial \M}} \in \mathcal{S}'(\Omega^{n-\fdeg+1},\R^n)$ and $\chi_{_{\M}} \in \mathcal{S}'(\Omega^{n-\fdeg},\R^n)$ are the characteristic differential forms of $\partial \M$ and $\M$, respectively.
The differential form wavelet representation of $\alpha \in \mathcal{S}(\Omega^{\fdeg-1},\R^n)$ and $\dd \alpha \in \mathcal{S}(\Omega^{\fdeg},\R^n)$, which are smooth, compactly supported extensions of the original forms on $\partial \M$ and $\M$ along the normal, are given by
\begin{align*}
  \alpha = \sum_{s} \alpha_s \, \psi_{s}^{\fdeg-1,\delta}
  \quad \quad \quad \quad
  \dd \alpha = \sum_{s} \alpha_s \, \psi_{s}^{\fdeg,\dd}
\end{align*}
where we immediately used that $\dd^2 = 0$ and $\partial \partial \M = \varnothing$ so that it suffices to consider the co-exact part of $\alpha$.
% COMMENT
% Decompose \alpha into its exact and coexact part and then apply Stokes theorem to exact one. Since $\partial \partial \M = \empty$,~\cite[Proposition 8.2.6]{Marsden2004} and Stokes theorem evaluates to zero when the boundary is empty (as stated explicitly in~\cite[Theorem 8.2.8]{Marsden2004}) the integral of the exact part is zero.
Note also that the coefficients $\alpha_s$ on both sides of the equation are identical.
With the above representations and using linearity, Stokes' theorem can be written as
\begin{align}
  \label{eq:stokes:diff_forms:1}
  \sum_{s} \alpha_s \int_{\R^n} \psi_{s}^{\fdeg-1,\delta} \wedge \chi_{_{\partial \M}}
  &= \sum_{s} \alpha_s \int_{\R^n} \psi_{s}^{\fdeg,\dd} \wedge \chi_{_{\M}} .
  \intertext{By interpreting the integrals as the projection of $\chi_{_{\partial \M}}$ and $\chi_{_{\M}}$ onto the wavelets we obtain}
  \label{eq:stokes:diff_forms:2}
  \sum_{s} \alpha_s \, \tilde{\chi}_s^{_{\partial \M}}
  &= \sum_{s} \alpha_s \, \tilde{\chi}_s^{_{\M}}
\end{align}
where we use the tilde to distinguish $\tilde{\chi}_s^{_{\partial \M}}$ and $\tilde{\chi}_s^{_{\M}}$ from the frame coefficients obtained with the $L_2$-inner product for differential forms with the Hodge dual.
%With Eq.~\ref{eq:stokes:diff_forms} Stokes' theorem hence becomes the scalar product between the coefficients $\alpha_s$ of the form $\alpha$ (or its exterior derivative $\dd \alpha$) and those of the characteristic functions $\tilde{\chi}_{_{\partial \M}}$ and $\chi_{_{\M}}$
It follows immediately from Eq.~\ref{eq:stokes:diff_forms:2} that the coefficients $\tilde{\chi}_s^{_{\partial \M}}$ and $\tilde{\chi}_s^{_{\partial \M}}$ are in fact equal, i.e. $\tilde{\chi}_s^{_{\partial \M}} = \tilde{\chi}_s^{_{\M}}$.

When $\psi_{s}^{\fdeg,\dd}$ is a volume form, i.e. $\fdeg = n$, then the integral on the right hand side of Eq.~\ref{eq:stokes:diff_forms:1} describes the frame coefficient of the scalar characteristic function $\chi_{_{\M}}$ with the volume form wavelet $\psi_{s}^{n,\dd} = \psi_s(x) \, \dd x$, cf. Theorem~\ref{thm:psiforms:exterior}, v.), i.e we can write
\begin{align*}
  \tilde{\chi}_s^{_{\M}} = \int_{\R^n} \psi_{s}^{\fdeg,\dd} \, \chi_{_{\M}} \, \dd x .
\end{align*}
The behavior of these coefficients $\tilde{\chi}_s^{_{\M}}$ has been studied extensively in the literature on ridgelets (e.g.~\cite{Candes1999a,Donoho2000a}), curvelets (e.g.~\cite{Candes1999b,Candes2004,Candes2005a,Candes2005b,Candes2006}), shearlets (e.g.~\cite{Guo2006,Kutyniok2012a,Kutyniok2016}), contourlets (e.g.~\cite{Do2003,Do2005a}) and related constructions such as $\alpha$-molecules (e.g.~\cite{Grohs2013,Grohs2016}).
From these results it is known that, for sufficiently fine levels $j_s$, the coefficients $\tilde{\chi}_s^{_{\M}}$ are non-negligible only when $\psi_{s}^{r,\dd}$ is in the neighborhood of the boundary $\partial \M$ and, in the anisotropic, curvelet-like case, when it is oriented along $\partial \M$.
By Eq.~\ref{eq:stokes:diff_forms:2}, the ``geometry'' coefficients $\tilde{\chi}_s^{_{\M}} = \tilde{\chi}_s^{_{\partial \M}}$ thus select the ``signal'' coefficients $\alpha_s$ (of $\alpha$ and hence also $\dd \alpha$) in the vicinity of $\partial \M$ so that $\alpha_s$ will provide a significant contribution to the term only when $\tilde{\chi}_s^{_{\M}}$ is non-negligible.
This selection implements the pullback along the inclusion map $i : \partial \M \to \M$ in the original form of Stokes' theorem in Eq.~\ref{eq:stokes_thm}.
%This becomes most explicit with the inner product form of Stokes' theorem, cf. Eq.~\ref{eq:stokes_thm:characteristic:inner_product},
%\begin{subequations}
%\begin{align}
%   \big\langle \! \big\langle \alpha \, , \, \chi_{_{T \partial \M}} \big\rangle \! \big\rangle
%  &= \big\langle \! \big\langle \dd \alpha \, , \, \chi_{_{T \M}} \big\rangle \! \big\rangle
%  \\[7pt]
%   \sum_{s \in \mathcal{S}} \alpha_s \big\langle \! \big\langle \psi_{s}^{\fdeg,\delta} \, , \, \chi_{_{T \partial \M}} \big\rangle \! \big\rangle
%  &= \sum_{s \in \mathcal{S}} \alpha_s \big\langle \! \big\langle \psi_{s}^{\fdeg+1,\dd} \, , \, \chi_{_{T \M}} \big\rangle \! \big\rangle
%\end{align}
%\end{subequations}
%Eq.~\ref{eq:stokes:diff_forms:projection}
The above argument also shows that anisotropic wavelets are more efficient for realizing the boundary integral numerically than isotropic ones since the coefficients $\tilde{\chi}_s^{_{\M}}$ are then sparser and the sums in Eq.~\ref{eq:stokes:diff_forms:2} contain fewer non-negligible terms.
In particular, when $\partial \M$ is $C^2$ then curvelets, shearlets, and contourlets yield quasi-optimally sparse representations of $\chi_{_{\M}}$.

We already discussed in Example~\ref{ex:stokes:characteristic:circle} that $\chi_{_{\partial \M}} = \delta_{\partial \M}(x) \, n_i \, \dd \tilde{x}^i$ is just the wavefront set $\mathrm{WF}(\chi_{_{\M}})$ of $\chi_{_{\M}}$.
The integral on the left hand side of Eq.~\ref{eq:stokes:diff_forms:1} is thus a projection of $\mathrm{WF}(\chi_{_{\M}})$ onto the $\psi_{s}^{r-1,\delta}$.
That curvelet-like wavelets can represent, or resolve, the wavefront set has been established previously in the literature~\cite{Candes2005a,Kutyniok2008}.
%This is consistent with the known results in the literature that curvelets and shearlets resolve the wavefront set~\cite{Candes2005a,Kutyniok2008}.
%But $\chi_{_{\partial \M}}$, which in coordinates can be given by
%This is consistent with our results since from Eq.~\ref{eq:stokes:diff_forms} one would, indeed, expect that Eq.~\ref{eq:stokes:diff_forms:projection:2} in determined by $\mathrm{WF}(\chi_{_{\M}})$, and
The formulation of Stokes' theorem using differential form wavelets provides, in our opinion, an interesting perspective on this work~\cite{Candes2005a,Kutyniok2008}.
% COMMENT: The discussion with the wavefront set here is still muddled

Eq.~\ref{eq:stokes:diff_forms:2} is valid for any $\fdeg$ and not just the case $r=n$ for which the existing results for curvelets, shearlets and related constructions apply.
To our knowledge, there are no results on the approximation power of ``vector-valued'' wavelets or curvelets.
We conjecture that also in the general case one has non-negligible coefficients $\tilde{\chi}_s^{_{\partial \M}}$ and $\tilde{\chi}_s^{_{\M}}$ only in the vicinity of $\partial \M$ and with an orientation aligned with it.
We leave a thorough investigation of this question to future work.

\begin{example}
  \label{ex:kelvins_circulation_theorem}
  We consider Kelvin's circulation theorem in fluid mechanics, e.g.~\cite[Ch. 1.2]{Chorin1993} or~\cite[Ch. 1]{Khesin1998}, that relates the fluid velocity and vorticity using Stokes' theorem.
  With differential forms it can be stated as
  \begin{align*}
    \int_{\partial \Sigma} u^{\flat} = \int_{\Sigma} \zeta
  \end{align*}
  where $u^{\flat}$ is the $1$-form associated with the fluid velocity field $\vec{u} \in \mathfrak{X}(U)$, $\zeta = \dd u^{\flat} \in \Omega^2(U)$ is the vorticity, and $\Sigma$ is an area embedded in the flow domain $U \subset \R^2$.
  Representing $u^{\flat}$ using the $1$-form frame,
  \begin{align*}
    u^{\flat} = \sum_{s} u_{s}^{\flat} \, \psi_{s}^{1,\delta} ,
  \end{align*}
  the representation for the associated vorticity is
  \begin{align*}
    \zeta = \dd u^{\flat} = \sum_{s} u_{s}^{\flat} \, \psi_{s}^{2,\dd}
  \end{align*}
  and both hence have the same expansion coefficients $u_{s}^{\flat}$.
  With Eq.~\ref{eq:stokes:diff_forms:2}, Kelvin's circulation theorem thus becomes
  \begin{align}
    \label{eq:ex:kelvins_circulation_theorem:5}
    \sum_{s} u_{s}^{\flat} \int_{\R^2} \psi_{s}^{1,\delta} \wedge \chi_{_{\partial \Sigma}}
    = \sum_{s} u_s^{\flat} \int_{\R^2} \psi_{s}^{2,\dd} \wedge \chi_{_{\Sigma}} .
  \end{align}
  As in Example~\ref{ex:stokes:characteristic:circle}, the characteristic function associated with $\partial \Sigma$ is the weak gradient
  \begin{align*}
    \chi_{_{\partial \Sigma}} = \dd \chi_{_{\Sigma}} = \delta_{\partial \Sigma}(x) \, n_i \, \dd x^i
  \end{align*}
  where the $n_i$ are the components of the normal of $\partial \Sigma$ and $\delta_{\partial \Sigma}(x)$ is its Dirac-distribution.
   In components we hence have for the integrand on the left hand side of Eq.~\ref{eq:ex:kelvins_circulation_theorem:5},
  \begin{align*}
     \psi_{s}^{1,\delta} \wedge \chi_{_{\partial \Sigma}}
     &= \left( \psi_{s}^{1,\delta} \dd x^1 + \psi_{s}^{1,\delta} \dd x^2 \right) \wedge \left( \delta_{\partial \Sigma}(x) \, n_1 \, \dd x^1 + \delta_{\partial \Sigma}(x) \, n_2 \, \dd x^2 \right)
     \\[8pt]
     &= \left( \psi_{s}^{1,\delta} \, n_2 - \psi_{s}^{1,\delta} \, n_1 \right) \, \delta_{\partial \Sigma}(x) \ \dd x^1 \wedge \dd x^2 .
  \end{align*}
  Since, as na{\"i}ve vectors, $(a_1,a_2)$ and $(a_2, -a_1)$ are orthogonal, the above wedge product vanishes when $\psi_{s}^{1,\delta}$ is in the normal direction and it is maximized when it is parallel to the boundary (and no metric is required).
  Hence, as expected, the integral over $\R^2$ implements the line integral along $\partial \Sigma$.

  The integral on the right hand side of Eq.~\ref{eq:ex:kelvins_circulation_theorem:5} is the representation problem for  the cartoon-like function $\chi_{_{\Sigma}}$, which, as discussed above, has been studied extensively in the literature.
  That a co-exact $1$-form wavelet attains the same approximation behavior for the boundary $\partial \Sigma$ is consistent with the results in~\cite{Lessig2018z} where it was shown that polar divergence free wavelets in $\R^2$, which correspond to our exact $1$-forms, attain the same approximation rates as scalar curvelets.
\end{example}

%%%%%%%%%%%%%%%%%%%%%%%%%%%%%%%%%%%%%%%%%%%%%%%%%%%%%%%%%%%%%%%%%%%%%%%%%%%%%%%%%
\subsubsection{Fiber Integration of Differential Forms}
\label{sec:psiform:slicing}

\begin{figure}
  \centering
  \includegraphics[trim={20 75 40 30},clip,width=1.0\columnwidth]{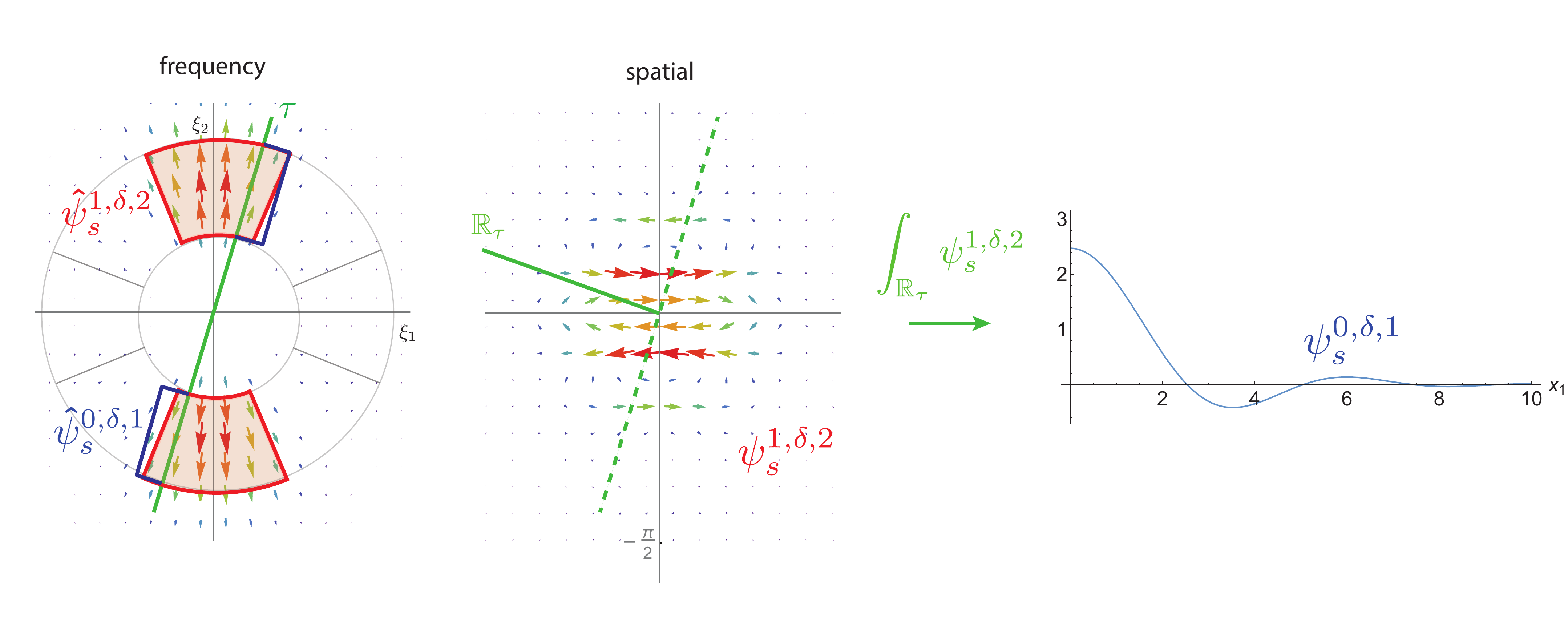}
  \caption{Conceptual depiction of fiber integration of differential form wavelets for a co-exact $1$-form wavelet $\psi_{s}^{1,\delta,2}$.
  \emph{Left:} In the Fourier domain the partial integration corresponds to the restriction to the plane $\tau$ (a line in $2D$) normal to the integration direction $\R_{\tau}$.
  \emph{Middle and right:} In the spatial domain, the projection of $\psi_{s}^{1,\delta,2}({x})$ along the $\R_{\tau}$ axis (middle) yields its restricted counter-part $\psi_{s}^{0,\delta,1}(x_1)$ (right), which is a differential form wavelet in one dimension (cf. Remark~\ref{rem:def:psiforms}).}
  \label{fig:zeta}
\end{figure}

A natural property of differential forms is that the integral of an $\fdeg$-form over a $p$-dimensional sub-manifold yields an $(\fdeg-p)$-form.
In the literature this is sometimes referred to as fiber integration~\cite{Bott1982}.
Our differential form wavelets verify this property for linear sub-manifolds and we will discuss it for the special case $p=1$.
Thus, we consider the integration along a direction $\R_{\tau}$,
\begin{align*}
  \psi_{s}^{\fdeg,\nu} \ \xrightarrow{\quad \int_{\R_{\tau}} \quad} \ \psi_{s'}^{\fdeg-1,\nu}
\end{align*}
for which the integral of an $\fdeg$-form  wavelet is an $(\fdeg-1)$-form one.
Furthermore, $\psi_{s'}^{\fdeg-1,\nu}$, which can be considered either in $\R^n$ or $\R^{n-1}$, has a closed form expression and $s$ and $s'$ are closely related, as we will discuss in the following.
See Fig.~\ref{fig:zeta} for a depiction for $n=2$.

In the Fourier domain, the integral along $\R_{\tau}$ is equivalent to restricting $\hat{\psi}_{s,a}^{\fdeg,\nu,n}$ to the plane $P_{\tau}$ with normal $\tau$, analogous to the classical intertwining between restriction and projection in the classical Fourier slice theorem~\cite{Bracewell1956}.
In the case of volume forms, this was already used in~\cite{Lessig2018d} for a local Fourier slice theorem based on polar wavelets.
The result carries over to arbitrary polar differential form wavelets since these are constructed using differential form basis functions in spherical coordinates.
% a planar restriction in the Fourier domain preserves spherical coordinates, i.e. the planar slice of a differential form separable in spherical coordinates in $\widehat{\R}^n$ is a differential form separable in spherical coordinates in $\widehat{R}^{n-1}$.

We will demonstrate this in $\R^3$ for $\psi_{s}^{2,\delta} \in \dot{L}_2^1(\Omega_{\delta}^2,\R^3)$.
With Def.~\ref{def:psiforms}, the Fourier representation of $\psi_{s}^{2,\delta,3}$, for notational simplicity restricted to the mother wavelet, is
\begin{align*}
  \hat{\psi}_{s}^{2,\delta,3}(\xi)
  = -\frac{i}{\vert \xi \vert} \hat{\psi}_s(\xi) \, \frac{\partial}{\partial \hat{r}}
  = -\frac{i}{\vert \xi \vert} \frac{2^{3/2}}{(2\pi)^{3/2}} \sum_{lm} \kappa_{lm} \, y_{lm}(\hat{\theta},\hat{\phi}) \, \hat{h}(\vert \xi \vert) \, \frac{\partial}{\partial \hat{r}}
\end{align*}
Without loss of generality, let $\tau$ be along the $x_3$-axis; by the closure of the spherical harmonics bands under rotation, cf. Sec.~\ref{sec:preliminaries:sh}, the general case follows by applying a rotation in the spherical harmonics domain using the Wigner-D matrices.
The integration along the $x_3$-axis becomes in the Fourier domain the restriction to the $\xi_1$-$\xi_2$ plane and in spherical coordinates this corresponds to $\theta = \pi / 2$.
Thus,
\begin{align*}
  \left. \hat{\psi}_{s}^{2,\delta,3}(\xi) \right\vert_{\xi_{1,2}}
  = -\frac{2^{3/2} \, i}{(2\pi)^{3/2} \vert \xi_{1,2} \vert} \sum_{lm} \kappa_{lm} \, y_{lm}(\pi/2, \hat{\phi}) \, \hat{h}(\vert \xi_{1,2} \vert) \, \frac{\partial}{\partial \hat{r}_{{1,2}}}
\end{align*}
where $\vert \xi_{1,2} \vert$ and $\partial / \partial \hat{r}_{{1,2}}$ are the restriction of $\vert \xi \vert$ and $\partial / \partial \hat{r}$ to the $\xi_1$-$\xi_2$ plane, respectively.
Note that these are the natural radial variables in the frequency plane $\widehat{\R}^2$.
Expanding the spherical harmonics $y_{lm}$, cf.~\ref{sec:preliminaries:sh}, we obtain
\begin{align*}
  \left. \hat{\psi}_{s}^{2,\delta,3}(\xi) \right\vert_{\xi_{1,2}}
  &= -\frac{2^{3/2} \, i}{(2\pi)^{3/2} \vert \xi_{1,2} \vert} \sum_{lm} \kappa_{lm} \left( C_{lm} \, P_l^m(0) \, e^{i m \hat{\phi}} \right) \hat{h}(\vert \xi_{1,2} \vert) \, \frac{\partial}{\partial \hat{r}_{{1,2}}}
  \\[6pt]
  &= -\frac{2^{3/2} \, i}{(2\pi)^{3/2} \vert \xi_{1,2} \vert} \sum_{m} \underbrace{\left( \sum_l \kappa_{lm} \, C_{lm} \, P_l^m(0) \right)}_{\displaystyle \bar{\beta}_m} e^{i m \hat{\phi}} \, \hat{h}(\vert \xi_{1,2} \vert) \, \frac{\partial}{\partial \hat{r}_{{1,2}}}
  \\[4pt]
  &= -\frac{2^{3/2} \, i}{(2\pi)^{3/2} \vert \xi_{1,2} \vert} \sum_{m} \bar{\beta}_m \, e^{i m \hat{\phi}} \, \hat{h}(\vert \xi_{1,2} \vert) \, \frac{\partial}{\partial \hat{r}_{{1,2}}} .
\end{align*}
Comparing to the definition of $\hat{\psi}_{s}^{1,\delta,2} \in \dot{L}_2^1(\Omega_{\delta}^1,\R^2)$ in Def.~\ref{def:psiforms} and using those of the polar wavelets in Eq.~\ref{eq:polarlets:2d:hat}, we see that the restriction $\hat{\psi}_{s}^{2,\delta,3}(\xi) \vert_{\xi_{1,2}}$ is the frequency representation of a co-exact $1$-form wavelet in the plane $\R_{1,2}^2$ with polar frequency coordinates $(\hat{\phi},\vert \xi_{1,2} \vert)$.
The angular localization coefficients $\bar{\beta}_m$ obtained through the restriction differ from the $\beta_m$ which are commonly used when one starts in $\R^2$.
From the restriction it is, however, clear that $\hat{\psi}_{s}^{2,\dd,3}(\xi) \vert_{\xi_{1,2}}$ will be non-negligible only when the angular window $\hat{\psi}_s(\xi)$ is approximately centered on the $\xi_1$-$\xi_2$ plane and that the projected window, characterized by the $\bar{\beta}_m$, will be localized there around its original location, cf. Fig.~\ref{fig:psis:3d}.
The localization is hence, in an appropriate sense, preserved under integration.
A precise, quantitative analysis of this property is left to future work.

%%%%%%%%%%%%%%%%%%%%%%%%%%%%%%%%%%%%%%%%%%%%%%%%%%%%
\subsubsection{Laplace--de Rham operator}
\label{sec:psiforms:laplace}

The Laplace operator acting on differential forms is known as the Laplace--de Rham operator $\Delta : \Omega^{\fdeg}(\R^n) \mapsto \Omega^{\fdeg}(\R^n)$.
It is defined as~\cite[Def. 8.5.1]{Marsden2004}
\begin{align}
  \label{eq:laplace_derahm}
  \Delta = \dd \delta + \delta \dd = \dd \star \dd \star + \star \dd \star \dd
\end{align}
and, analogous to the usual Laplace operator, it plays a fundamental role in many applications, e.g.~\cite[Ch. 9]{Marsden2004}.
Since our differential form wavelets are closed under the exterior derivative and the Hodge dual has in the Fourier domain a well defined form, also the Laplace--de Rham operator can be computed in closed form.
For example, for a co-exact $2$-form wavelet in $\R^3$ we have
\begin{align*}
  \Delta \psi_{s}^{2,\delta}
  &= \dd \delta \psi_{s}^{2,\delta} + \delta \dd \psi_{s}^{2,\delta} .
  % (-1)^{n(\fdeg-1)+1} \big( \dd \star \dd \star \psi_{s}^{2,\delta,3} + \star \, \dd \star \dd \,\psi_{s}^{2,\delta,3} \big)
\end{align*}
The first term vanish by Theorem~\ref{thm:psiforms:exterior}, ii).
%We will compute the two terms independently using the results of Theorem~\ref{thm:psiforms:exterior}.
For the second term we have, using Theorem~\ref{thm:psiforms:exterior} and Proposition~\ref{prop:hodge_dual:fourier}, that
\begin{align*}
   \star \, \dd \star \dd \, \psi_{s}^{2,\delta}
   &= \star \, \dd \star \psi_{s}^{3,\dd}
   \\
   &= \star \, \dd \mathfrak{F}^{-1} \big( \vert \xi \vert \, \hat{\psi}_{s}^{0,\delta} \big)
   \\
   &= \star \, \mathfrak{F}^{-1} \big( \vert \xi \vert \, \hat{\psi}_{s}^{1,\dd} \big)
   \\
   &= \mathfrak{F}^{-1} \big( \vert \xi \vert^2 \, \hat{\psi}_{s}^{2,\delta} \big) .
\end{align*}
Thus
\begin{align*}
  \Delta \psi_{s}^{2,\delta} = \mathfrak{F}^{-1} \big( \vert \xi \vert^2 \, \hat{\psi}_{s}^{2,\delta} \big)
\end{align*}
and the Laplace--de Rham operator is a differential form of the same type but with a $\vert \xi \vert^2$-weight in the radial direction.
This is also what one would expect from the scalar Laplace operator.
% but it is non-trivial given Eq.~\ref{eq:laplace_derahm}.
One can check that, up to a sign, this holds for all $\psi_{s}^{\fdeg,\nu,n}(x)$ and we hence have the following the proposition.

\begin{proposition}
  \label{prop:laplace_derahm:psi}
  The Laplace--de Rahm operator $\Delta = \dd \delta + \delta \dd$ of a polar differential form wavelet $\psi_{s,a}^{\fdeg,\nu,n}$ is
  \begin{align*}
    \Delta \psi_{s,a}^{\fdeg,\nu,n} = \mathfrak{F}^{-1}\big( (-1)^{\vert A_{n,r,\nu} \vert} \vert \xi \vert^2 \, \hat{\psi}_{s,a}^{\fdeg,\nu,n} \big) .
  \end{align*}
  Furthermore, $\Delta \psi_{s,a}^{\fdeg,\nu,n}$ is given by the expression for the differential forms wavelets in Theorem~\ref{thm:psiforms:frames} with the modified radial window, $h_l^{(\cdot + 2)}(r)$, obtained through the (spherical) Hankel transform with the additional $\vert \xi \vert^2$-weight.
\end{proposition}

The second part of the proposition follows from the fact that the closed form expressions in Theorem~\ref{thm:psiforms:frames} are computed using the (inverse) Fourier transform in spherical coordinates, see the proof of the theorem.

\begin{remark}
  \label{rem:laplace_derahm_galerkin}
  Proposition~\ref{prop:laplace_derahm:psi} shows that the symbol $\hat{\Delta}$ of the Laplace--de Rahm operator $\Delta = \dd \delta + \delta \dd$ is, up to a sign, $\vert \xi \vert^2$ and this holds not only for differential form wavelets but for differential forms in general.
\end{remark}

\begin{figure}
  \centering
  \includegraphics[width=0.8\textwidth]{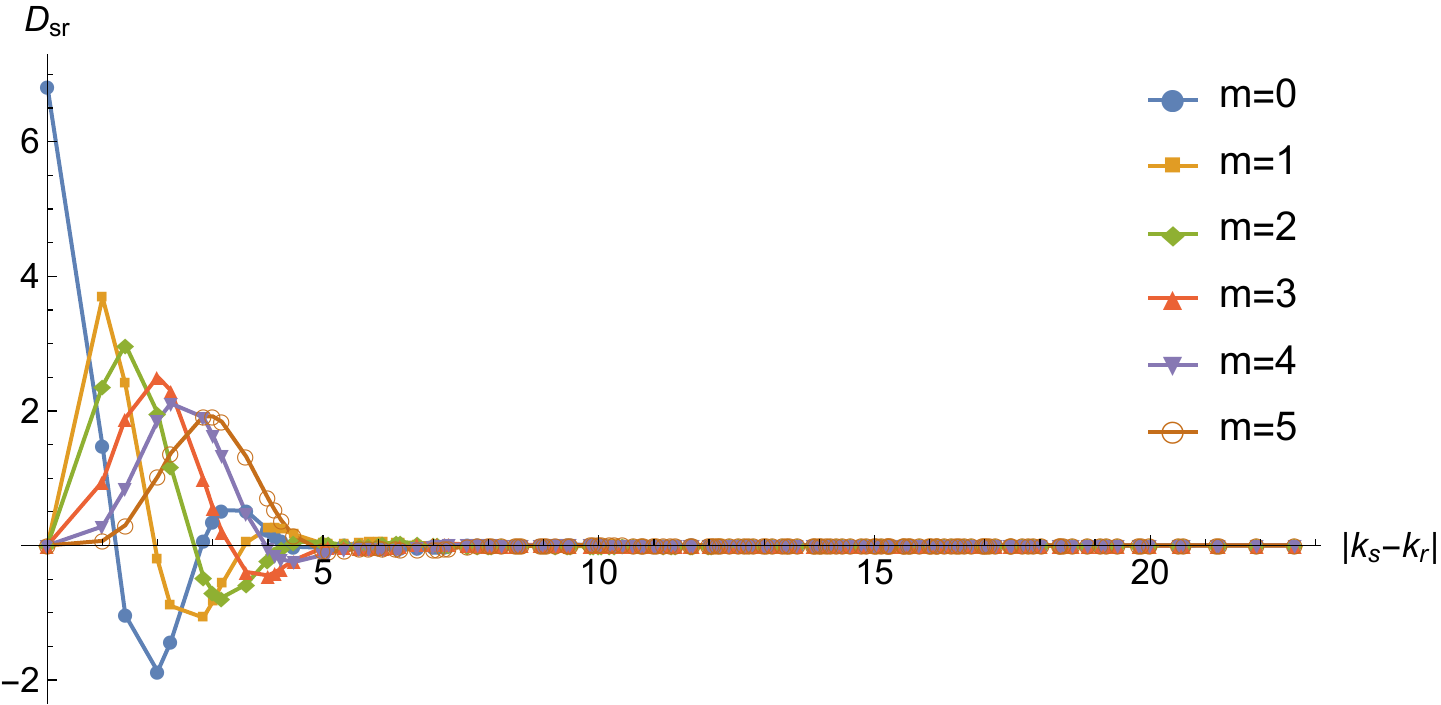}
  \caption{Decay of the Galerkin projection of the Laplace--de Rahm operator $D_{qp}$, as a function of the separation $\vert k_q - k_p \vert$ in $\R^2$. The decay for isotropic wavelets corresponds to $m=0$. For anisotropic wavelets one has a linear combination of the different $m$. Note that the decay does not depend on the wavelet type since our differential form wavelets use the same frequency window independent of the type of the form.}
  \label{fig:laplace_decay_spatial}
\end{figure}

Important for numerical calculations is the Galerkin projection,
\begin{align*}
  D_{q p } = \left\langle \! \left\langle \Delta \psi_{q,a}^{\fdeg,\nu,n} \, , \, \psi_{p,a}^{\fdeg,\nu,n} \right\rangle \! \right\rangle ,
\end{align*}
of the Laplace--de Rahm operator.
With the definition of our differential form wavelets in the Fourier domain and Proposition~\ref{prop:laplace_derahm:psi} it has a closed form solution.
Furthermore, by the compact support of our wavelets in the frequency domain the matrix elements $D_{q p }$ are nonzero only when $\min( -1, j_p -1 ) \leq j_q \leq j_p + 1$ and when the orientations align, i.e. $t_q \approx t_p$ (assuming compactly supported windows directional windows $\hat{\gamma}(\bar{\xi})$; otherwise nonzero has to be replaced by non-negligible).
The coefficients $D_{qp}$ as a function of the separation $\vert k_q - k_p \vert$ are shown in Fig.~\ref{fig:laplace_decay_spatial} where it can be seen that these decay fast as $\vert k_q - k_p \vert$ grows.
The discrete Laplace--de Rahm operator $D_{qp}$ is hence well localized in space and frequency by our construction.

\begin{example}
  \label{ex:laplace_de_rahm:cavity}
  One of the first applications where the importance of differential forms for numerical modeling became clear was electromagnetic theory~\cite{Bossavit1997,Hiptmair2002}.
  The resonant cavity, which we consider in this example, serves thereby as a reference problem~\cite{Bell2012}.
  It seeks the electromagnetic field in a simple domain consisting of a perfect conductor containing no enclosed charges.
  Maxwell's equations then reduce to an eigenvalue problem for which na{\"i}ve vector calculus discretizations fail to provide the correct answer~\cite{Bossavit1997,Hiptmair2002,Bell2012}.

  Let $\mathcal{B} \subset \R^2$ be the square domain with side length $\pi$.
  The resonant cavity problem then seeks the electric $1$-form field $E \in \Omega_{\delta}^1(\mathcal{B})$ such that
  \begin{subequations}
    \label{eq:cavity_problem}
  \begin{align}
    \Delta E = \delta \dd E &= \lambda E \quad \mathrm{in } \ \mathcal{B}
    \\
    i^* E &= 0 \quad \mathrm{on } \ \partial \mathcal{B}
  \end{align}
  \end{subequations}
  where the pullback $i^* E = 0$ describes that the component of $E$ tangential to the boundary vanishes and $E$ is co-exact since the electric field is divergence free in a region free of charges.

  To be able to apply the polar differential form wavelets constructed for $\R^2$ to the above cavity problem on $\mathcal{B}$, we consider $E$ as a differential form in $\R^2$, analogous to the treatment for Stokes' theorem in Sec.~\ref{sec:psiforms:stokes}.
  This means that the coordinate functions are multiplied by the characteristic function $\chi_{_{\mathcal{B}}}$ and, by abuse of notation, we will not distinguish $E$ and its embedding into $\R^2$.
%  Representing $E$ then in the co-exact $1$-form wavelets, Eq.~\ref{eq:cavity_problem} becomes
  Galerkin projection of Eq.~\ref{eq:cavity_problem} using the co-exact $1$-form wavelets then yields
  \begin{align*}
    \Big\langle \! \! \Big\langle \psi_{p}^{1,\delta} \, , \, \sum_{q} E_q \, \Delta \, \psi_{q}^{1,\delta} \Big\rangle \! \! \Big\rangle
    &= \lambda \big\langle \! \big\langle \psi_{p}^{1,\delta} \, , \, E \big\rangle \! \big\rangle
    \\[1pt]
     \sum_{q} E_q  \, \big\langle \! \big\langle \psi_{q}^{1,\delta} \, , \, \Delta \, \psi_{p}^{1,\delta} \big\rangle \! \big\rangle
    &= \lambda \, E_p
    \\[1pt]
    \sum_{q} E_q  \, D_{q p}    &= \lambda \, E_p
  \end{align*}
  where $D_{qp}$ is the discrete Laplace--de Rahm operator from Remark~\ref{rem:laplace_derahm_galerkin}.
%  As before, these can be determined in closed form in the Fourier domain.
  By linearity of the pullback, the boundary condition becomes in the frame representation
  \begin{align*}
    \sum_{q} E_q \, i^* \psi_{q}^{1,\delta} &= 0 \quad \mathrm{on } \ \partial \mathcal{B} .
  \end{align*}
  It is hence enforced when the differential form wavelets $\psi_{s}^{1,\delta}$ satisfy $i^* \psi_{s}^{1,\delta} = 0$, i.e. when the pairing of $\psi_{s}^{1,\delta}$ with the normal vanishes.
%  for all $\psi_{s}$ that are used to represent the solution $E$, see also the discussion on Stokes' theorem in Sec.~\ref{sec:exterior_algebra:fourier} where the same pullback arose.
  Fig.~\ref{fig:psi_delta_22}, right columns in the bottom row, shows that suitably constructed anisotropic polar differential $1$-form wavelets  $\psi_{s}^{1,\delta}$ satisfy this to good approximation.
  These can hence be used to model the boundary conditions.

  A numerical investigation of this example is, however, warranted.
  For example, in practice the boundary conditions do not hold perfectly with the above approach in the finite-dimensional, numerically practical case and there is always leakage of the electric field $E$, which was embedded in $\R^2$, into the exterior domain $\R^2 \setminus \mathcal{B}$.

%  It has to be noted that, compared to the discretizations of Eq.~\ref{ex:laplace_de_rahm:cavity} in the literature~\cite{Bossavit1997,Bell2012}, no weak formulation is required in our case since our wavelet differential forms have sufficient regularity.

\end{example}

%%%%%%%%%%%%%%%%%%%%%%%%%%%%%%%%%%%%%%%%%%%%%%%%%%%%%%%%%%%%%%%%%%%%%%%%%%%%%%%%%
\subsection{Other Operators on Differential Form Wavelets}

In the following we will briefly discuss other important operations on differential forms where we do not yet have an elegant realization in $\Psiec$.

%%%%%%%%%%%%%%%%%%%%%%%%%%%%%%%%%%%%%%%%%%%%%%%%%%%%
\subsubsection{Wedge product}

As mentioned before, polar differential form wavelets $\psi_{s,a}^{\fdeg,\nu,n}$ are forms in the sense of the continuous theory.
Hence, the wedge product $\psi_{s,a}^{\fdeg,\nu,n}\, \wedge \, \psi_{s',a'}^{l,\nu',n}$ is well defined and a $(k+l)$-form.
For numerical calculations one would hope that the nonzero coefficients of the product are sparse, at known locations, and can be computed efficiently.
As shown in Proposition~\ref{prop:exterior:convolution}, the wedge product becomes a convolution in frequency space.
Hence, considerable sparsity is lost and precise conclusions about the number and location of nonzero coefficients are not easily established (although the compact support of the wavelets in the frequency domain does provide bounds).

In practice we currently use the closed form formulas for the spatial representations of the differential form wavelets in Theorem~\ref{thm:psiforms:exterior}.
With these, the multiplication of the coordinate functions that arises as part of the wedge product can be computed analytically.
Although the resulting expressions are rather complicated, they can be combined with numerical quadrature to implement the wedge product.
However, we currently have no insight into the sparsity of the coefficients of the product or their decay properties.
An alternative to the above approach would be to use the fast transform method~\cite{Eliasen1970,Orszag1970}, i.e. to evaluate the multiplication in the spatial domain.
For this, a fast transform is required, which is hence an important direction for future work.

%%%%%%%%%%%%%%%%%%%%%%%%%%%%%%%%%%%%%%%%%%%%%%%%%%%%
\subsubsection{Lie derivative}

The Lie derivative $\mathcal{L}_{X} \alpha$ describes the infinitesimal transport of a differential form $\alpha$ along the vector field $X$.
It hence plays a fundamental role in the description of physical systems using exterior calculus.
Using Cartan's formula, the Lie derivative can be written as
\begin{align*}
  \mathcal{L}_{X} \alpha = \dd \ip_X \alpha + \ip_X \dd \alpha .
\end{align*}
% COMMENT:
% The interior product can be split into a radial and a spherical part. The Fourier transform of the radial part is the exterior derivative. It seems there should also be a away to deal with the tangential one.
While the exterior derivative is well defined for differential form wavelets, the interior product with an arbitrary vector field, which amounts to a multiplication of the coordinate functions, is not.
An efficient, general evaluation of the Lie derivative is in $\Psiec$ hence currently not available.
The next example demonstrates that, at least in certain cases, one can resort to the convenient analytic expressions of the differential form wavelets to evaluate the Lie derivative.
%Similar to the wedge product, the multiplication of the coordinate functions has a closed form solution in the spatial domain, and this can be used to implement it in practice.

\begin{example}
  \label{ex:lie_advection:euler}
  The Euler fluid equation in vorticity form is given by~\cite[Ch. 1]{Khesin1998}
  \begin{align*}
    \frac{\partial \zeta}{\partial t} - \mathcal{L}_{\vec{u}} \, \zeta = 0
  \end{align*}
  where $\vec{u} \in \mathfrak{X}_{\mathrm{div}}(\R^n)$ is the divergence free velocity vector field and $\zeta = \dd u^{\flat}$ is the vorticity, cf. Example~\ref{ex:kelvins_circulation_theorem}.
  In $\R^2$, vorticity is a volume form, i.e. $\zeta \in \Omega_{\dd}^2(\R^2)$, and using Cartan's formula we hence have
  \begin{align*}
     \mathcal{L}_{\vec{u}} \, \zeta = \dd \, \ip_{\vec{u}} \, \zeta + \ip_{\vec{u}}\, \dd \, \zeta = \dd \, \ip_{\vec{u}} \zeta .
  \end{align*}
  In coordinates this equals
  \begin{align*}
    \mathcal{L}_{\vec{u}} \, \zeta
%    &= \dd \! \left( \ip_{\vec{u}} \, \zeta \, \dd x^1 \wedge \dd x^2 - \ip_{\vec{u}} \, \zeta \, \dd x^2 \wedge \dd x^1 \right)
%    \\[5pt]
%    &= \dd \! \left( u_1 \, \zeta \, \dd x^1 - u_2 \, \zeta \, \dd x^2 \right)
%    \\[5pt]
%    &= \left( \frac{\partial u_1}{\partial x^2} \zeta + u_1 \frac{\partial \zeta}{\partial x^2} \right) \dd x^2 \wedge \dd x^1 - \left( \frac{\partial u_2}{\partial x^1} \zeta + u_2 \frac{\partial \zeta}{\partial x^1} \right) \dd x^1 \wedge \dd x^2
%    \\[5pt]
%      \label{eq:ex:lie_advection:euler:5}
    &= -\left( \frac{\partial u_1}{\partial x^2} \zeta + u_1 \frac{\partial \zeta}{\partial x^2} + \frac{\partial u_2}{\partial x^1} \zeta + u_2 \frac{\partial \zeta}{\partial x^1} \right) \dd x^1 \wedge \dd x^2 .
  \end{align*}
  Representing $\vec{u}$ with the divergence free wavelets from~\cite{Lessig2018z}, which we denote by $\vec{\psi}_s$,
  and $\zeta$ with the $2$-form wavelets $\psi_{s}^{2,\dd}$ we obtain
  \begin{align*}
    \frac{\partial \zeta_q}{\partial t} = \sum_{s} \sum_r u_s \, \zeta_r \left\langle  \mathcal{L}_{\vec{\psi}_s} \psi_{r}^{2,\dd} \ , \ \psi_{q}^{2,\dd} \right\rangle .
  \end{align*}
  The advection or interaction coefficients (cf.~\cite{Platzman1960})
  \begin{align*}
    C_{sr}^q = \left\langle  \mathcal{L}_{\vec{\psi}_s} \psi_{r}^{2,\dd} \ , \ \psi_{q}^{2,\dd} \right\rangle
  \end{align*}
  can then be computed with a closed form expression%using Eq.\ref{eq:ex:lie_advection:euler:5}
  , see the supplementary material, and implementing the reprojection onto $\psi_{q}^{2,\dd}$ using a quadrature rule, cf.~\cite{Lessig2018z}.
\end{example}

%%%%%%%%%%%%%%%%%%%%%%%%%%%%%%%%%%%%%%%%%%%%%%%%%%%%
\subsubsection{Pullback}

The pullback $\eta^* \alpha \in \Omega^r(\mathcal{N})$ of a differential form $\alpha \in \Omega^r(\mathcal{M})$ by a map $\eta : \mathcal{N} \to \M$ between two manifolds $\mathcal{N},\M$ is a another important operation in the exterior calculus.
%, in particular since it commutes with the exterior derivative, i.e. $\dd \eta^* \alpha = \eta^* \dd \alpha$.
For linear maps, such as rotation or shear, closed form solutions for the pullback of a differential form wavelet can be derived in the Fourier domain~\cite{Azencott2009}, cf. Theorem~\ref{thm:psiforms:exterior}, iv.).
For general diffeomorphisms a result by Cand{\`e}s and Demanet~\cite[Thm. 5.3]{Candes2005} shows that curvelet-like frames are essentially preserved.
It would be interesting to extend this result into a numerically practical form.
A special case that is of particular relevance are volume preserving diffeomorphisms, which, e.g., describe the time evolution of inviscid fluids.
Then the map $\eta$ can be associated with a unitary operator~\cite{Koopman1931} and Galerkin projection can be used to obtain the representation of $\eta$ in the wavelet domain.

%%%%%%%%%%%%%%%%%%%%%%%%%%%%%%%%%%%%%%%%%%%%%%%%%%%%%%%%%%%%%%%%%%%%%%%%%%%%%%%%%%%%%%%%%%%%%%%%%%%%%%%%%%
\subsection{Proofs}
\label{sec:psiforms:proofs}

\begin{proof}[Proof of Theorem~\ref{thm:psiforms:frames}]

  We begin with the frame property.
  The cases of $0$- and $n$-forms are equivalent to the existing results in the literature~\cite{Unser2011,Unser2013,Ward2014}; for $\R^3$ we also provide an alternative, more direct proof in~\ref{sec:polarlets:admissibility:3d}.
  In $\R^2$, for $1$-form one can use the argument from~\cite[Proposition~1]{Lessig2018z}, i.e. a short calculation reduces the problem to the known scalar case.
  This argument carries over to exact forms in $\R^3$, for which we detail the calculations for $L_2(\Omega_{\dd}^2,\R^3)$ below.
  For co-exact forms, which span the homogeneous Sobolev space $\dot{L}_{2}^1(\Omega_{\delta}^r,\R^n)$, one can then use the Hodge-Helmholtz decomposition, as we will also explain.

  \paragraph{Exact forms in $\R^3$}
  \begin{subequations}
  Let $\alpha \in L_2(\Omega_{\dd}^2,\R^3)$ be an exact $2$-form.
  By Theorem~\ref{thm:hodge_helmholtz}, its Fourier transform is given by
  \begin{align}
    \label{eq:proof:thm:psiforms:frames:1}
    \widehat{\alpha} = \widehat{\alpha}_{\theta} \frac{\partial}{\partial \hat{\theta}} + \widehat{\alpha}_{\phi} \frac{\partial}{\partial \hat{\phi}}
  \end{align}
  We want to show that
  \begin{align}
    \label{eq:proof:thm:psiforms:frames:2}
    \alpha = \sum_{s \in \mathcal{S}} \left\langle \! \! \left\langle \alpha \, , \, \psi_{s,1}^{2,\dd} \right\rangle \! \! \right\rangle \psi_{s,1}^{2,\dd} + \sum_{s \in \mathcal{S}} \left\langle \! \! \left\langle \alpha \, , \, \psi_{s,2}^{2,\dd} \right\rangle \! \! \right\rangle \psi_{s,2}^{2,\dd} .
  \end{align}
  Taking the Fourier transform and using Parseval's theorem for differential forms we obtain
  % The \hat{\psi}_{s,1}^{2,\dd} here are those that include the normalization
  \begin{align}
    \label{eq:proof:thm:psiforms:frames:3}
    \widehat{\alpha}
    &= \sum_{s \in \mathcal{S}} \left\langle \! \! \left\langle \widehat{\alpha} \, , \, \hat{\psi}_{s,1}^{2,\dd} \right\rangle \! \! \right\rangle \hat{\psi}_{s,1}^{2,\dd} + \sum_{s \in \mathcal{S}} \left\langle \! \! \left\langle \widehat{\alpha} \, , \, \hat{\psi}_{s,2}^{2,\dd} \right\rangle \! \! \right\rangle \hat{\psi}_{s,2}^{2,\dd} .
  \end{align}
  Considering for the moment only the first term and expanding it yields
  \begin{align*}
    \sum_{s \in \mathcal{S}} \left\langle \! \! \left\langle \widehat{\alpha} \, , \, \hat{\psi}_{s,1}^{2,\dd} \right\rangle \! \! \right\rangle \hat{\psi}_{s,1}^{2,\dd}
    = \sum_{s \in \mathcal{S}} \left\langle \! \! \! \left\langle \widehat{\alpha}_{\theta} \frac{\partial}{\partial \hat{\theta}} + \widehat{\alpha}_{\phi} \frac{\partial}{\partial \hat{\phi}} \, , \, \hat{\psi}_{s}^*(\xi) \, \frac{\partial}{\partial \hat{\theta}}  \right\rangle \! \! \! \right\rangle \hat{\psi}_{s,1}^{2,\dd} .
    \intertext{Writing out the inner product with the Hodge dual we obtain}
    \sum_{s \in \mathcal{S}} \left\langle \! \! \left\langle \widehat{\alpha} \, , \, \hat{\psi}_{s,1}^{2,\dd} \right\rangle \! \! \right\rangle \hat{\psi}_{s,1}^{2,\dd}
    = \sum_{s \in \mathcal{S}} \int_{\widehat{\R}^3} \left( \widehat{\alpha}_{\theta} \frac{\partial}{\partial \hat{\theta}} + \widehat{\alpha}_{\phi} \frac{\partial}{\partial \hat{\phi}} \right) \! \wedge \! \left( \hat{\psi}_{s}^*(\xi) \, \frac{\partial}{\partial \hat{\phi}} \wedge  \frac{\partial}{\partial \hat{r}} \right) \hat{\psi}_{s,1}^{2,\dd} .
    \nonumber
    \end{align*}
    The integral is nonzero only for the first term of $\alpha$, since only then one obtains a volume form.
    Hence
    \begin{align*}
      \sum_{s \in \mathcal{S}} \left\langle \! \! \left\langle \widehat{\alpha} \, , \, \hat{\psi}_{s,1}^{2,\dd} \right\rangle \! \! \right\rangle \hat{\psi}_{s,1}^{2,\dd}
    &= \sum_{s \in \mathcal{S}} \int_{\widehat{\R}^3} \! \! \left( \widehat{\alpha}_{\theta}(\xi) \, \hat{\psi}_{s}^*(\xi) \, \frac{\partial}{\partial \hat{\theta}} \wedge \frac{\partial}{\partial \hat{\phi}} \wedge  \frac{\partial}{\partial \hat{r}} \right) \hat{\psi}_{s,1}^{2,\dd} .
  \end{align*}
  The integral is now the $L_2$-inner product for scalar functions.
  Also expanding $\hat{\psi}_{s,1}^{2,\dd}$ on the right hand side of the last equation we can write
  \begin{align}
    \label{eq:proof:thm:psiforms:frames:7}
    \sum_{s \in \mathcal{S}} \left\langle \! \! \left\langle \widehat{\alpha} \, , \, \hat{\psi}_{s,1}^{2,\dd} \right\rangle \! \! \right\rangle \hat{\psi}_{s,1}^{2,\dd}
    &= \underbrace{\sum_{s \in \mathcal{S}} \left\langle  \widehat{\alpha}_{\theta}(\xi) \, , \, \hat{\psi}_{s}(\xi) \right\rangle \hat{\psi}_{s}(\xi)}_{\textrm{scalar polar wavelet expansion}} \, \frac{\partial}{\partial \hat{\theta}} .
  \end{align}
  The term with the brace is equivalent to the case of $0$-forms, i.e. the known scalar result applies for the representation of the coordinate function and hence the frame property holds for the first part of $\hat{\alpha}$ in Eq.~\ref{eq:proof:thm:psiforms:frames:1}.
  Analogously, one obtains for the second term in Eq.~\ref{eq:proof:thm:psiforms:frames:3} that
  \begin{align*}
  \sum_{s \in \mathcal{S}} \left\langle \! \! \left\langle \widehat{\alpha} \, , \, \hat{\psi}_{s,1}^{2,\dd} \right\rangle \! \! \right\rangle \hat{\psi}_{s,1}^{2,\dd}
    &= \sum_{s \in \mathcal{S}} \int_{\widehat{\R}^3} \left( i \, \widehat{\alpha}_{\phi}(\xi) \, \hat{\psi}_{s}(\xi) \, \frac{\partial}{\partial \hat{\phi}} \wedge \frac{\partial}{\partial \hat{\theta}} \wedge  \frac{\partial}{\partial \hat{r}} \right) \hat{\psi}_{s,1}^{2,\dd}
    \\[7pt]
%    \label{eq:proof:thm:psiforms:frames:9}
    &= \underbrace{\sum_{s \in \mathcal{S}} \left\langle  \widehat{\alpha}_{\phi}(\xi) \, , \, \hat{\psi}_{s}(\xi) \right\rangle \hat{\psi}_{s}(\xi)}_{\textrm{scalar polar wavelet expansion}} \, \frac{\partial}{\partial \hat{\phi}} .
  \end{align*}
  \end{subequations}
%  The signs in Eq.~\ref{eq:proof:thm:psiforms:frames:7} and Eq.~\ref{eq:proof:thm:psiforms:frames:7} cancel with those in Eq.~\ref{eq:proof:thm:psiforms:frames:3} and.
%   Eq.~\ref{eq:proof:thm:psiforms:frames:2} thus follows from the scalar theory for polar wavelets.
%  For different $\fdeg$ and $\R^2$ the results can be established using analogous calculations.
%  This shows Proposition~\ref{thm:psiforms:frames} for the case of exact forms.

  \paragraph{Co-exact forms in $\R^3$}
%  \red{There's a direct proof. Should we rather present it? It seems to fit more natural (the rather pedestrian) nature of the rest of the paper}
  The case of co-exact forms can be handled analogously to the above one by reducing it to a known result for the scalar case.
  \commentCL{
    The result to be used is
    \begin{align}
      \hat{f}(\xi)
      &= \sum_{s} \langle \hat{f}(\eta) , \hat{\psi}_s(\eta) \rangle \hat{\psi}_s(\xi)
      \\[4pt]
      &=  \langle \hat{f}(\eta) , \sum_{s} \hat{\psi}_s(\eta) \hat{\psi}_s(\xi) \rangle .
    \end{align}
    Since it is known that the equality has to hold, e.g.~\cite{Unser2013}, this implies that, in an appropriate sense,
    \begin{align}
      \delta(\xi - \eta) = \sum_{s} \hat{\psi}_s(\eta) \hat{\psi}_s(\xi) .
    \end{align}
    The case of co-exact forms can, using manipulations analogous to those for exact ones, be reduced to a point there this can be used and this cancels the second $\vert \xi \vert$ in the definition of the homogenous Sobolev space.
  }
  We present here an alternative proof where we exploit the correspondence between co-exact $\fdeg$-forms and exact $(\fdeg+1)$ forms that is established by the Hodge-Helmholtz decomposition.
  Let $\Psi_{}^{\fdeg,\nu,a}$ be the frame operator associated with the $\psi_{s,a}^{\fdeg,\nu}$ for fixed $\nu$, $\fdeg$, $a$.
  The Gramian is $G^{\fdeg,\nu,a} = \Psi^{\fdeg,\nu,a} (\Psi^{\fdeg,\nu,a})^*$ and its entries are
  \begin{align*}
    G_{q p}^{\fdeg,\nu,a}
    = \left\langle \! \left\langle \psi_{q,a}^{\fdeg,\nu,n} \, , \, \psi_{p,a}^{\fdeg,\nu,n} \right\rangle \! \right\rangle_{\dot{L}_2^{\vert \nu \vert}} .
  \end{align*}
  Using the definition of the $\dot{L}_2^1$ inner product for differential forms and that $\dd \psi_{q,a}^{r,\delta,n} = \psi_{q,a}^{r+1,\dd,n} $ we have for co-exact forms that
  \begin{align*}
    G_{qp}^{\fdeg,\delta,a} = \left\langle \! \left\langle \psi_{q,a}^{\fdeg,\delta,n} \, , \, \psi_{p,a}^{\fdeg,\delta,n} \right\rangle \! \right\rangle_{\dot{L}_2^1}
    = \left\langle \! \left\langle \psi_{q,a}^{\fdeg+1,\dd,n} \, , \, \psi_{p,a}^{\fdeg+1,\dd,n} \right\rangle \! \right\rangle_{L_2} = G_{qp}^{\fdeg+1,\dd,a} ,
  \end{align*}
  i.e. the Gramians for co-exact $\fdeg$-form frame functions and exact $(\fdeg+1)$-form ones are identical.
  The lower and upper frame bounds, $A$ and $B$, respectively, of a frame can be characterized using the Gramian as
  \begin{subequations}
  % http://www.ens-lyon.fr/DI/wp-content/uploads/2011/10/lecture-1-2-frames.pdf
  \begin{align*}
    A &= \mathrm{inf} \! \left\{ \left\langle a , G a \right\rangle_{\ell_2}, \ a \in \mathrm{ran}( \Psi ), \ \vert a \vert = 1 \right\}
    \\[5pt]
    B &= \mathrm{sup} \! \left\{ \left\langle a , G a \right\rangle_{\ell_2}, \ a \in \mathrm{ran}( \Psi ), \ \vert a \vert = 1 \right\} .
  \end{align*}
  \end{subequations}
  Since the Gramians are identical and by the Hodge-Helmholtz decomposition $\dd ( \mathrm{ran}( \Psi^{\fdeg,\delta} )) = \mathrm{ran}( \Psi^{\fdeg+1,\dd} )$, the tightness of the frame for exact $(\fdeg+1)$-forms implies those for co-exact $\fdeg$-forms.

  \paragraph{Spatial representation}

%    That the differential form wavelets are real-valued follows from the usual symmetry properties in the frequency representation and can be seen using the explicit calculation of the inverse Fourier transforms.
%    For these, we use the Jacobi-Anger and Rayleigh formulas that describe the complex exponential in polar and spherical coordinates, respectively, see \ref{sec:preliminaries:fourier:polar}.
    We will demonstrate the computation of the spatial representation of the differential form wavelets for one example in $\R^2$ and one in $\R^3$.
    An important aspect of the computations is the use of the Jacobi-Anger and Rayleigh formulas that describe the scalar complex exponential in polar and spherical coordinates, respectively, see \ref{sec:preliminaries:fourier:polar}.

    Without loss of generality, we consider the mother wavelet $\psi^{1,\dd}(x) \in L_2( \Omega_{\dd}^2,\R^2)$ whose definition in the Fourier domain is given by
    \begin{align*}
      \hat{\psi}^{1,\dd}(\xi)
      &= i \sum_m \beta_m \, e^{i m \hat{\theta}} \, \hat{h}(\vert \xi \vert) \frac{\partial}{\partial \hat{\theta}} .
      \intertext{Expanding the form basis functions in polar coordinates into the Cartesian ones we obtain}
      \hat{\psi}^{1,\dd}(\xi) &= i \sum_m \beta_m \, e^{i m \hat{\theta}} \, \hat{h}(\vert \xi \vert) \left(  \sin{\hat{\theta}} \dpp{}{\xi^1} - \cos{\hat{\theta}} \dpp{}{\xi^2} \right) .
    \end{align*}
    The trigonometric functions can be written as Fourier series.
    The Cartesian components $\hat{\psi}_1^{1,\dd}(\xi)$ and $\hat{\psi}_2^{1,\dd}(\xi)$ of $\hat{\psi}^{1,\dd}(\xi)$ then become
    \begin{align*}
      \hat{\psi}_1^{1,\dd}(\xi)
      &= -\frac{1}{2} \left( e^{-i \hat{\theta}} - e^{i \hat{\theta}} \right) \left( \sum_{n} \beta_n \, e^{i n \hat{\theta}} \right) \, \hat{h}(\vert \xi \vert) \, \dpp{}{\xi^1}
      \\
      &= -\frac{1}{2} \left( \sum_{n} \beta_n \, e^{i (n-1) \hat{\theta}} \, \hat{h}(\vert \xi \vert) - \sum_{n} \beta_n \, e^{i (n+1) \hat{\theta}} \, \hat{h}(\vert \xi \vert)  \right) \dpp{}{\xi^1} \nonumber
      \intertext{and}
      \hat{\psi}_2^{1,\dd}(\xi)
      &= -\frac{i}{2} \left( e^{-i \hat{\theta}} + e^{i \hat{\theta}} \right) \left( \sum_{n} \beta_n \, e^{i n \hat{\theta}} \right) \, \hat{h}(\vert \xi \vert) \, \dpp{}{\xi^2}
      \\
       &= -\frac{i}{2} \left( \sum_{n} \beta_n \, e^{i (n-1) \hat{\theta}} \, \hat{h}(\vert \xi \vert) + \sum_{n} \beta_n \, e^{i (n+1) \hat{\theta}} \, \hat{h}(\vert \xi \vert) \right) \dpp{}{\xi^2} \nonumber
    \end{align*}
    To obtain the inverse Fourier transform of the $1$-forms above we require those of each of the sums in the above formulas.
    Using the definition of the inverse Fourier transform for differential forms in Def.~\ref{def:diff_form:ft} we have, for example,
    \begin{align*}
      &\mathfrak{F}^{-1}\left( -\sum_{n} \beta_n \, e^{i (n-1) \hat{\theta}} \, \hat{h}(\vert \xi \vert) \, \dpp{}{\xi^1} \right)
      \\[5pt]
      &= \frac{1}{2\pi} \int_{\widehat{\R}^2} \left( \sum_{n} \beta_n \, e^{i (n-1) \hat{\theta}} \, \hat{h}(\vert \xi \vert) \, \dpp{}{\xi^1} \right) \wedge e^{i \langle x , \xi \rangle} \left( \dpp{}{\xi^2} \, \dd x^2 \right)
    \end{align*}
    where we immediately only introduced the term of the exponential $e^{\dpp{}{\xi^q} \dd x^q}$ that yields a volume form.
    Changing back to polar coordinates and using the Jacobi-Anger formula to write the scalar complex exponential in polar coordinates we have,
    \begin{align*}
      &\mathfrak{F}^{-1}\left( -\sum_{n} \beta_n \, e^{i (n-1) \hat{\theta}} \, \hat{h}(\vert \xi \vert) \, \dpp{}{\xi^1} \right)
      \\[5pt]
      &= \frac{1}{2\pi} \int_{\widehat{\R}^+} \int_{0}^{2\pi} \sum_{n} \beta_n \, e^{i (n-1) \hat{\theta}} \, \hat{h}(\vert \xi \vert) \, \sum_{m \in \Z} i^m e^{i m (\theta - \hat{\theta})} \, J_m(\vert \xi \vert \, \vert x \vert) \, \vert \xi \vert \dd \vert \xi \vert \, \dd\hat{\theta} \, \dd x^2
      \\[5pt]
      &= \frac{1}{2\pi} \sum_n \sum_{m \in \Z} i^m \, \beta_n \, e^{i m \theta} \int_{0}^{2\pi} \! e^{i (n-1) \hat{\theta}} \, e^{-i m \hat{\theta}} \, \dd\hat{\theta} \int_{\widehat{\R}^+} \! \hat{h}(\vert \xi \vert) \, J_m(\vert \xi \vert \, \vert x \vert) \, \vert \xi \vert \dd \vert \xi \vert \, \dd x^2
      \\[5pt]
      &= \sum_{n} i^{n-1} \, \beta_{n} \, e^{i (n-1) \theta} \, h_m(\vert x \vert) \, \dd x^2.
    \end{align*}
    Carrying this out for all terms and using the Fourier transform of the form basis functions in Eq.~\ref{eq:diff_form:ft:r2} we obtain
    \begin{align*}
      \psi_1^{1,\dd}(x) &= \frac{1}{2} \sum_{\sigma \in \{ -1, 1\} } \sum_{m} \sigma \, i^{m + \sigma} \beta_m \, e^{i (m + \sigma) \theta} \, h_{m + \sigma}(\vert x \vert) \, \dd x^2
      \\[5pt]
      \psi_2^{1,\dd}(x) &= \frac{i}{2} \sum_{\sigma \in \{ -1, 1\} } \sum_{m} i^{m + \sigma} \beta_m \, e^{i (m + \sigma) \theta} \, h_{m + \sigma}(\vert x \vert) \, \dd x^1
    \end{align*}
    Rewriting this as differential form $\psi^{1,\dd} = \psi_2^{1,\dd}(x) + \psi_1^{1,\dd}(x)$ yields the result in Theorem~\ref{thm:psiforms:exterior}.
    That the wavelets are real-valued follows from the interaction of $\sigma \in \{-1,1\}$ with the factor $i^{m+\sigma}$ that is introduced by the Jacobi-Anger formula.
    When $\beta_m = \delta_{m0}$ one has $\psi^{1,\dd} = h_1(\vert x \vert) (\cos{\theta} \dd x^1 + \sin{\theta} \dd x^2) = h_1(\vert x \vert) \, \dd r$, i.e. a radial differential $1$-form field in space.

  In $\R^3$, the calculations proceed analogously to $\R^2$ by first representing the spherical form basis functions in Cartesian ones, e.g. $\partial / \partial \hat{\phi} = \hat{\phi}^j \partial / \partial \xi^j$ with $\hat{\phi}^j = \hat{\phi}^j(\hat{\theta},\hat{\phi})$, and then using the results of Sec.~\ref{sec:fourier_forms}.
  We will present a detailed calculation for the $1$-form basis function $\psi_{s,2}^{1,\dd,3}$.

  The wavelet is in the frequency domain given by
  \begin{align*}
    \hat{\psi}_{s,2}^{1,\dd,3}
    = -\frac{1}{\vert \xi \vert} \hat{\psi}_s(\xi) \frac{\partial}{\partial \hat{\phi}} \wedge \frac{\partial}{\partial \hat{r}}
    = i \hat{\psi}_s(\xi) \left( \sum_{j_1 < j_2} \hat{\phi}^{j_1}  \frac{\partial}{\partial \xi^{j_1}} \wedge  \hat{r}^{j_2} \frac{\partial}{\partial \xi^{j_2}} \right)
  \end{align*}
  By linearity and with the representation of the scalar wavelet from Eq.~\ref{eq:polarlets:3d:hat} we obtain
  \begin{align*}
    \hat{\psi}_{s,2}^{1,\dd,3}
    = \frac{-2^{3j/2}}{(2\pi)^{3/2} \vert \xi \vert} \sum_{l,m} \kappa_{lm}^{j,t} \, y_{lm} \, \hat{h}(2^{-j} \vert {\xi} \vert) \, \left(  \sum_{j_1 < j_2} \hat{\phi}^{j_1} \hat{r}^{j_2} \frac{\partial}{\partial \xi^{j_1}} \wedge \frac{\partial}{\partial \xi^{j_2}} \right) .
  \end{align*}
  where, without loss of generality, we omitted the translation factor for notational convenience, i.e. we consider $s = (j,0,t)$.
  The two terms $\kappa_{lm}^{j,t} y_{lm}$ and $\hat{\phi}^{j_1} \hat{r}^{j_2}$ that depend on the angular variables $(\hat{\theta},\hat{\phi})$ can be combined using the representation of the product in the spherical harmonics domain, cf.~\ref{sec:preliminaries:sh}.
  This yields
  \begin{align*}
    \Big( -\sum_{l,m} \kappa_{lm}^{j,t} \, y_{lm} \Big) \! \cdot \! (\hat{\phi}^{j_1} \hat{r}^{j_2}) = \sum_{lm} \underbrace{\left( -\sum_{l_1,m_2} \sum_{l_2,m_2} \kappa_{l_1 m_1}^{j,t} \, (\hat{\phi}^{j_1} \hat{r}^{j_2})_{l_2 m_2} \, G_{l_1 m_1;l_2 m_2}^{l m} \right)}_{\displaystyle \big( \hat{\psi}_{s,2}^{1,\dd} \big)_{lm}^{j_1,j_2} } y_{lm}
  \end{align*}
  where the $G_{l_1 m_1;l_2 m_2}^{l m}$ are the spherical harmonics product coefficients and the negative sign from the definition of the wavelets is subsumed.
  We thus have
    \begin{align*}
    \hat{\psi}_{s,2}^{1,\dd,3}
    = \frac{2^{3j/2}}{(2\pi)^{3/2} \vert \xi \vert} \sum_{j_1 < j_2} \sum_{l,m} \big( \hat{\psi}_{s,2}^{1,\dd} \big)_{lm}^{j_1,j_2} \, y_{lm} \, \hat{h}(2^{-j} \vert {\xi} \vert) \, \frac{\partial}{\partial \xi^{j_1}} \wedge \frac{\partial}{\partial \xi^{j_2}}  .
  \end{align*}
  The spatial representation of the wavelet is given by the inverse Fourier transform in Def.~\ref{def:diff_form:ft}.
  When we also immediately use the Rayleigh formula to expand the scalar complex exponential in spherical coordinates we obtain
  \begin{align*}
    {\psi}_{s,2}^{1,\dd,3} &= \frac{-2^{3j/2}}{(2\pi)^{3}} \sum_{j_1 < j_2} \sum_{l,m} \big( \hat{\psi}_{s,2}^{1,\dd} \big)_{lm}^{j_1,j_2}  \int_{\R^n} \frac{1}{\vert \xi \vert} \, y_{lm}(\hat{\theta},\hat{\phi}) \, \hat{h}(2^{-j} \vert {\xi} \vert) \, \frac{\partial}{\partial \xi^{j_1}} \wedge \frac{\partial}{\partial \xi^{j_2}}
    \\[4pt]
    & \qquad \qquad \qquad \qquad \times \frac{1}{4\pi} \sum_{l'm'} i^{l'} y_{l'm'}^*(\hat{\theta},\hat{\phi}) \, y_{l'm'}({\theta},{\phi}) \, j_l(\vert \xi \vert \, \vert x \vert) \, \left( \frac{\partial}{\partial \xi^{j_3}} \, \dd x^3 \right)
  \end{align*}
  For each $(j_1,j_2)$ there is, by the anti-symmetry of the wedge product, only one $j_3$ so that one obtains a volume form, and the necessary reordering of the form basis functions introduces a factor $\mathrm{sgn}(\sigma_j)$ where $j = (j_1,j_2,j_3)$.
  In each case we can then change back to spherical coordinates and resolve the integrals over $(\hat{\theta},\hat{\phi})$, where the orthonormality of the spherical harmonics applies, and the radial direction $\vert \xi \vert$, where one obtains the modified spherical Hankel transform $h_l^{(q)}$ in Theorem~\ref{thm:psiforms:frames}.
  This yields the spatial form for the wavelets listed in the theorem.

  That the wavelets are real-valued follows from a careful analysis of the spherical harmonics coefficients $\big( \hat{\psi}_{s,a}^{1,\nu} \big)_{lm}^{j_1,j_2}$ and $\big( \hat{\psi}_{s,a}^{2,\nu} \big)_{lm}^{j_1}$.
  For instance, $(\hat{\phi} \! \cdot \hat{r})_{lm}$ in the above example is nonzero only for even $l$ and the same holds for the $\kappa_{l_1 m_1}^{j,t}$ by the symmetry of the directional localization window $\hat{\gamma}(\hat{\theta},\hat{\phi})$ that is required for a real-valued wavelet in the scalar case.
  The product coefficients then preserve the even parity so that the $i^l$ factor from the Rayleigh formula yields consistently $\pm 1$.
  \commentCL{In contrast, $(\hat{\theta} \! \cdot \hat{\phi})_{lm}$ and $(\hat{\theta} \! \cdot \hat{r})_{lm}$ are nonzero only for odd $l$, so that the $i^l$ is always imaginary, cancelling the imaginary unit in the definition of the basis functions. See \texttt{form\_wavelets\_inverse\_fourier.wls} for details.}
\end{proof}

\begin{proof}[Proof of Theorem~\ref{thm:psiforms:exterior}]
  We have:
  \begin{enumerate}[i)]

    \item The closure under the exterior derivative holds because $\hat{\dd} = \ip_{i \xi}$ only acts on the radial $\partial / \partial \hat{r}$ component.
        For instance,
        \begin{align*}
          \hat{\dd} \hat{\psi}_{s}^{0,\delta,3}
          &= \ip_{i \xi} \left( \frac{1}{\vert \xi \vert} \hat{\psi}_s(\xi) \frac{\partial}{\partial \hat{\theta}} \wedge \frac{\partial}{\partial \hat{\phi}} \wedge \frac{\partial}{\partial \hat{r}} \right)
          \\[5pt]
          &= \frac{i}{\vert \xi \vert} \hat{\psi}_s(\xi) \frac{\partial}{\partial \hat{\theta}} \wedge \frac{\partial}{\partial \hat{\phi}} \wedge \frac{\partial}{\partial \hat{r}} \Big( \vert \xi \vert \dd \hat{r} \Big)
          \\[5pt]
          &= i \hat{\psi}_s(\xi) \, \frac{\partial}{\partial \hat{\theta}} \wedge \frac{\partial}{\partial \hat{\phi}}
        \end{align*}
        where we used that in spherical coordinate $\xi = \vert \xi \vert \dd \hat{r}$ and that the basis vectors satisfy the duality condition $\partial / \partial \hat{r}( \dd \hat{r} ) = 1$.
      The other cases are analogous.
      That $\dd$ is a unitary operator on co-exact differential form basis functions follows immediately from the definition of the homogeneous Sobolev space $\dot{L}_2^1(\Omega_{\delta}^{\fdeg},\R^n)$.

    \item The closure of co-exact differential form wavelets $\psi_{s,a}^{\fdeg,\delta,n}$ under the co-differential holds since these are co-exact differential forms in the sense of the continuous theory.

    \item This can be checked by direct calculations.
      \commentCL{The Hodge dual of the \emph{normalized} basis functions in spherical coordinates are analogous to the Cartesian ones in $\R^3$ since the metric is also the identity in this case. This can also be checked by expanding the spherical form basis functions in Cartesian ones and then evaluating the Hodge dual by linearity (for $2$-form basis functions one has surprising simplifications of the products.)}
%    The first part is just Proposition~\ref{prop:polarlet:2d:parseval}. The second part holds since we are in Euclidean space where the metric is the identity. The Hodge dual is thus also the identity w.r.t the coordinate functions and Theorem~\ref{thm:psiforms:exterior}, iii.) follows then immediately from the usual rules for the Hodge dual.

    \item This follows from the closure of scalar polar wavelets under rigid body transformations~\cite{Azencott2009} and the fact that we use frame functions in spherical coordinates that are compatible with the scalar window.

   \end{enumerate}
\end{proof}

\section{Conclusion}
\label{sec:conclusion}

We introduced $\Psiec$, a wavelet-based discretization of exterior calculus. 
Its central objects are polar differential $\fdeg$-form wavelets $\psi_{s,a}^{\fdeg,\nu,n}$ that provide tight frames for the spaces $\dot{L}_2^{\vert \nu \vert}(\Omega_{\nu}^{\fdeg}, \R^n)$ and that satisfy important properties of Cartan's exterior calculus, such as closure under the exterior derivative.
In contrast to existing discretizations, which typically construct a discrete analog of the de Rahm complex, the $\psi_{s,a}^{k,\nu,n}$ are bona fide differential forms in the sense of the continuous theory.
All operations on differential forms are hence naturally defined.
Finite closure and the computational efficiency of the discretization hence become the questions of interest.
%We showed that Stokes' theorem holds for differential form wavelets and that with a finite number of levels anisotropic curvelet- or ridgelet-like frame functions provide quasi optimal approximations approximations (for $\fdeg =n$).
%Operators that are currently not well described in our calculus are the wedge product and the Lie derivative.

The present work provides, in our opinion, a first step towards a more complete theory.
For example, except when $n=0$ or $\fdeg=n$ and the existing results for cartoon-like functions apply, we did not establish approximation rates for polar differential form wavelets.
%When $k + 1 \neq n$,  then such estimates are necessary for approximation result for Stokes' theorem.
Similarly, for the Laplace--de Rahm operator we are missing decay estimates for its Galerkin projection $D_{qp}$, which are, for example, necessary for error bounds for Example~\ref{ex:laplace_de_rahm:cavity}.
We also did not yet present numerical experiments. 
An implementation of the differential form wavelets is available, cf. Fig.~\ref{fig:psi_delta_22} and Fig.~\ref{fig:psis:3d}, and numerical results for example applications will be presented in a forthcoming publication. 
Since our wavelets have non-compact support in space, the efficiency of numerical calculations depends critically on the decay of the radial window.
It should hence be optimized, cf.~\cite{Ward2015}, or an ansatz should be developed that yields compactly supported wavelets.
% This implies that one has to give up covariance under translation but in many applications compact support is more critical.
Then also a fast transform would be available.

Our work can be extended in various directions. 
We defined our differential form wavelets only for $n=2,3$ but the construction carries over to $n > 3$. 
One then would require more generic proofs then the ones currently used by us that typically proceed explicitly with the individual cases.
However, with the tools introduced in Sec.~\ref{sec:exterior_calculus} this should be possible. 
%, although also, perhaps, less insightful.
It would also be interesting to extend the functional analytic setting of our work.
Troyanov~\cite{Troyanov2009} considers, for example, $L_p(\Omega^{\fdeg},\R^n)$ and spaces with variable regularity, for example in the sense of H{\"o}lder, which would be relevant for applications.

Various extensions of our treatment of Stokes' theorem are also possible. 
We only considered smooth manifolds but using the tools of geometric measure theory~\cite{Federer1996} it should be possible to establish a similar result for the non-smooth case.
Given that we are limited to an extrinsic treatment of manifolds, it would also be interesting to explore if our construction can be combined with those by Berry and Giannakis~\cite{Berry2018} who use the eigenfunctions of the Laplace--Beltrami operator to define a spectral exterior calculus on arbitrary manifolds.
A special case of particular relevance is the sphere $S^2$ where one still has an extensive analytic theory available and where geophysical fluid dynamics provides an important application.
The construction of $\Psi\mathrm{ec}$ on $S^2$ and numerical results on the simulation of the shallow water equation are presented in~\cite{Silva2020}.

%Preliminary investigations indicate that there one could use needlets~\cite{Narcowich2006a} as the scalar wavelets that bootstrap the construction of differential form wavelets.
%A local spectral exterior calculus on $S^2$ would have applications, for example, to climate simulations.
%Presuming Stokes' theorem, Ex.~\ref{eq:stokes:decay:circle:fourier} suggests that the decay conditions of $\hat{\chi}_{\M}$ and $\hat{\chi}_{\partial \M}$ have to be in a fixed relationship not only for the example but more generally.
%It would be interesting to explore this approach to the decay of sets and its Fourier transform further.

%\input{exec}
%\input{preliminaries}

%%%%%%%%%%%%%%%%%%%%%%%%%%%%%%%%%%%%%%%%%%%%%%%%%%%%%%%%%%%%%%%%%%%%%%%%%%%%%%%%%%%%%%%%%%%%%%%%%%%%%%%%%%%%%%%%%%%%%%%%%%%%%%
\section*{Acknowledgements}

I would like to thank Mathieu Desbrun and Tudor Ratiu for helpful discussions.
The work also benefited from discussions at the Workshop on Mathematics of Planet Earth organized by Darryl Holm in June 2019.
The insightful suggestions by the anonymous reviewer are also gratefully acknowledged.
Continuing support by Eugene Fiume over many years also helped to make this work possible.

%%%%%%%%%%%%%%%%%%%%%%%%%%%%%%%%%%%%%%%%%%%%%%%%%%%%%%%%%%%%%%%%%%%%%%%%%%%%%%%%%%%%%%%%%%%%%%%%%%%%%%%%%%%%%%%%%%%%%%%%%%%%%%
%\section*{References}
\bibliography{psiec}

\appendix
\section{Conventions}
\label{sec:preliminaries}

%%%%%%%%%%%%%%%%%%%%%%%%%%%%%%%%%%%%%%%%%%%%%%%%%%%%%%%%%%%%%%%%%%%%%%%%%%%%%%%%%%%%%%%%%%%%%%%%%%%%%%%%%%%%%%%%%
\subsection{Notation}
\label{sec:preliminaries:notation}

We denote the norm in $\R^n$ by $\vert \cdot \vert$ and those in a function space by $\Vert \cdot \Vert$. 
Unit vectors are written as $\bar{x}$, i.e. for $x \in \R^n$ we have $\bar{x} = x / \vert x \vert$.
We will not always distinguish between a unit vector and its corresponding coordinates in polar or spherical coordinates and depending on the context $\bar{x}$ might thus be a geometric vector or its spherical coordinates $(\theta,\phi)$.

We use standard spherical coordinates with latitude $\theta \in [0,\pi]$ and longitude $\phi \in 2 \pi$. 
The standard area form in polar coordinates is $r \, d\theta \, dr$ and in spherical coordinates the volume form is $r^2 \sin{\theta} \, d\theta \, d\phi \, dr$.
For example, the integral of $f : \R^3 \to \R$ takes in spherical coordinates thus the form
\begin{align*}
  \int_{\R^3} f(x) \, dx 
  &= \int_{\R^+} \int_{S^2} f(r,\theta,\phi) \, r^2 \, \sin{\theta} \, d\theta \, d\phi \, dr
  \\[4pt]
  &= \int_{\R^+} \int_{\theta=0}^{\pi} \int_{\phi=0}^{2\pi} f(r,\theta,\phi) \, r^2 \, \sin{\theta} \, d\theta \, d\phi \, dr
\end{align*}
where $\R^+$ means that the radial variable $\vert x \vert$ is integrated over the positive real line. 
We will use similar notation whenever confusion might arise otherwise. 

%%%%%%%%%%%%%%%%%%%%%%%%%%%%%%%%%%%%%%%%%%%%%%%%%%%%%%%%%%%%%%%%%%%%%%%%%%%%%%%%%%%%%%%%%%%%%%%%%%%%%%%%%%%%%%%%%
\subsection{The Fourier transform and Homogeneous Sobolev Spaces}
\label{sec:preliminaries:fourier}
\label{sec:preliminaries:homoegeneous_sobolev}

The scalar unitary Fourier transform of a function $f \in L_2(\R^n) \cap L_1(\R^n)$ is 
\begin{align*}
  \mathfrak{F}(f)(\xi) = \hat{f}(\xi) = \frac{1}{(2\pi)^{n/2}} \int_{\R^n} f(x) \, e^{-i \langle x , \xi \rangle} \, d x
\end{align*}
with inverse transform
\begin{align*}
  \mathfrak{F}^{-1}(f)(\xi) = f(x) = \frac{1}{(2\pi)^{n/2}} \int_{\widehat{\R}^n} \hat{f}(\xi) \, e^{i \langle \xi , x \rangle} \, d\xi .
\end{align*}

%%%%%%%%%%%%%%%%%%%%%%%%%%%%%%%%%%%%%%%%%%%%%%%%%%%%%%%%%%%%%%%%%%%%%%%%%%%%%%%%%%%%%%%%%%%%%%%%%%%%%%%%%%%%%%%%%
%\subsection{The Homogeneous Sobolev Space $\dot{L}_2^1$}
%\label{sec:preliminaries:homoegeneous_sobolev}

%Homogeneous Sobolev spaces $\dot{L}_p^k(\R^n)$ are used to control regularity of solutions of partial differential equations on unbounded domains, e.g.~\cite[Ch. 1]{Bahouri2011} or ~\cite[Ch. II.6]{Galdi2011}.\footnote{In the literature often the notation $\dot{H}^s(\R^n)$ is used but for our purposes $\dot{L}_2^1$ will help to clarify the construction.}
The homogeneous Sobolev space $\dot{L}_2^1(\R^n)$ is defined as
\begin{align*}
  \dot{L}_2^1(\R^n) = \left\{ f \in L_1(\R^n) \ \Big\vert \ \big\Vert [ f ]_{0} \big\Vert_{\dot{L}_2^1} < \infty \right\} 
\end{align*}
with the norm being those induced by the inner product~\cite[Ch. 1.2.1]{Bahouri2011}
% the definition uses 2s and hence one has |\xi| for each of the functions
\begin{align*}
  \left\langle f , g \right\rangle_{\dot{L}_2^1} = \int_{\widehat{\R}^n} \hat{f}(\xi) \, \hat{g}(\xi) \,  \vert \xi \vert^{2} \, d \xi  
\end{align*}
and $[ f ]_{0}$ denoting the co-sets of functions modulo constant polynomials~\cite[Ch. II.6]{Galdi2011}.
As is customary, the co-sets are usually omitted in the notation.
The so defined space $\dot{L}_2^1(\R^n)$ is a Hilbert space~\cite[Ch. II.6]{Galdi2011}.

%%%%%%%%%%%%%%%%%%%%%%%%%%%%%%%%%%%%%%%%%%
\subsection{Spherical harmonics}
\label{sec:preliminaries:sh}

The analogue of the Fourier transform on the sphere is the spherical harmonics expansion.
For any $f \in L_2(S^2)$ it is given by
\begin{align*}
  f(\omega)
  = \sum_{l=0}^{\infty} \sum_{m=-l}^l \langle f(\eta) , y_{lm}(\eta) \rangle \, y_{lm}(\omega)
  = \sum_{l=0}^{\infty} \sum_{m=-l}^l f_{lm} \, y_{lm}(\omega)
\end{align*}
where $\langle \cdot , \cdot \rangle$ denotes the standard (complex) $L_2$ inner product on $S^2$ given by
\begin{align*}
  \left\langle f(\omega) , g(\omega) \right\rangle
  = \int_{\theta=0}^{\pi} \int_{\phi=0}^{2 \pi} f(\theta,\phi) \, g^*(\theta,\phi) \, \sin{\theta} \, d\theta d\phi .
\end{align*}
%We use standard (geographic) spherical coordinates with $\theta \in [0,\pi]$ being the polar angle and $\phi \in [0,2\pi]$ the azimuthal one.
The spherical harmonics basis functions are given by
\begin{align*}
  y_{lm}(\omega) = y_{lm}(\theta,\phi) = C_{lm} \, P_l^m( \cos{\theta} ) \, e^{i  m  \phi}
\end{align*}
where the $P_l^m(\cdot)$ are associated Legendre polynomials and $C_{lm}$ is a normalization constant so that the $y_{lm}(\omega)$ are orthonormal over the sphere.
%The associated Legendre polynomials are defined as\footnote{The Condon-Shortely phase factor of $(-1)^m$ is included to be consistent with Mathematica.}
%\begin{subequations}
%  \label{eq:associated_legendre}
%\begin{align}
%  \label{eq:associated_legendre:1}
%  P_l^m(\cos{\theta}) = (-1)^m \sum_{p=0}^r c_{lmp} \, \sin{\theta}^{m}  ( \cos{\theta} )^{l-m-2p}
%\end{align}
%where $r = \lfloor (l - m) / 2 \rfloor$ and
%\begin{align}
%  \label{eq:associated_legendre:2}
%  c_{lmp} = (-1)^p \frac{2^{-l} (2l - 2p)!}{p! (l-p)! (l-m-2p)!} .
%\end{align}
%The associated Legendre polynomials are not $L_2$-normalized by satisfy 
%\begin{align}
%  \int_{-1}^1 P_{l_1}^{m_1}(x) \, P_{l_2}^{m_2}(x) \, dx = \frac{2}{2l+1} \frac{(l+m)!}{(l-m)!} \delta_{l_1 l_2} .
%\end{align}
%\end{subequations}

The spherical harmonics addition theorem is 
\begin{align}
  \label{eq:sh:addition_thm}
  P_l(\bar{x}_1 \cdot \bar{x}_1) = \frac{4\pi}{2l+1} \sum_{m=-l}^l y_{lm}(\omega_1) \, y_{lm}^*(\omega_2) 
\end{align}
where $\omega_1$ are the spherical coordinates for the unit vector $\bar{x}_1 \in \mathbb{R}^3$ and $\omega_2$ are those for $\bar{x}_2 \in \mathbb{R}^3$.
It follows immediately from the above formula that $P_l(\bar{x}_1 \cdot \bar{x}_1)$ is the reproducing kernel for the space $\mathcal{H}_l$ spanned by all spherical harmonics of fixed $l$.
The $\mathcal{H}_l$ are closed under rotation, i.e. $f \in \mathcal{H}_l \Rightarrow R f \in \mathcal{H}_l$ for $R \in \mathrm{SO}(3)$ and the action is given by $(R f)(\omega) = f(R^T \omega)$. 
In the spherical harmonics domain the rotation is represented by Wigner-D matrices $W_{lm}^{m'}(R)$ that act as
\begin{align*}
  f_{lm}(R) = D_{lm}^{m'} \, f_{lm'}
\end{align*}
where the $f_{lm}(R)$ are the spherical harmonics coefficients of the rotated signal.
When $R$ describes a rotation of the North pole to $\lambda \in S^2$, then $D_{lm}^{0}$ takes a simple form,
\begin{align*}
  D_{lm}^{0}(\lambda) = \sqrt{\frac{4\pi}{2l+1}} \, y_{lm}^*(\lambda) .
\end{align*}
%Comparing to Eq.~\ref{eq:wigner_d:euler} and recognizing that the Euler angles for the rotation from the North pole are $(\alpha,\beta,\gamma) = (0,\theta,\phi)$ we see that the Wigner small-D matrices in fact coincide with the Legendre polynomials in this case.

In contrast to the Fourier series, where the product of two basis functions $e^{i m_1 \theta}$ and $e^{i m_2 \theta}$ is immediately given by another Fourier series function, $e^{i (m_1 + m_2) \theta}$, for spherical harmonics the product is not diagonal and instead characterized by Clebsch-Gordon coefficients $C_{l_1,m_1; l_2,m_2}^{l,m}$.
In particular, the product of two functions on the sphere is in the spherical harmonics domain given by
\begin{align*}
  \big( f \cdot g \big)(\omega) = \sum_{lm} \underbrace{\Bigg( \sum_{l_1 m_1} \sum_{l_2 m_2} f_{l_1 m_1} g_{l_2 m_2} G_{l_1,m_1; l_2,m_2}^{l,m} \Bigg)}_{\displaystyle (f \cdot g)_{lm}} \, y_{lm}(\omega)
\end{align*}
where the spherical harmonics product coefficients $G_{l_1,m_1; l_2,m_2}^{l,m}$ are  
\begin{align*}
  G_{l_1,m_1; l_2,m_2}^{l,m} 
  &\equiv 
  \int_{S^2} y_{l_1 m_1}(\omega) \, y_{l_2 m_2}(\omega) \, y_{l m}^*(\omega) \, d\omega
  \\[5pt]
  &=\sqrt{\frac{(2l_1 + 1) (2l_2 + 1)}{4\pi (2l+1)}} C_{l_1,0; l_2,0}^{l,0} \, C_{l_1,m_1; l_2,m_2}^{l,m}  .
\end{align*}
%and the $C_{l_1,m_1; l_2,m_2}^{l,m}$ are Clebsch-Gordon coefficients. 
%These are the projection of the product of $y_{l_1,m_1}(\omega)$ and $y_{l_2,m_2}(\omega)$ onto the spherical harmonics $y_{lm}(\omega)$, 
%\begin{align*}
%\int_{S^2} y_{l_1 m_1}(\omega) \, y_{l_2 m_2}(\omega) \, y_{l m}^*(\omega) \, d\omega = 
%  \sqrt{\frac{(2l_1 + 1) (2l_2 + 1)}{4\pi (2l+1)}} C_{l_1,0; l_2,0}^{l,0} \, C_{l_1,m_1; l_2,m_2}^{l,m} .
%\end{align*}
The Clebsch-Gordon coefficients $C_{l_1,m_1; l_2,m_2}^{l,m}$ are sparse and nonzero only when
\begin{align*}
  m = m_1 + m_2 ,
\end{align*}
that is the $m$ parameter is superfluous but conventionally used,  
and 
\begin{align*}
  l_1 + l_2 - l &\geq 0
  \\
  l_1 - l_2 + l &\geq 0
  \\
  -l_1 + l_2 + l &\geq 0 .
\end{align*}

\subsection{Fourier Transform in Polar and Spherical Coordinates}
\label{sec:preliminaries:fourier:polar}

In the plane, the Fourier transform can also be written in polar coordinates by using the Jacobi-Anger formula~\cite{Watson1922},
\begin{align*}
  e^{i \langle \xi , x \rangle} = \sum_{m \in \Z} i^m \, e^{i m (\theta - \hat{\theta})} J_m( \vert \xi \vert \, \vert x \vert ) ,
\end{align*}
that relates the complex exponential in Euclidean and polar coordinates.
%The ordering of the $\phi_x$ and $\phi_{\xi}$ on the right hand side is arbitrary and when the left hand side is conjugated $i^m$ becomes $i^{-m}$.\footnote{This follows since the sum over $m$ runs over all integers, cancelling the imaginary part.}
In the above equation, $J_m( z )$ is the Bessel function of the first kind and $(\theta,\vert x \vert)$ and $(\hat{\theta},\vert \xi \vert)$ are polar coordinates for the spatial and frequency domains, respectively.
The analogue of the Jacobi-Anger formula in three dimensions is the Rayleigh formula,
\begin{align*}
  e^{i \langle \xi , x \rangle} 
  &= \sum_{l=0}^{\infty} i^l \, (2l+1) \, P_l\big(\bar{\xi} \cdot \bar{x}\big) \, j_l( \vert \xi \vert \, \vert x \vert )
  \\[5pt]
  &= 4 \pi \sum_{l=0}^{\infty} \sum_{m=-l}^l i^l \, y_{lm}\big(\bar{\xi}\big) \, y_{lm}\big(\bar{x}\big) \, j_l( \vert \xi \vert \, \vert x \vert )
\end{align*}
where $j_l(\cdot)$ is the spherical Bessel function and the second line follows by the spherical harmonics addition theorem in Eq.~\ref{eq:sh:addition_thm}.

\section{Admissibility of Polar Wavelets for $L_2(\R^3)$}
\label{sec:polarlets:admissibility:3d}

\begin{proof}[Proof of Proposition~\ref{prop:tight_frame:3d}]
Our proof will be for a scale and directionally discrete frame that is continuous in the spatial domain, i.e. where all translations in $\R^3$ are considered.
  Since our window functions are bandlimited, the result carries over to the discrete case using standard arguments based on the Shannon sampling theorem, see for example~\cite[Sec. 4]{Candes2005b} or~\cite[Sec. IV]{Unser2011}. 
  
We want to show
\begin{align*}
  u
  &= \sum_{j=-1}^{\infty} \sum_{t=1}^{T_j} \big( u \star \psi_j \big) \star \psi_j
\end{align*}
for $u \in L_2(\R^3)$.
Taking the Fourier transform of both sides we have
% TODO: the convolution is not unitary with our normalization
\begin{align*}
   \hat{u}
   &= \sum_{j=-1}^{\infty} \sum_{t=1}^{T_j} ( \hat{u} \, \hat{\psi}_j) \, \hat{\psi}_j 
   = \hat{u} \, \sum_{j=-1}^{\infty} \sum_{t=1}^{T_j} \hat{\psi}_j^* \, \hat{\psi}_j
\end{align*}
and it thus suffices to show that the scalar window functions satisfy the Calder{\'o}n admissibility condition
\begin{align*}
  \sum_{j=-1}^{\infty} \sum_{t=1}^{T_j} \vert \hat{\psi}_j \vert^2 = 1  \quad , \quad \forall \xi \in \R^3 .
\end{align*}
With the definition of the window functions in Eq.~\ref{eq:polarlets:3d:hat} one obtains
\begin{align*}
  \sum_{j=-1}^{\infty} \sum_{t = 1}^{T_j} \vert \hat{\psi}_j \vert^2
  &= \sum_{j} \sum_{t = 0}^{T_j} \sum_{l_1,m_1} \sum_{l_2 m_2} \kappa_{l_1 m_1}^{j,t} y_{l_1 m_1}\big(\bar{\xi}\big) \, \kappa_{l_2 m_2}^{j,t} y_{l_2 m_2}^*\big(\bar{\xi}\big) \big\vert \hat{h}(2^{-j} \vert \xi \vert) \big\vert^2 .
\end{align*}
Assuming the radial windows satisfy the Calder{\'o}n condition, i.e. 
\begin{align*}
  \sum_{j=-1}^{\infty} \big\vert \hat{h}( 2^{-j} \vert \xi \vert ) \big\vert^2 = 1
\end{align*}
the proposition holds when the product of the angular part evaluates to the identity for every band $j$ and every direction $\bar{\xi}$.
Since $y_{00}$ is the constant function, this means that for every $j$ the projection of the angular part in the above equation onto spherical harmonics has to satisfy
\begin{align*}
 \delta_{l,0} \, \delta_{m,0}
 &= \sum_{t = 0}^{T_j} \left\langle \sum_{l_1,m_1} \kappa_{l_1 m_1}^{j,t} y_{l_1 m_1}\big(\bar{\xi}\big) \, , \, \sum_{l_2,m_2} \kappa_{l_2 m_2}^{j,t} y_{l_2 m_2}^*\big(\bar{\xi}\big) \, , \, y_{lm}(\bar{\xi}) \right\rangle .
 \intertext{Rearranging terms we obtain}
 \delta_{l,0} \, \delta_{m,0}
 &= \sum_{t = 0}^{T_j} \sum_{l_1,m_1} \sum_{l_2,m_2}  \kappa_{l_1 m_1}^{j,t} \, \kappa_{l_2 m_2}^{j,t*} \, \big\langle y_{l_1 m_1}\big(\bar{\xi}\big)  y_{l_2 m_2}^*\big(\bar{\xi}\big) \, , \, y_{lm}^*(\bar{\xi}) \big\rangle .
\end{align*}
The product of two spherical harmonics projected into another spherical harmonic is given by the product coefficients $G_{l_1,m_1;l_2,m_2}^{l,m}$ introduced in ~\ref{sec:preliminaries:sh}. We can hence write
\begin{align*}
 \delta_{l,0} \, \delta_{m,0}
 &= \sum_{t = 0}^{T_j} \sum_{l_1,m_1} \sum_{l_2,m_2} \kappa_{l_1 m_1}^{j,t} \, \kappa_{l_2 m_2}^{j,t} \, G_{l_1,m_1;l_2,-m_2}^{l,m_1-m_2} .
 \end{align*}
Collecting the  $\kappa_{l_i m_i}^{j,t}$ into vectors we obtain the condition in the proposition.
%The remaining sum over $j$ yields the identity by construction of the $\hat{h}$ windows.
\end{proof}

\end{document}